\documentclass[11pt,oneside,a4paper,mathscr]{amsart}

\usepackage{amsthm}
\usepackage{amsmath}
\usepackage{amssymb}
\usepackage{amsfonts}
\usepackage{amscd}
\usepackage{mathrsfs} 
\usepackage{mathtools}
\usepackage[dvipdfmx]{graphicx}
\usepackage{bm}   
\usepackage[all]{xy}
\usepackage{ulem}
\makeatletter

\@addtoreset{equation}{section}
\makeatother   

\newtheorem{theorem}{Theorem}[section]
\newtheorem{lemma}[theorem]{Lemma}
\newtheorem{proposition}[theorem]{Proposition}
\newtheorem{corollary}[theorem]{Corollary}

\theoremstyle{definition}
\newtheorem{definition}[theorem]{Definition}
\newtheorem{example}[theorem]{Example}
\newtheorem{remark}[theorem]{Remark} 
\def\C{\mathbb{C}} 
\def\R{\mathbb{R}}
\def\Q{\mathbb{Q}}
\def\P{\mathbb{P}}
\def\Z{\mathbb{Z}}

\def\i{{\sqrt{-1}}}
\def\del{\partial}
\def\id{\mathrm{id}}
\def\pr{\mathrm{pr}}

\newcommand{\Hom}{\mathrm{Hom}}

\def\hom{{\mathscr{H}\! \! om}}
\def\Cok{{\mathrm{Cok}}}

\def\({\left(}
\def\){\right)}
\def\<{\langle}
\def\>{\rangle}
\newcommand{\simeqto}{\xrightarrow{\sim}}

\newcommand{\cal}[1]{\mathcal{#1}}
\newcommand{\scr}[1]{\mathscr{#1}}

\def\rd{{\rm rd}}
\def\mg{{\rm mod}}

\def\tw{\lambda}
\def\eq{\mu}

\def\X{{\mathcal{X}}}
\renewcommand{\L}{{\mathscr{L}}}
\def\A{{\mathscr{A}}}
\def\O{{\mathscr{O}}}

\newcommand{\St}{{\mathrm{St}}}
\newcommand{\bl}{{\mathrm{Bl}}}
\newcommand{\gr}{{\mathrm{gr}}}
\def\Gr{{\mathrm{Gr}}}

\renewcommand{\k}{{\bm{k}}}
\newcommand{\et}{{\acute{e}t}}

\newcommand{\Crit}{{\mathrm{Crit}}}

\def\shift{{\mathbb{S}}}
\def\euler{{\mathfrak{a}}}

\def\DR{{\mathrm{DR}}}
\def\M{{\mathscr{M}}}
\def\N{{\scr{N}}}
\def\H{{\mathscr{H}}}
\def\j{\iota}
\def\T{{{T}}}
\def\IQ{{\mathscr{Q}}}
\def\K{{\mathscr{K}}}
\def\F{{\mathscr{F}}}

\def\dR{{\rm dR}}
\def\Be{{\rm Be}}
\def\Per{{\rm Per}}

\def\nb{{\mathrm{nb}}}
\def\Re{{\mathrm{Re}}}
\def\Im{{\mathrm{Im}}}
\def\e{{\bm{e}}}
\def\q{{q}}
\def\s{{\eq}}
\def\Db{{\mathfrak{Db}}}
\def\RH{{\mathrm{RH}}}

\begin{document}
\title{Stokes filtered sheaves and 
differential-difference modules}
\author{Yota Shamoto}
\address{Kavli Institute for the Physics and Mathematics of the Universe (WPI), The University of Tokyo Institutes for Advanced Study, The University of Tokyo, Kashiwa, Chiba 277-8583, Japan}
\email{yota.shamoto@ipmu.jp}

\begin{abstract}
We introduce the notion of 
Stokes filtered quasi-local systems.
It is proved that the 
category of Stokes filtered quasi-local systems is an Abelian category.
We also give a geometric way to
construct Stokes filtered quasi-local systems,
which describe 
the asymptotic behavior of
certain classes of solutions to 
some differential-difference modules. 
\end{abstract}
\maketitle

\section{Introduction}

In his letter \cite{Deligne} to B. Malgrange, P. Deligne
introduced the 
notion of a sheaf
with a filtration indexed by 
a local system of ordered sets 
in order to express the Stokes phenomenon of 
a linear differential equation of one complex variable
in a sheaf-theoretic way. 
The notion has been developed and extended to 
arbitrary dimensions (See \cite{Sabbah} and references therein)
and is now called a \textit{Stokes filtered sheaf}.
Here, the term ^^ sheaf' is often replaced by a more precise
term like local system, perverse sheaf, and so on.

In this paper, 
we 
introduce an analogous 
notion of a Stokes filtered sheaf
to express the 
Stokes phenomenon of a
differential-\textit{difference}
module of two complex variables 
in a sheaf-theoretic way.
Although 
we only treat a
special class of
differential-difference modules, 
we expect that 
this approach gives a clue 
to investigate more general difference equations 
in a sheaf theoretic way.


To clarify the analogy, in \S \ref{Intro recall}
we will briefly recall some parts of 
the theory of Stokes filtered local systems 
for differential equations.
We also give a class of examples
of the Stokes filtered local systems constructed in a geometric way
since
we will mainly consider the analogue of such examples. 
We then
explain our notion of Stokes filtered ^^ ^^ quasi-local systems"
for differential-difference modules 
in \S \ref{Intro q-Stokes}
and the main results of this paper
in \S \ref{Intro main}. 
Further direction related to mirror symmetry and 
Dubrovin's conjecture
will be discussed in \S \ref{Intro mirror}. 

\subsection{Stokes filtered local systems for differential equations}\label{Intro recall}
Let us briefly recall the theory of Stokes filtered local systems
on $S^1$ following \cite{Sabbah}. 
Let $\scr{I}_1$
denote the constant
sheaf on $S^1$
with fiber $\tw^{-1}\C[\tw^{-1}]$
equipped with the order depends on the point 
$e^{\i\theta}\in S^1$ as follows:
For $\varphi,\psi\in \tw^{-1}\C[\tw^{-1}]$
and $e^{\i\theta}\in S^1$,
$
\varphi\leqslant_{\theta}\psi$ (resp. $\varphi<_\theta\psi$)
if and only if 
$\exp(\varphi(\tw)-\psi(\tw))$
is of moderate growth (resp. rapid decay)
when 
$\lambda$ tends to zero 
satisfying $\arg(\tw)=\theta$.

Let $\cal{L}$ be a local system on a circle $S^1$. 
A non-ramified pre-Stokes filtration 
$\cal{L}_{\leqslant}$
is a family of subsheaves $\cal{L}_{\leqslant\varphi}\subset \cal{L}$
on $\cal{L}$
indexed by $\varphi\in \tw^{-1}\C[\tw^{-1}]$
with the following condition:
For any $e^{\i\theta}\in S^1$, 
$\varphi\leqslant_\theta\psi$
implies 
$\cal{L}_{\leqslant\varphi,\theta}\subset \cal{L}_{\leqslant\psi,\theta}$. 
We may naturally define the 
non-ramified pre-Stokes filtration $\gr\cal{L}_\leqslant$ on 
$\gr\cal{L}\coloneqq \bigoplus_\varphi\gr_\varphi\cal{L}$, 
$\gr_\varphi\cal{L}\coloneqq \cal{L}_{\leqslant \varphi}/\cal{L}_{<\varphi}$.
Then a non-ramified pre-Stokes filtration 
is called a 
non-ramified Stokes filtration 
if $(\cal{L},\cal{L}_\leqslant)$ is locally isomorphic to 
$(\gr\cal{L},\gr\cal{L}_\leqslant)$. 
A morphism of non-ramified Stokes filtered local is defined in an obvious way. 
The following theorem is fundamental
(cf. \cite[Theorem 3.1, Theorem 3.5]{Sabbah}):
\begin{theorem}\label{diff. abel}
The category of 
non-ramified Stokes filtered local systems is abelian.  
\end{theorem}
\begin{remark}
This theorem holds for more general Stokes filtered local systems explained below. 
More generally, the notion of Stokes filtered sheaves 
can be defined in higher dimensions. In higher dimensional case, 
the notion of ^^ ^^ goodness" plays an important role
to see the abelianity of the category. 
\end{remark}
We may consider 
(ramified) Stokes filtered local systems.
In the definition, the index sheaf $\scr{I}_1$
is replaced by 
$\scr{I}=\bigcup_{d\geqslant 1}\scr{I}_d$, 
where, roughly speaking, 
$\scr{I}_d$ denotes the local system on $S^1$
with the fiber $\tw^{-1/d}\C[\tw^{-1/d}]$
and the monodromy $\exp(-2\pi\i/d)$.  
Let $\scr{I}^\et$ denote the \'etale 
space of $\scr{I}$ and $\tau \colon \scr{I}^\et\to S^1$
denote the projection. 
The (pre-)Stokes filtration $\cal{L}_\leqslant$ is
defined as the subsheaf of $\tau^{-1}\cal{L}$
with some conditions 
(See \cite{Sabbah} for more details).  

\begin{theorem}[{Deligne \cite{Deligne}, Malgrange \cite{Mal}, see also
 \cite[Theorem 5.8]{Sabbah}}]\label{Riemann-Hilbert}
There is a functor 
 \[\cal{H}\mapsto \mathrm{RH}(\cal{H})= \(\mathrm{RH}(\cal{H}),\mathrm{RH}_\leqslant(\cal{H})\)\]
from the category of germs of meromorphic connections on $(\C,0)$
to the 
category of Stokes filtered local systems on $S^1$. \qed
\end{theorem}
\begin{remark}
In this theorem, the functor $\mathrm{RH}$ is called the Riemann-Hilbert functor. 
The local system $\RH(\cal{H})$ on $S^1$ is denoted by $\H^0(\widetilde{\DR}(\cal{H}))$ 
and the Stokes filtration $\RH_{\leqslant}(\cal{H})$ on it 
is denoted by $\H^0(\DR_\leqslant(\cal{H}))$ in \cite{Sabbah}.
\end{remark}

We can construct interesting examples
of Stokes filtered local system 
and the corresponding differential equation in a geometric way. 
For simplicity, we restrict ourselves to one dimensional case. 
Let $X$ be a compact Riemann surface. 
Let $f\colon X\to \P^1$ be a 
meromorphic function on it. 
Put $P=f^{-1}(\infty)$. 
Assume that
$f$ has only $A_1$-singularities on $U\coloneqq X\setminus P$. 
Put $\mathfrak{X}\coloneqq \C^*_\tw\times X$
and $\mathfrak{P}\coloneqq \C^*_\tw\times P$. 

Consider the meromorphic connection
$\cal{M}(f)=(\O_{\mathfrak{X}}(*\mathfrak{P}),d+d(\tw^{-1}f))$ on $\mathfrak{X}$. 
Then, we obtain a 
meromorphic connection on $\C^*_\tw$
by taking the pushing forward
\[\cal{H}_\dR^1(f)\coloneqq
\Cok\left[\pi_{\mathfrak{X*}}\O_{\mathfrak{X}}(*\mathfrak{P})\xrightarrow{d_{\mathfrak{X}/\C^*}+\tw^{-1}df}
\pi_{\mathfrak{X}*}\Omega_{\mathfrak{X}/\C^*}^1(*\mathfrak{P})\right] \]
where $\pi_{\mathfrak{X}}\colon \mathfrak{X}\to \C^*_\tw$ denotes the projection
and $d_{\mathfrak{X}/\C^*}\colon \O_{\mathfrak{X}}\to\Omega_{\mathfrak{X}/\C^*}^1$
denotes the relative differential. 
It is easy to see that 
$\cal{H}^1_\dR(f)$ has singularities only on $\{0,\infty\}$. 
$\cal{H}_\dR^1(f)$ can be seen as 
a filtered de Rham cohomology studied in the theory of 
primitive forms by Kyoji Saito (see \cite{ST} and references therein). 

Applying the Riemann-Hilbert functor in Theorem \ref{Riemann-Hilbert}
to the germ of $\cal{H}_\dR^1(f)$, 
we obtain the Stokes filtered local system $\RH(\cal{H}_\dR^1(f))$.
It is well known that
$\RH(\cal{H}_\dR^1(f))$
is non-ramified and moreover of exponential type 
(c.f. \cite{Pham} for the case where $U=\C^n$, and \cite{Sabbah-Saito} for general case).
In other words, 
the graded part
$\gr_\varphi\RH(\cal{H}_\dR^1(f))$
is non-zero if and only if 
$\varphi=-c/\tw$
for a critical value $c$ of $f$.

There is a geometric description of $\mathrm{RH}(\cal{H}_\dR^1(f))$. 
For each $\tw\in \C^*$, 
let 
$H^\rd_1(U, f/\tw)$
denote the rapid decay homology of 
Bloch-Esnault-Hien (\cite{BE}, \cite{Hien})
of a meromorphic connection 
$(\O_X(*P),d+\tw^{-1}df)$. 
As proved in \cite{HR} (in a more general setting),
$\cal{H}_1^\rd(f)\coloneqq\bigcup_{\tw\in \C^*}H^\rd_1(U, f/\tw)$
is a local system, and the period integral
\begin{align*}
\Per\colon\cal{H}_1^\rd(f)\longrightarrow \hom_{\scr{D}_{\C^*}}(\cal{H}_\dR^1(f)^\vee,\O_{\C^*}),
\quad
[\gamma\otimes e^{-f/\tw}]\mapsto (\omega\mapsto \int_{\gamma}e^{-f/\tw}\omega)
\end{align*}
gives an isomorphism. 
Here, we have put $\cal{H}_\dR^1(f)^\vee\coloneqq \hom_\O(\cal{H}_\dR^1(f),\O_{\C^*})$
and hence $\Per$ induces an inclusion 
$\cal{H}^\rd_1(f)\hookrightarrow \cal{H}^1_\dR(f)$. 
Take the real blowup $\bl^\R_0(\C)$
of $\C$ at the origin
and consider the inclusions
$S^1\xrightarrow{i_{S^1}}\bl^\R_0(\C)\xleftarrow{j_{\C^*}}\C^*.$
Then \[\cal{L}^\rd\coloneqq i_{S^1}^{-1}j_{\C^**}\cal{H}^\rd_1(f)\]
is a local system of $S^1$. 
Take any meromorphic basis $e_1,\dots, e_r$ of $\cal{H}^1_{\dR}(f)$. 
Then we can define the 
filtration on $\cal{L}^\rd(f)$
as follows: 
For a section $\varepsilon\in\cal{L}^\rd(f)$, 
take representative $\widetilde{\varepsilon}\in \cal{H}^\rd_1(f)$. 
Then there is an expression
\begin{align*}
\widetilde{\varepsilon}=\sum_{i=1}^r h_i(\tw)e_i
\end{align*}
where $h_i(\tw)$ denotes holomorphic function on a sector in $\C^*$. 
Then $\varepsilon \in \cal{L}^\rd_{\leqslant \varphi}(f)$
if and only if $e^{-\varphi}h_i(\tw)$ is of moderate growth when 
$\tw$ tends to zero on the sector for all $i=1,\dots, r$. 
Then we can prove that 
$(\cal{L}^\rd(f),\cal{L}^\rd_\leqslant(f))\simeq (\RH(\cal{H}_\dR^1(f)),\RH_\leqslant (\cal{H}_\dR^1(f)))$. 

Let $\omega_1,\dots,\omega_r$ denote the section of 
$\pi_{\mathfrak{X}*}\Omega^1_{\mathfrak{X}/\C^*}(*\mathfrak{P})$
which represents dual bases $e_1^\vee,\dots, e_r^\vee$ of $e_1,\dots,e_r$ in  $\cal{H}^1_\dR(f)^\vee$. 
Assume that $\varepsilon=\gamma\otimes e^{-f/\tw}$ for some family of paths $\gamma$. 
Then we have 
\[h_i(\tw)=\int_{\gamma}e^{-f/\tw}\omega_i\]
By the saddle point method, 
we can directly obtain the following theorem 
without using the Riemann-Hilbert correspondence (Theorem \ref{Riemann-Hilbert}).  
\begin{theorem}\label{L(f)}
The pair $(\cal{L}^\rd(f), \cal{L}^\rd_\leqslant(f))$
is a non-ramified Stokes filtered local system on $S^1$
such that $\gr_\varphi\cal{L}^\rd(f)\neq 0$
iff $\varphi=-c/\tw$ for a critical value $c$ of $f$. \qed
\end{theorem}
\subsection{Stokes filtered quasi-local systems}\label{Intro q-Stokes}
The purpose of 
this paper is to introduce 
an analogous notion of a
Stokes filtered sheaf
to express the asymptotic behavior
of a differential-difference module of two variables.
Before explaining the relation to differential-difference modules, 
we would like to explain the idea of the definition of
such Stokes filtered sheaves. 

Let $T=(S^1)^2\simeq \{(\theta_u,\theta_v)\in (\R/2\pi\Z)^2\}$ be the torus
considered as the corner of the real blowing up 
\[\varpi_{B}\colon \widetilde{B}=\bl^\R_Z(B)\longrightarrow B=\C^2
\]
along the divisor $Z=\{(u, v)\in \C\mid uv=0\}$.
Put $B^*\coloneqq B\setminus Z$ and 
let $T\xrightarrow{\imath_T} \widetilde{B}\xleftarrow{\jmath_B}B^*$ denote the inclusions.
As the counterpart of $\scr{I}_1$ (or $\tw^{-1}\C_{S^1}\subset \scr{I}_1$), 
we consider the sheaf $\IQ$ of index 
defined as 
the restrictions of the sub-sheaf of $j_{B*}\O_{B^*}$ 
generated by the sections of the form
\begin{align*}
u^{-1}\(n\log v+\frac{h(v)}{v}\)\quad\quad(n\in\Z, h(v)\in \O_{\C})
\end{align*}
to $T$ (see Definitions \ref{tilde Q} and \ref{IQ definition} for more precise). 
The sheaf $\IQ$ admits a sheaf of order $\leqslant$ on $T$ (see Definition \ref{IQ order}). 

Let $\k$ be a field. 
As a counterpart of $\k$-local system on $S^1$, 
we consider a quasi-local system of $\k[q^{\pm1}]$-modules on $T$. 
Here, by a quasi-local system
of $\k[q^{\pm 1}]$-module  on $T$,
we mean an $\R$-constructible sheaf 
$\L$ of finite rank free $\k[q^{\pm 1}]$-modules on $T$
such that 
$\L(q)=\L\otimes_{k[q^{\pm 1}]}k(q)$ is a local system of $\k(q)$-vector spaces. 
We moreover assume that 
$\L$ is constructible with respect to the stratification 
$\Theta=(T_{\R_+}, T_{\R_-}, T_{+}, T_-)$ of $T$, 
where $T_{\R_\pm}=\{e^{\i\theta_u}\in \R_\pm\}$ and $T_{\pm}=\{\pm\Im e^{\i\theta_u}>0\}$. 

Let $\tau\colon \IQ^\et\to T$ denote the \'etale space of $\IQ$. 
Then, a pre-Stokes filtration is defined as the  
subsheaf $\L_\leqslant$ in 
$\tau^{-1}\L$ of $\k$-modules with the similar properties as in \S \ref{Intro recall}
(see Definition \ref{pre-q}). 
The new property we add here is the
compatibility of the filtration with
the action of $\k[q^{\pm 1}]$. 
More precisely, we impose the condition 
that 
\begin{align*}
q\cdot \L_{\leqslant \varphi}= \L_{\leqslant\varphi+2\pi\i u^{-1}}
\end{align*}
for any local section $\varphi$ of $\IQ$. 
By this compatibility, 
we can induce the \textit{coarse} filtration on $\L$, which
consists of $\k[q^{\pm 1}]$-submodules in $\L$(see \S \ref{2.2}).

Then, 
we can define the 
notion of a (good)
Stokes filtered quasi-local system $(\L,\L_\leqslant)$
by imposing existence of the local isomorphism to its graded part (see Definition \ref{Def q-Stokes} for more precise). 
Because of some technical reasons, 
we also impose some conditions on the graded part with respect to the coarse filtrations.
The following is the main theorem of this part, which is an analogue of Theorem \ref{diff. abel}:
\begin{theorem}[See Theorem \ref{STRICTNESS} for a more precise statement]
The category of good Stokes filtered quasi-local systems on $T$ is an abelian category. 
\end{theorem}

\subsection{Geometric construction}\label{Intro main}
In this paper, we do not try to  formulate an analog of Theorem \ref{Riemann-Hilbert}. 
Instead, we shall give an analog of Theorem \ref{L(f)}.

\subsubsection{De Rham cohomology}\label{intro dR}
As in \S \ref{Intro recall}, 
let $X$ be an compact Riemann surface and 
$f$ be a meromorphic function on it. 
We also consider a meromorphic function $g$ on $X$.
Let $D$ be the union of $P$ and the pole of $g^{-1}dg$. 
Let $S= \C^2$ be the surface with coordinates $(\tw,\eq)$. 
We put $\X=S\times X$ and $\cal{D}=S\times D$.

We then define the module $\M(f,g)=\O_{\X}(*\cal{D})$
with operators $\nabla$, $\nabla_\euler$, and $\shift$. 
Here, $\nabla=d+\tw^{-1}(df-\eq g^{-1}dg)$ denotes the relative connection, 
$\nabla_\euler$ denotes the differential operator corresponding to 
the vector field $\euler=\tw^2\del_\tw+\tw\eq \del_\eq$
defined by $\nabla_\euler(1)=-f$, 
and  the difference operator $\shift$ corresponding to the 
shift of the parameter $\sigma\colon (\tw,\eq)\mapsto(\tw,\eq-\tw)$
determined by $\shift(1)=g$ (see Definition \ref{MFG} for more precise). 
Since $(\nabla,\nabla_\euler,\shift)$ satisfies a kind of integrability condition (Lemma \ref{INT}), 
the de Rham cohomology group (or, a pushing forward)
$\H^1_\dR(f,g)_{|S^\circ}$ of $\M(f, g)$ 
restricted to $S^\circ=\{(\tw,\eq)\in S^\circ\mid \tw\neq 0\}$ 
is naturally equipped with the operators $\nabla_\euler$ and $\shift$
(see \S \ref{DR2} for the definition). 

Roughly speaking, the differential-difference module $\H^1_\dR(f,g)_{|S^\circ}$ is a counterpart of $\cal{H}^1_\dR(f)$.
However, there appears some difficulties when 
$E=D\setminus P$ is not empty (as we will see below, this case contains an interesting example). 
In this case, 
$\H^1_\dR(f,g)(*\tw)$ is not locally free of finite rank
over $\O_S(*\tw)$, although we can show that 
$\H^1_\dR(f,g)_{|S^\circ}$ is locally free over $\O_{S^\circ}$
(see Theorem \ref{dRT}). 
It means that we can not 
take a meromorphic frame of $\H^1_\dR(f,g)$ around $\tw=0$. 
This causes some problem since we 
would like to investigate 
the asymptotic behavior
when $\tw\to0$. 

To avoid this problem, we take some sub-sheaves $\H^1_{\dR, a, b}(f,g)$
of $\H^1_{\dR}(f,g)$ of free $\O_S$-modules
indexed by two integers $(a, b)\in \Z^2$ (see \S \ref{S lattice}).
These two integers correspond to 
the pole orders along the components $E=E_0\sqcup E_\infty$
where $E_0=g^{-1}(0)\setminus P$ and $E_\infty=g^{-1}(\infty)\setminus P$.  
If $a\leq a'$ and $b\leq b'$, then we have the inclusion
$\H^1_{\dR, a, b}(f,g)\subset \H^1_{\dR,a',b'}(f,g)$. 
The operators $\nabla_\euler$ and $\shift$ act as
\[\nabla_\euler\colon \H_{\dR, a, b}^1(f,g)\to\H_{\dR,a, b}^1(f,g),
\text{ and }\shift\colon \H^1_{\dR, a, b}(f,g)\to \sigma_*\H^1_{\dR,a+1,b-1}(f,g).\]
The limit $\lim_{a, b\to\infty}\H^1_{\dR,a, b}(f,g)$
is isomorphic to $\H^1_{\dR}(f,g)$
and the other limits when
$(a, b)\to (\infty,-\infty), (-\infty,\infty)$ and $(-\infty,-\infty)$
also exist and have geometric meaning 
concerning the asymptotic behavior 
of the de Rham complexes 
along the divisor $E$
(see Proposition \ref{limits} for more precise).


\subsubsection{Betti homology}\label{intro Betti}
Put $Y=X\setminus D$ and $\cal{Y}^\circ=S^\circ\times Y$. 
Let $\k$ be a subfield in $\C$. 
Then we consider the following local system of 
$\k[q^{\pm1}]$-modules on $\cal{Y}$:
\begin{align*}
\K(f,g)\coloneqq \k[q^{\pm 1}]e^{-f/\tw}g^{\eq/\tw}\subset \M(f,g)_{|\cal{Y}^\circ},
\end{align*}
where we put $q=\exp(2\pi\i\eq/\tw)$. 
Although $g^{\eq/\tw}$ is a multivalued function, 
the submodule $\K(f,g)$ is well defined. 
We regard $\K(f,g)$ as a counterpart 
of the local system of flat sections of $\cal{M}(f)_{|\mathfrak{X}\setminus \mathfrak{P}}$. 

In the case where $E=D\setminus P$ is empty, 
we consider the family of rapidly decay homology group
to obtain a local system 
$\H^\rd_1(f,g)$ of 
$\k[q^{\pm 1}]$-modules on $S^\circ$,
which we regard as a counterpart of 
$\cal{H}^\rd_1(f)$. 

In the general case where $E$ is not necessarily empty, as in 
\S \ref{intro dR}, 
we should consider 
moderate growth or rapid decay 
condition on the components $E_0$ and $E_\infty$
and construct four kinds of 
local systems 
$\H^\mg_1(f,g)$, $\H^\Be_{1,E_0!}(f, g)$, $\H^\Be_{1,E_\infty!}(f, g)$
and $\H^\rd_1(f,g)$ of $\k[q^{\pm 1}]$-modules on $S^\circ$.

Using the relations of 
(co)homology intersection pairings and period integrals 
(c.f. \cite{MMT}, \cite{FSY}, and a review in \S \ref{3.4} in this case),
we obtain the 
inclusion of the four sheaves above to the four limits considered in \S \ref{intro dR}.
In particular, we have the inclusion 
$\H^\mg_1(f,g)\hookrightarrow \H^1_{\dR}(f,g)$.

Glueing 
the sheaves $\H^\mg_1(f,g)$, $\H^\Be_{1,E_0!}(f, g)$, $\H^\Be_{1,E_\infty!}(f, g)$
and $\H^\rd_1(f,g)$
on $S^*\coloneqq S^\circ\setminus \{\eq=0\}$,
we obtain a quasi-local system
$\H^\Be_{1,\bm{0}}$ on $S^*$ (see \S \ref{Glueing Betti}). 
Put $\H^1_{\dR,\bm0}=\H^1_{\dR,0,0}$.
The inclusions defined
above
induce
the inclusion
$\H^\Be_{1,\bm{0}}\hookrightarrow \H^1_{\dR,\bm0|S^*} $
such that it induces the isomorphism
$\H^\Be_{1,\bm{0}}\otimes\O_{S^*}\simeq \H^1_{\dR,\bm0|S^*}$
(see \S \ref{Glueing Betti}).
This inclusion plays a similar role as the inclusion 
$\cal{H}^\rd_1(f)\hookrightarrow \cal{H}^1_\dR(f)$. 

\subsubsection{Main theorem and an example}

Take the holomorphic map 
$\phi_S\colon B\to S$ defined by $\phi_S(u, v)=(uv, v)$,
which induces an isomorphism 
$\phi_S\colon B^*\simeqto S^*$. 
Put
\begin{align*}
\L^\Be(f,g)\coloneqq \imath_T^{-1}\jmath_{B*}\phi_{S}^{-1}\H^\Be_{1,\bm0}(f,g),
\end{align*}
which is a quasi-local system on $T$. 
Using a local frame of $\phi_S^*\H^1_{\dR,\bm0}$ around the origin 
of $B$, 
we can define the pre-Stokes filtration $\L^\Be_{\leqslant}(f,g)$ on 
$\L^\Be(f,g)$.

Let $\Crit(f)$ denote the set of critical points 
of $f_{|U}$. 
Assume 
$E\cap \Crit(f)=\emptyset$. 
Under this assumption,
we can define the goodness condition on $(f,g)$
(See Definition \ref{good pair}).
The following is the main theorem of this paper:
\begin{theorem}[Theorem \ref{main theorem}]
Assume that $(f,g)$ is good. 
Then the pair
 \[(\L^\Be(f,g),\L^\Be_{\leqslant}(f,g))\]
is a good Stokes filtered quasi-local system. 
\end{theorem}

As an easy example, 
we consider the case 
$X=\P^1$, $f=z$ and $g=z$,
where $z$ denotes the non-homogeneous coordinate on 
the projective line $\P^1$. 
In this case, $E=E_0=\{0\}$ is non-empty,
and the differential-difference equation 
essentially corresponds to the difference equation 
for gamma function (See Example \ref{EXG}).
The Stokes filtered quasi-local system $\L^\Be(z, z)$
is naturally isomorphic to $(\L_{\Gamma},\L_{\Gamma\leqslant})$ defined as follows 
(see Example \ref{GAMMA FUN} and \S \ref{last examples}):
Let $\L_\Gamma$ be a sub-sheaf of 
$\imath_T^{-1}\jmath_{B*}\O_{B^*}$ defined by
\begin{align*}
\L_{\Gamma}(V)\coloneqq 
\begin{cases}
\C[q^{\pm 1}]u^{u^{-1}}v^{u^{-1}}\Gamma(u^{-1})& (V\cap T_{\R_-}= \emptyset)\\
\C[q^{\pm 1}]u^{u^{-1}}v^{u^{-1}}(1-q)\Gamma(u^{-1})& (V\cap T_{\R_-}\neq \emptyset)
\end{cases}
\end{align*}
where $V$ is a small open subset of $T$ and $\Gamma(u^{-1})$ denotes the gamma function.
Then the moderate growth condition on $\imath_T^{-1}\jmath_{B*}\O_{B^*}$
induces the filtration $\L_{\Gamma,\leqslant}$. 
By the Stirling formula and the reflection formula for gamma function,
the graded part is given by
\begin{align*}
\tau_!\gr\L_{\Gamma}\simeq \C[q^{\pm 1}]u^{1/2}v^{u^{-1}}e^{-u^{-1}}.
\end{align*}
In this way, the theory of the Stokes filtered quasi-local system 
contains the sheaf theoretic expression of asymptotic behavior of
solutions to some classical difference equations. 
We will also deal with the 
example related to some cylindrical functions.
\subsection{Further direction related to mirror symmetry}\label{Intro mirror}
One of the motivations of this study 
is to formulate the equivariant version 
of Dubrovin's conjecture
from the viewpoint of the author's previous study 
with F. Sanda \cite{Sanda-Shamoto}. 
See \cite{FIMS}, \cite{TV}, and \cite{CV} for pioneering works 
on this topic. 

From the viewpoint of mirror symmetry, 
the differential-difference module considered here 
corresponds to the equivariant quantum cohomology
with the grading operator (as a differential operator) and the shift operator (as a difference operator). 
The example of gamma function corresponds to 
the affine line $\mathbb{A}^1$ with the canonical 
$\C^*$-action.  
The example related to cylindrical functions corresponds to 
the projective line $\P^1$ with the $\C^*$-action. 

To relate our Stokes filtered quasi-local systems to 
the equivariant derived categories, 
it seems important to formulate 
the notion of Stokes data
which should be described in terms of modules over $\k[q^{\pm 1}]$.

\subsection{Notations}
For a complex number $\alpha$, 
$\Re(\alpha)$ and $\Im(\alpha)$ denote the real and imaginary part of $\alpha$, respectively.
For a complex manifold $M$, 
$\O_M$ denotes the sheaf of holomorphic functions on $M$. 
For a meromorphic function $F$ on $M$, 
$(F)_0$ and $(F)_\infty$ denote the 
zero divisor of $F$ and the pole divisor of $F$, respectively.  
$|(F)_0|$ and $|(F)_\infty|$ denote their supports. 
For a hypersurface $N\subset M$, 
 $\O_M(*N)$ denote the sheaf of meromorphic functions 
 whose poles are contained in $N$. 

\section{Stokes filtered quasi-local systems}\label{stokes}
\subsection{A sheaf of ordered abelian groups}
Let $B=\C^2$ denote the complex surface with coordinate $(u, v)$.
Let ${Z}$ denote the divisor $|(uv)_0|$ in $B$. 
We take the real blowing up 
\begin{align*}
\varpi_B\colon \widetilde{B}=\bl^\R_{{Z}}(B)\longrightarrow B
\end{align*}
of $B$ along the normal crossing divisor ${Z}$, which is identified with 
the projection
\begin{align*}
\(\R_{\geq 0}\times (\R/2\pi\Z)\)^2\longrightarrow \C^2,
\quad ((r_u,\theta_u),(r_v,\theta_v))\mapsto (u, v)=(r_u e^{\i \theta_u}, r_v e^{\i \theta_v}).
\end{align*}
Let $B^*\coloneqq B\setminus {Z}$ be the complement 
of ${Z}$ in $B$.
Let $\widetilde{\jmath}_B\colon B^*\to \widetilde{B}$
denote the inclusion. 
Let $\varpi_{v}\colon \widetilde{B}\to \C$
denote the projection to $v$-plane. 
\begin{definition}\label{tilde Q}
Let $\widetilde{\IQ}$ denote the sheaf of $\Z$-submodules of $\widetilde{\jmath}_{B*}\O_{B^*}$
locally generated by the sections of the form
\begin{align*}
&u^{-1}\(n\log v+\frac{h(v)}{v}\)&(n\in\Z, h(v)\in \varpi_v^{-1}\O_{\C})
\end{align*}
where the branch of $\log v$ is locally defined.
\end{definition}
Let $\widetilde{\jmath}_{B*}\O^{\rm lb}_{B^*}$
denote the subsheaf of $\widetilde{\jmath}_{B*}\O_{B^*}$
whose sections are locally bounded.
In other words, 
for an open subset $V\subset \widetilde{B}$, 
a section $\varphi\in\widetilde{\jmath}_{B*}\O_{B^*}(V)=\O_{B^*}(V\cap B^*)$
is a section of $\widetilde{\jmath}_{B*}\O^{\rm lb}_{B^*}$
if and only if the following condition is satisfied:
for any compact subset $K\subset V$, 
there exists a positive constant $C_K>0$
such that $|\varphi(u, v)|<C_K$
for any $(u, v)\in K\cap B^*$. 
Let $T$ be the corner of $\widetilde{B}$, which is identified with $(\R/2\pi\Z)^2$. 
Let $\imath_T\colon T\hookrightarrow \widetilde{B}$ denote the inclusion.
\begin{definition}\label{IQ definition}
We set $\IQ\coloneqq \imath_T^{-1}\widetilde{\IQ}/(\widetilde{\IQ}\cap \widetilde{\jmath}_{B*}\O^{\rm lb}_{B^*})$. 
\end{definition}
Let $\A_{\widetilde{B}}^{\leqslant {Z}}$ 
denote the subsheaf of $\widetilde{\jmath}_{B*} \O_{B^*}$
whose section has moderate growth along $\partial \widetilde{B}$ (\cite[\S 8.3]{Sabbah}).
Recall that a section $\varphi$ of $\widetilde{\jmath}_{B*} \O_{B^*}(V)$
for an open subset $V\subset \widetilde{B}$ is in $\A_{\widetilde{B}}^{\leqslant {Z}}(V)$
if and only if, for any compact subset $K\subset V$, there exists
$N_K\geqslant 0$ and $C_K>0$ such that \[|\varphi|<C_K|uv|^{-N_K}\]
on $K\cap B^*$. 
Let $\log \A_{\widetilde{B}}^{\leqslant {Z}}$ denote  the subsheaf of $\widetilde{\jmath}_{B*} \O_{B^*}$
locally generated by the the sections with the property that the exponential of them
have moderate growth along $\partial \widetilde{B}$.
In other words, we put 
\[\log \A_{\widetilde{B}}^{\leqslant {Z}}\coloneqq \exp^{-1}(\A_{\widetilde{B}}^{\leqslant {Z}})\]
where $\exp\colon \widetilde{\jmath}_{B*}\O_{B^*}\to \widetilde{\jmath}_{B*}\O_{B^*}$, 
$\varphi\mapsto e^\varphi$ denotes the
exponential map.
\begin{definition}\label{IQ order}
Let $\IQ_{\leqslant 0}$ be the subsheaf of 
$\IQ$ defined as the restriction of the quotient 
$
(\widetilde{\IQ}\cap\log \A_{\widetilde{B}}^{\leqslant Z})/
(\widetilde{\IQ}\cap \widetilde{\jmath}_{B*} \O_{B^*}^{\rm lb})$
to ${\T}$. Note that we have $\IQ_{\leqslant 0}\cap (-\IQ_{\leqslant 0})=0$. 
For two sections $\varphi, \psi\in \IQ(V)$
for an open subset $V\subset {\T}$, 
we define the order $\leqslant_V$ on $\IQ(V)$ by
\begin{align*}
\varphi\leqslant_V\psi\overset{\text{def}}{\Longleftrightarrow}
\varphi-\psi\in \IQ_{\leqslant 0}. 
\end{align*}
We also denote $\varphi <_V \psi $ if and only if $\varphi\leqslant_V \psi$ and 
$\varphi\neq \psi$. 
\end{definition}
In the following, we regard $\IQ=(\IQ,\leqslant)$
as the sheaf of ordered abelian groups. 
For every $\bm{\theta}\in\T$, we also use the following notation: 
For two germs $\varphi_{\bm{\theta}}, \psi_{\bm{\theta}}\in\IQ_{\bm{\theta}}$, 
we write $\varphi_{\bm{\theta}}\leqslant_{\bm{\theta}} \psi_{\bm{\theta}}$ if and only if 
there exists a representative $\varphi, \psi\in \IQ(V)$ on a open neighborhood $V$ of $\bm{\theta}$
such that $\varphi\leqslant_V\psi$. 

For $n\in \Z$, and $h(v)\in \O_{\C,0}$
we have the following sub-sheaf of sets in $\IQ$:
\begin{align*}
\Phi_{n, h(v)}\coloneqq \left[u^{-1}\(n \log v +\frac{h(v)}{v}+2\pi\i\Z\)  \right].
\end{align*}
\begin{definition}\label{GOODNESS}
 A finite disjoint union $\Phi=\bigsqcup_{j=1}^m \Phi_{n_j, h_j(v)}$ 
 $(n_j\in \Z, h_j(v)\in \O_{\C,0})$
 is called a \textit{good factor} if 
 $h_i(0)\neq h_j(0)$ or $n_i-n_j\neq 0$ for $i\neq j$. 
\end{definition}
\begin{definition}
Let $\Phi=\bigsqcup_{j=1}^m \Phi_{n_j, h_j(v)}$ 
 $(n_j\in \Z, h_j(v)\in \O_{\C,0})$ be a good factor.  
For each pair $i, j$ with $i, j=1,\dots, m$,  $h_i(0)-h_j(0)\neq 0$, 
we define \textit{the Stokes line} $\mathrm{St}_{i j}(\Phi)$  by
\begin{align*}
\mathrm{St}_{i j}(\Phi)
\coloneqq \left\{(\theta_u, \theta_v)\in \T\middle| \Re\(e^{-\i(\theta_u+\theta_v)}(h_i(0)-h_j(0))\)=0\right\}.
\end{align*}
For a pair $i, j$ with $h_i(0)=h_j(0)$ and $n_i-n_j\neq 0$, 
we set 
\begin{align*}
\mathrm{St}_{i j}(\Phi)\coloneqq \{(\theta_u, \theta_v)\in \T\mid \theta_u=\pm \pi/2 \}.
\end{align*}
For each $i=1,\dots,m$, we set 
\begin{align*}
\St_{ii}(\Phi)\coloneqq \{(\theta_ v, \theta_u)\in \T\mid \theta_u= 0 \text{ or }\pi \}. 
\end{align*}
\end{definition}
\begin{remark}
For two sections $\varphi, \psi\in \IQ(V)$
on an open subset $V\subset T$, 
we set 
\begin{align*}
V_{\varphi\leqslant \psi}\coloneqq \{t\in V\mid \varphi_t \leqslant_t \psi_t \}, 
\end{align*}
which is an open subset of $V$. 
Let $\St(\varphi, \psi)$ denote the boundary of $V_{\varphi\leqslant \psi}$ in $V$.
Assume that $\varphi$ and $\psi$ are sections of $\Phi_{n_i, h_i(v)}$ and $\Phi_{n_j, h_j(v)}$
for a good covering $\Phi=\bigsqcup_{j=1}^m \Phi_{n_j, h_j(v)}$, respectively.  
Then,   we have 
\begin{align*}
\St(\varphi, \psi)=\St_{i, j}(\Phi)\cap V. 
\end{align*}
Indeed, if $h_i(0)\neq h_j(0)$, then the equality holds because $\lim_{ v\to 0} v\log  v=0 $. 
If $h_i(0)=h_j(0)$
and $i\neq j$, then 
 we should have $n_{i j}\coloneqq n_i-n_j\neq 0$ 
by the goodness of $\Phi$. 
The following expression implies the equality:
\begin{align*}
\varphi-\psi=r_u^{-1} e^{-\i\theta_u}(n_{ij}\log r_ v+n_{ij}\i\theta_ v +\gamma( v))
\end{align*}
where $\gamma( v)\in \O_{\C,0}$. 
If $i=j$, then we have $\varphi-\psi=u^{-1}2\pi \i k$ for some $k\in\Z$, which also implies 
$\St(\varphi, \psi)=\St_{i, j}(\Phi)\cap V$.
In this sense, $\St_{i, j}(\Phi)$ describes the Stokes lines for sections of $\Phi$. 
\end{remark}
\subsection{Pre-Stokes filtrations}\label{2.2}
Let ${\tau}\colon \IQ^\et\to T$ denote the \'etale space of $\IQ$. 
Note that $\IQ^\et$ is a Hausdorff space. 
For $m\in \Z$, let $\rho(m)$ denote the class in $\IQ(T)$ represented by $2\pi\i m u^{-1}$. 

\begin{definition}\label{pre-q}
Let $\k$ be a field. 
Let $\L$ be a sheaf of $\k[q^{\pm 1}]$-modules on $T$.
A   \textit{pre-Stokes filtration} on $\L$ is a subsheaf $\L_{\leqslant}\subset {\tau}^{-1}\L$
of $\k$-vector spaces satisfying the following conditions:
\begin{enumerate}
\item For each $\bm{\theta}\in T$, and $\psi_{\bm{\theta}}, \varphi_{\bm{\theta}}\in {\tau}^{-1}(\bm{\theta})$
with $\psi_{\bm{\theta}}\leqslant_{\bm{\theta}}\varphi_{\bm{\theta}}$, we have 
\begin{align*}
\L_{\leqslant \psi_{\bm{\theta}}}\subset\L_{\leqslant \varphi_{\bm{\theta}}}
\end{align*}
as subsets in $({\tau}^{-1}\L)_{\psi_{\bm{\theta}}}=({\tau}^{-1}\L)_{\varphi_{\bm{\theta}}}=\L_{\bm{\theta}}$.
\item For each $\bm{\theta}\in T$, $\varphi_{\bm{\theta}}\in {\tau}^{-1}(\bm{\theta})$, 
and $m\in \Z$, we have 
\begin{align*}
q^m\cdot \L_{\leqslant \varphi_{\bm{\theta}}}=
 \L_{\leqslant \varphi_{\bm{\theta}}+\rho(m)}
\end{align*}
also as subsets in $\L_{\bm{\theta}}$. 
The action of $q^m$ on the left hand side comes from the $\k[q^{\pm 1}]$-module structure on $\L_{\bm{\theta}}$. 
\end{enumerate}
\end{definition}

Since $\IQ^\et$ is a Hausdorff space, 
there exists a unique subsheaf $\L_{<}$ of $\L_{\leqslant}$
such that, 
for any $\bm{\theta}\in T$ and $\varphi_{\bm{\theta}}\in \IQ_{\bm \theta}$, 
we have $\L_{<\varphi_{\bm\theta}}=\sum_{\psi_{\bm\theta}<\varphi_{\bm\theta}}\L_{\leqslant \psi_{\bm\theta}}$. 
\begin{definition}[c.f. {\cite[Definition 1.34]{Sabbah}}]\label{GR}
For a pre-Stokes filtration $ \L_\leqslant$
on a sheaf $\L$ of $\k[q^{\pm 1}]$-modules, 
let $\gr\L$ denote the quotient sheaf $\L_{\leqslant}/\L_{<}$. 
For a point $\varphi\in \IQ^\et$, 
the stalk of $\gr\L$ at $\varphi$ is denoted by 
$\gr_\varphi\L$. 
\end{definition}

By condition (2) in Definition \ref{pre-q},  
the proper push-forward ${\tau_!}\gr\L$
is naturally equipped with the structure of sheaf of  $\k[q^{\pm 1}]$-modules.
For each point 
$\bm{\theta}\in \T$, the action of $\k[q^{\pm 1}]$ on $({\tau_!}\gr\L)_{\bm\theta}$ is described as follows: 
we have 
\begin{align*}
({\tau_!}\gr\L)_{\bm{\theta}}=\bigoplus_{\varphi_{\bm\theta}\in{\tau}^{-1}(\bm\theta)} \gr_{\varphi_{\bm\theta}}\L.
\end{align*}
By condition (2) in Definition \ref{pre-q}, the action
$q^m\colon \L_{\leqslant \varphi_{\bm\theta}}\to \L_{\leqslant \varphi_{\bm\theta}+\rho(m)}$
 induces $q^m\colon \gr_{\varphi_{\bm{\theta}}}\L\to \gr_{\varphi_{\bm\theta}+\rho(m)}\L$
for $m\in \Z$. 

Similar discussion as in  \cite[Example 1.35]{Sabbah}
shows that ${\tau_!}\gr\L$ is naturally equipped with 
a pre-Stokes filtration $({\tau_!}\gr\L)_{\leqslant}$.
We have
\begin{align*}
({\tau_!}\gr\L)_{\leqslant \varphi_{\bm\theta}}=\bigoplus_{\psi_{\bm\theta}\leqslant_{\bm\theta}\varphi_{\bm\theta}} \gr_{\psi_{\bm\theta}}\L
\end{align*}
at each point $\varphi_{\bm\theta}\in {\tau}^{-1}(\bm\theta)$, $(\bm\theta\in \T)$.
We have the identification
$\gr({\tau_!}\gr\L)=\gr \L$. 

\begin{definition}\label{q-covering}
Let $\L_\leqslant$ be a pre-Stokes filtration on a sheaf $\L$ of $\k[q^{\pm 1}]$-modules on $T$. 
Let  $\Phi(\L,\L_{\leqslant})$ denote the support of 
$\gr\L$ on $\IQ^\et$.
We call $\Phi(\L,\L_{\leqslant})$ the \textit{exponential factor}
of the pair $(\L, \L_{\leqslant})$. 
\end{definition}
Fix a finite union $\Phi=\bigsqcup_{j=1}^m \Phi_j$ of 
$\Phi_j=\Phi_{n_j, h_j(v)}$ with $n_j\in \Z$, $h_j(v)\in \O_{\C,0}$.
Let $(\L,\L_\leqslant)$
be a pre-Stokes filtered sheaf of $\C[q^{\pm 1}]$-module
with $\Phi(\L,\L_\leqslant)\subset \Phi$. 
For each $j$, put 
\begin{align*}
&\L_{<\Phi_{j}}(V)\coloneqq \bigcap_{\varphi\in \Gamma(V,\Phi_{j})}\L_{\leqslant \varphi}(V),
&\L_{\leqslant \Phi_{j}}(V)\coloneqq \sum_{\varphi\in \Gamma(V,\Phi_{j})}\L_{\leqslant \varphi}(V).
\end{align*}
for contractible open subsets $V\subset T$. 
We then obtain the family of sub-sheaves
$\L_{<\Phi_{j}}\subset \L_{\leqslant\Phi_{j}}\subset \L$ of 
$\k[q^{\pm 1}]$-modules, which we will call \textit{coarse filtrations}. 
We then put 
\begin{align*}
\Gr_{\Phi_{j}}(\L )\coloneqq \L_{\leqslant \Phi_{j}}/\L_{<\Phi_{j}}.
\end{align*}
The pre-Stokes filtration 
on $\L$ induces a 
pre-Stokes filtration on $\Gr_{\Phi_{j}}(\L)$.

\subsection{Stokes filtrations}\label{QLS}
An $\R$-constructible sheaf $\F$ of locally finitely generated free 
$\k[q^{\pm 1}]$-modules 
on a real analytic manifold will be called  a
\textit{quasi-local system} in this paper
if it becomes a local system after the localization, i.e.
if the tensor product
\begin{align*}
\F(q)\coloneqq \F\otimes_{\k[q^{\pm 1}]} \k(q)
\end{align*}
is a local system.

Recall that 
$T$ is identified with the 
torus $\{(\theta_u,\theta_v)\mid \theta_u,\theta_v\in \R/2\pi \Z\}$. 
We set 
\begin{align*}
&T_{\R_+}\coloneqq \{(\theta_u,\theta_v)\in T\mid e^{-\i\theta_u}\in \R_{>0}\}, 
&&T_{\R_-}\coloneqq\{ e^{-\i\theta_u}\in \R_{<0}\},\\
&T_+\coloneqq\{\Im (e^{-\i\theta_u})>0\}, 
&&T_-\coloneqq\{\Im (e^{-\i\theta_u})<0\}.
\end{align*}
Then $\Theta\coloneqq\{T_{\R_+}, T_{\R_-},T_+, T_-\}$
is a stratification of $T$. 
\begin{definition}\label{Q-local}
By \textit{a quasi-local system on $(T,\Theta)$},
we mean a quasi-local system on $T$
which is constructible with respect to $\Theta$. 
In other words, 
a sheaf $\L$ of locally finitely generated free $\k[q^{\pm 1}]$-modules
on $T$ is called a \textit{quasi-local system on $(T,\Theta)$}
if it satisfies the following conditions:
\begin{enumerate}
\item
There is a non-negative integer $r$ (\textit{the rank of $\L$})
such the restriction 
$\L_{|T_\star}$ to each strata $T_\star\in \Theta$ is a local system of free $\k[q^{\pm 1}]$-modules of rank $r$.  
\item For any two connected open subsets $W, V\subset T$ with $W \subset V$, 
the restriction map
$\L(V)\to \L(W)$ is injective.
\end{enumerate}
\end{definition}

\begin{lemma}\label{Glueing L}
Let $\L$ be a quasi-local system on $(T,\Theta)$. 
Let $r$ be the rank of $\L$. 
Then, there exist a unique pair of
local systems $\L^{+}$ on $T\setminus T_{\R_-}$ and $\L^{-}$ on $T\setminus T_{\R_+}$
of free $\k[q^{\pm 1}]$-modules of rank $r$
with the following properties:
\begin{enumerate}
\item $\L^{+}\subset \L_{|T\setminus T_{\R_-}}$ and 
$\L^-\subset \L_{|T\setminus T_{\R_+}}$.
\item $\L_{|T_{\R_+}}=\L^+_{|T_{\R_+}}$ and $\L_{|T_{\R_-}}=\L^-_{|T_{\R_-}}$.\qed
\end{enumerate}
\end{lemma}
\begin{definition}
A quasi-local system $\L$ on $(T,\Theta)$
is called \textit{saturated} if
we have 
$\L(V)=\L^+(V)+\L^-(V)$
for any open $V\subset T_+\cup T_-$.
If we have $\L(V)=\L^+(V)$ (resp. $\L(V)=\L^-(V)$), then 
$\L$ is called \textit{$+$saturated} (resp. \textit{$-$saturated}).
\end{definition}
Let $\L_{\leqslant}$
be a pre-Stokes filtration on 
a quasi-local system $\L$ on $(T,\Theta)$. 
Then, $\L_{\leqslant}$
induces the filtrations on $\L^\pm$ in a natural way:
\begin{align*}
\L^\pm_{\leqslant }\coloneqq 
\tau_{|\tau^{-1}(T\setminus T_{\R_\mp})}^{-1}\L^\pm\cap \L_{\leqslant |\tau^{-1}(T\setminus T_{\R_\mp})},
\end{align*}
We define 
$\L^\pm_{<}$ 
 and  $\gr\L^\pm$ in the same way as above. 
 \begin{lemma}
The natural morphisms
\begin{align*}
\gr\L^\pm\to \gr\L_{|\tau^{-1}{(T\setminus T_{\R_\mp}})}.
\end{align*}
are isomorphisms. 
\end{lemma}
\begin{proof}
We shall see that the morphisms are injective and surjective on 
$T_+\cup T_-$. We consider the morphism $\gr\L^+\to \gr\L$
(The case of the morphism  $\gr\L^-\to \gr\L$ can be proved in a similar way). 
Firstly, we shall see that 
the morphism is injective. 
If a section $\mathfrak{s}\in \L_{\leqslant\varphi}^{+}$ 
is in $\L_{<\varphi}$, then it is contained in $\L^+_{<\varphi}$ by the definition of
the filtration. This proves the injectivity. 
Next, 
we shall see that the morphism is surjective. 
Let $\mathfrak{s}$ be a section of $\L_{\leqslant\varphi}$
on an open subset in $T_+$.  
Then, there exists a polynomial
$P(q)\in\k[q]$ such that 
$P(0)=1$ and 
$P(q)\mathfrak{s}\in \L^+$. 
By the condition (2) in Definition \ref{pre-q}, 
we obtain $P(q)\mathfrak{s}\in \L^+_{\leqslant \varphi}$
and $\gr_\varphi(P(q)\mathfrak{s})=\gr_\varphi(\mathfrak{s})$, 
which proves the surjectivity.
\end{proof}
\begin{definition}[Stokes filtration]\label{Def q-Stokes}
Let $\Phi=\bigsqcup_{j=1}^m \Phi_j$ be a finite disjoint union of $\Phi_j=\Phi_{n_j, h_j(v)}$
$(n_j\in \Z, h_j(v)\in \O_{\C,0})$. 
Let $\L$ be a saturated quasi-local system on $(T,\Theta)$.
 Let $\L_{\leqslant}$
be a 
pre-Stokes filtrations on $\L$
with $\Phi(\L,\L_\leqslant)\subset \Phi$. 
Then, $\L_{\leqslant}$
is called a \textit{Stokes filtration} if
the following conditions are satisfied:
\begin{enumerate}
\item $\gr\L$
is a local system of $\k$-modules on $\Phi(\L,\L_{\leqslant})$.
\item
For each point $\bm\theta\in T$ (resp. $\bm\theta_\pm\in T\setminus T_{\R_\mp}$), 
there exists
an isomorphism
\begin{align*}
&\eta_{{\bm\theta}}\colon \tau_!\gr\L_{\bm\theta}\simeqto \L_{\bm\theta},
&(\text{resp. } \eta_{\bm\theta_\pm}\colon \tau_!\gr\L_{\bm\theta_\pm}
\simeqto \L^\pm_{\bm\theta_\pm})
\end{align*}
 such that 
 \begin{enumerate}
 \item The isomorphism $\eta_{\bm\theta}$ (resp. $\eta_{\bm\theta_\pm}$)
 preserves the filtration, i.e.
 for any germ $\varphi \in \IQ_{\bm\theta}$ 
 (resp. $\varphi_\pm\in \IQ_{\bm\theta_\pm}$), 
 it induces the morphism
 \begin{align*}
 &\eta_{\bm\theta}\colon \gr_\varphi\L\to \L_{\leqslant \varphi}
 &(\text{resp. } 
 \eta_{\bm\theta_\pm}\colon \gr_{\varphi_\pm}\L
\to \L^\pm_{\leqslant \varphi_\pm})
 \end{align*}
 \item The associated morphism
 \begin{align*}
&\gr_\varphi(\eta_{\bm\theta})\colon \gr_\varphi\L\to 
\gr_\varphi\L
&(\text{resp. } \gr_{\varphi_\pm}(\eta_{\bm\theta_\pm})
\colon \gr_{\varphi_\pm}\L\to \gr_{\varphi_\pm}\L)
 \end{align*}
 is identity for any $\varphi$ (resp. $\varphi_\pm$). 
 \end{enumerate}
The morphisms $\eta_{{\bm\theta}}$ and $\eta_{{\bm\theta_\pm}}$
are called splittings of the filtrations. 
\item Each coarse graded part $\Gr_{\Phi_j}(\L)$ 
is a $+$saturated (resp. $-$saturated) quasi-local system on $(T,\Theta)$ if $n_j\in \Z_{\geq 0}$
(resp. $n_j\in \Z_{\leq 0}$).
\end{enumerate} 
The pair $(\L, \L_{\leqslant})$ of a saturated quasi-local system $\L$  and 
a Stokes filtration $\L_{\leqslant}$ on it is called a \textit{Stokes filtered
quasi-local system}. 
\end{definition}
\begin{remark}
The notion of Stokes filtered quasi-local systems
might be defined in a more general setting 
considering ramifications e.t.c. 
In that case, we should call the notion defined above 
non-ramified or of exponential type. 
We will not pursue such generalizations in the present paper, 
and do not add such adjectives. 
\end{remark}

\begin{definition}
A Stokes filtered quasi-local system is called \textit{good}
if its exponential factor $\Phi(\L,\L_{\leqslant})=\bigsqcup_{j=1}^m \Phi_{n_j, h_j(v)}$ is good. 
The \textit{Stokes direction} of $(\L,\L_{\leqslant})$
is defined as that of $\Phi(\L,\L_{\leqslant})$
and denoted by $\St_{i, j}(\L,\L_{\leqslant})$ $(i, j=1,\dots, m)$.
\end{definition}
\begin{example}\label{trivial q}
Let $(\k[q^{\pm 1}]_T,\k[q^{\pm 1}]_{\leqslant})$
be a pair of the constant sheaf $\k[q^{\pm 1}]_T$ on $T$ 
and the filtration $\k[q^{\pm 1}]_{T\leqslant}$
characterized by the conditions 
\begin{align*}
\Phi(\k[q^{\pm 1}]_T,\k[q^{\pm 1}]_{T\leqslant})=[u^{-1}2\pi\i\Z]
\end{align*}
and 
\begin{align*}
\k[q^{\pm 1}]_{T\leqslant 0}=
\begin{cases}
\k[q]_{T_+} & \text{ on $T_+$}\\
\k[q^{-1}]_{T_-}&\text{ on $T_-$}\\
\k_{T_\star} q^0&\text{ on $T_{\R_{+}}\cup T_{\R_{-}}$}.
\end{cases}
\end{align*}
Then $(\k[q^{\pm 1}]_T,\k[q^{\pm 1}]_{\leqslant})$
is a Stokes filtered quasi-local system.
\end{example}
\begin{example}[Gamma function]\label{GAMMA FUN}
Take a subsheaf 
$\L_{\Gamma}\subset \imath_T^{-1}\widetilde{\jmath}_{B*}\O_{B^*}$
by
\begin{align*}
\L_{\Gamma}(V)\coloneqq 
\begin{cases}
\C[q^{\pm 1}]u^{u^{-1}}v^{u^{-1}}\Gamma(u^{-1})& (V\cap T_{\R_-}= \emptyset)\\
\C[q^{\pm 1}]u^{u^{-1}}v^{u^{-1}}(1-q)\Gamma(u^{-1})& (V\cap T_{\R_-}\neq \emptyset)
\end{cases}
\end{align*}
where $V$ is a small open subset of $T$, $q=\exp(2\pi\i u^{-1})$, and $\Gamma(z)$
denotes the gamma function.
Define the filtration 
$\L_{\Gamma,\leqslant}$
as follows:
\begin{align*}
&\L_{\Gamma,\leqslant \varphi}\coloneqq \L_{\Gamma}\cap 
\imath_T^{-1}(e^\varphi\scr{A}_{\widetilde{B}}^{\leqslant Z})
&(\varphi\in \IQ).
\end{align*}
We also consider 
a subsheaf 
$\scr{G}_\Gamma\coloneqq\C[q^{\pm 1}]u^{1/2}v^{u^{-1}}e^{-u^{-1}}
\subset \imath_T^{-1}\widetilde{\jmath}_{B*}\O_{B^*}$, 
which is also equipped with the Stokes filtration in a
similar way. It is easy to see that $\scr{G}_\Gamma$ is a local system and 
that
$(\scr{G}_\Gamma, \scr{G}_{\Gamma,\leqslant})$
is a Stokes filtered quasi-local system
with
$\tau_!\gr\scr{G}_\Gamma=\scr{G}_\Gamma$.

By the Stirling formula,
we have
\begin{align*}
\Gamma(u^{-1})
=e^{-u^{-1}}u^{-u^{-1}+1/2}\sqrt{2\pi}\left\{1+O(u)\right\}
\end{align*}
when $|u|\to 0$ with $-\pi+\delta<\arg(u)<\pi-\delta$
for any $\delta>0$. 
It follows that 
we have an isomorphism
\begin{align*}
\eta_{T\setminus T_{\R_-}}\colon (\scr{G}_{\Gamma},\scr{G}_{\Gamma\leqslant})_{|T\setminus T_{\R_-}}\to 
(\L_\Gamma,\L_{\Gamma,\leqslant})_{|T\setminus T_{\R_-}}
\end{align*}
defined as
\begin{align*}
\eta_{T\setminus T_{\R_-}}\(u^{1/2}v^{u^{-1}}e^{-u^{-1}}\sqrt{2\pi}\)
=u^{u^{-1}}v^{u^{-1}}\Gamma(u^{-1}),
\end{align*}
where the branches of $\log u$ and $\log v$ are defined 
in a suitable way. 

By the reflection formula for Gamma function, 
we have
\begin{align*}
(1-q)\Gamma(u^{-1})=2\pi\i\frac{e^{\pi\i u^{-1}}}{\Gamma(1-u^{-1})}
\end{align*}
and hence
\begin{align*}
(1-q)u^{u^{-1}}v^{u^{-1}}\Gamma(u^{-1})=
u^{1/2}e^{-u^{-1}}v^{u^{-1}}\sqrt{2\pi}\{1+O(u)\}
\end{align*}
when $|u|\to 0$ with $\delta<\arg(u)<2\pi-\delta$.
It then also follows that
we have an isomorphism
\begin{align*}
\eta^+_{T\setminus T_{\R_+}}\colon
(\scr{G}_{\Gamma},\scr{G}_{\Gamma\leqslant})_{|T\setminus T_{\R_+}}\to 
(\L_\Gamma^+,\L_{\Gamma,\leqslant}^+)
\end{align*}
defined as
\begin{align*}
\eta^+_{T\setminus T_{\R_+}}(u^{1/2}v^{u^{-1}}e^{-u^{-1}}\sqrt{2\pi})
\coloneqq (1-q)u^{u^{-1}}v^{u^{-1}}\Gamma(u^{-1})
\end{align*}
where the branches of $\log u$ and $\log v$ are defined 
in a suitable way. 

In conclusion, we have
that $(\L_{\Gamma},\L_{\Gamma,\leqslant})$ is a Stokes filtered quasi-local 
system
such that
\begin{align*}
\tau_!\gr\L_{\Gamma}\simeq (\scr{G}_{\Gamma},\scr{G}_{\Gamma\leqslant})
\end{align*}
and 
$\Phi(\L_{\Gamma},\L_{\Gamma,\leqslant})=\left[u^{-1}(\log v-1+2\pi\i\Z)\right]$
hold.
\end{example}

\subsection{Strictness and Abelianity--Statement}
Let $(\L,\L_\leqslant)$ and $(\L',\L'_{\leqslant})$
be Stokes filtered quasi-local systems. 
\begin{definition}
\textit{A morphism $\xi\colon (\L,\L_\leqslant)\to (\L',\L'_{\leqslant})$ 
of Stokes filtered quasi-local systems}
is a morphism $\xi\colon \L\to \L'$ of sheaves of $\k[q^{\pm 1}]$-modules
such that 
the pull back $\tau^{-1}\xi$ induces the morphism
$\tau^{-1}\xi\colon \L_{\leqslant}\to \tau^{-1}\L'_{\leqslant}.$
The category of Stokes filtered quasi-local system 
is denoted by $\bm{\St}^q$.
\end{definition}

\begin{definition}
Let $\xi\colon (\L,\L_\leqslant)\to(\L',\L'_\leqslant)$
be a morphism in $\bm{\St}^q$.
$\xi$ is called \textit{strict}
if it satisfies 
$\xi(\L_{\leqslant \varphi})=\xi(\L_{\bm\theta})\cap \L'_{\leqslant\varphi}$
for any ${\bm\theta}\in T$ and $\varphi\in \tau^{-1}(\bm{\theta})$. 
\end{definition}
Let $\Phi$ be a good factor. 
Let $\bm{\St}^q_\Phi$ denote the full sub-category of $\bm{\St}^q$
whose objects $(\L,\L_\leqslant)$ satisfy $\Phi(\L,\L_\leqslant)\subset \Phi$. 
The following is the main theorem of this section:
\begin{theorem}\label{STRICTNESS}
$\bm{\St}^q_\Phi$ is an abelian category and every morphism in $\bm{\St}^q_\Phi$ is strict. 
\end{theorem}
The proof will be finished in \S \ref{end of the strictness}.
\begin{lemma}\label{first vanishing}
Let $I$ be an open interval.
Let $A_+$ and $A_-$  be principal ideal domains.   
Take $t\in I$  and let $I_+$ and $I_-$ be the connected components of 
$I\setminus\{t \}$. 
Let $\F$ be a sheaf of modules
such that the restrictions
$\F_{|I_+}, $ and $\F_{|I_-}$
are  local systems of finite rank free modules over $A_+$ and $A_-$, respectively. 
Then we have $H^1(I, \F)=0$. 
\end{lemma}
\begin{proof}
By the same discussion as \cite[Lemma 2.8]{Sabbah}, 
it reduces to show that
for an inclusion $\iota\colon (a, b)\hookrightarrow (a, b]$
with $a, b\in\R$, $a<b$ and a princepal ideal domain $A$, we have
\begin{align*}
H^1\((a, b], \iota_!A_{(a, b)}\)=0.
\end{align*} 
This is well known. 
\end{proof}

For a given $\bm\theta=(\theta_u,\theta_v)\in T$ and $a\in \Z$, 
take 
a circle
\begin{align*}
\ell_a^{\bm\theta}\colon S^1\to T,\quad e^{\i\theta}\mapsto (\theta_u+\theta,\theta_v+a\theta).
\end{align*}
The image $\ell_a^{\bm\theta}(S^1)$ is also denoted by 
$\ell_a^{\bm\theta}$. 
We firstly assume that 
$\Phi(\L,\L_\leqslant)=\Phi_{n, h(v)}$
for some $n\in \Z$, and $h(v)\in\O_{\C,0}$. %

\begin{lemma}\label{Split 1}
Let $(\L, \L_\leqslant)$ be an object in $\bm\St_{\Phi}^q$
with $\Phi= \Phi_{n, h(v)}$.
Let $I\subset \ell^{\bm\theta}_a$ be an open interval
such that 
$\mathrm{card}(I\cap (T_{\R_+}\cup T_{\R_-}))\leq 1$. 
Then, there exists a splitting
\begin{align*}
\eta_I\colon H^0(I, \tau_!\gr\L)\simeqto H^0(I, \L)
\end{align*}
of the filtration. 
If $\mathrm{card}(I\cap (T_{\R_+}\cup T_{\R_-}))=1$,
then the splitting is unique. 
\end{lemma}
\begin{proof}
We first note that 
$\Phi$ is trivial over $I$. 
The existence of such splitting follows
from Definition \ref{Def q-Stokes} and  Definition \ref{Q-local}.
We shall prove the uniqueness and 
give another proof of the existence using Lemma \ref{first vanishing}. 
Consider the exact sequence
\begin{align*}
0\longrightarrow \L_{<\varphi}\longrightarrow \L_{\leqslant \varphi}\longrightarrow \gr_\varphi\L
\longrightarrow 0
\end{align*}for $\varphi\in \Gamma(I,\Phi)$.
Taking the section on $I$, we obtain 
the long exact sequence
\begin{align*}
0\longrightarrow H^0(I,\L_{<\varphi})\to H^0(I,\L_{\leqslant\varphi })\to H^0(I,\gr_\varphi\L)
\to H^1(I, \L_{<\varphi})\to\cdots.
\end{align*}

Put 
$I_{\pm}\coloneqq I\cap T_{\pm}$. 
Then $\L_{<\varphi|I_{+}}$ (resp.  $\L_{<\varphi|I_{-}}$) is a
local system of free $\k[q]$-modules (resp. $k[q^{-1}]$-modules) of finite rank. 
Hence $H^1(I, \L_{<\varphi})=0$ by Lemma \ref{first vanishing}. 
We also have $H^0(I, \L_{<\varphi})=0$ when $I\cap (T_{\R_+}\cup T_{\R_-})\neq \emptyset$
since $\Phi=\Phi_{n,h(v)}$. 
Therefore, the morphism $ H^0(I,\L_{\leqslant\varphi })\to H^0(I,\gr_\varphi\L)$
is an isomorphism.
This fact implies the existence and the uniqueness of the splitting.
\end{proof}


\begin{lemma}\label{m=1}
Theorem $\ref{STRICTNESS}$ holds when $\Phi=\Phi_{n, h(v)}$
for $n\in\Z$ and $h(v)\in \O_{\C,0}$. 
\end{lemma}
\begin{proof}
Let $\lambda\colon (\L,\L_{\leqslant})\to (\L',\L'_{\leqslant})$
be a morphism in $\bm\St_\Phi^q$. 
For each point $\bm\theta\in T$, 
there exist an open interval $I\subset \ell_0^{\bm\theta}(S^1)$
such that (i) $\bm\theta\in I$, 
(ii) the restriction maps
\begin{align*}
\Gamma(I,\L)\to \L_{\bm\theta},\quad 
\Gamma(I,\L')\to \L'_{\bm\theta}
\end{align*}
are isomorphisms,
and (iii) $I\cap (T_{\R_+}\cup T_{\R_-})$
consists of exactly one point (We have used the condition (3) in Definition \ref{Def q-Stokes}). 
Then we have the following diagram by Lemma \ref{Split 1}:
\begin{align*}
\xymatrix{
\tau_!\gr\L_{\bm\theta}\ar@{.>}[d]
&\ar[l]\Gamma(I,\tau_!\gr\L)\ar[r]^{\eta_I}\ar@{.>}[d]
&\Gamma(I,\L)\ar[d]^{\Gamma(I,\lambda)}\ar[r]
&\L_{\bm\theta}\ar[d]^{\lambda_{\bm\theta}}\\
\tau_!\gr\L_{\bm\theta}
&\ar[l]\Gamma(I,\tau_!\gr\L')\ar[r]^{\eta'_I}
&\Gamma(I,\L')\ar[r]&\L'_{\bm\theta}
}
\end{align*}
where $\eta_I$ and $\eta_I'$
denote the unique splittings and dotted arrows denote the induced maps. 
Then, all morphisms preserve the filtrations, 
and by the condition (iii) above,
we obtain that the dotted arrows coincide with
\begin{align*}
\Gamma(I, \tau_!\gr(\lambda))\colon \Gamma(I,\tau_!\gr\L)\to \Gamma(I,\tau_!\gr\L'),
\text{ and }
\tau_!\gr(\lambda)_{\bm\theta}\colon \tau_!\gr\L_{\bm\theta}\to \tau_!\gr\L'_{\bm\theta}.
\end{align*}
In other words, $\lambda_{\bm\theta}$
is diagonalized and hence we obtain the strictness.

It also follows that 
we have 
\begin{align*}
\mathrm{Cok}(\lambda_{\bm\theta})\simeq \mathrm{Cok}(\gr_\varphi(\lambda))\otimes_\k \k[q^{\pm 1}]
\end{align*}
for any $\varphi\in \Phi_{\bm\theta}$.
This implies that
$\mathrm{Cok}(\lambda)$ again a saturated quasi-local system which
satisfies (3) in Definition \ref{Def q-Stokes}.
Then the proof of abelianity is the same as in the case of usual Stokes structures.
%
%
\end{proof}

\subsection{Strictness and Abelianity--Proof}\label{end of the strictness}
We shall consider the general case of Theorem \ref{STRICTNESS}. 
For a good factor 
$\Phi=\bigsqcup_{j=1}^m\Phi_{n_j, h_j(v)}$,
we put $\Phi_j\coloneqq \Phi_{n_j, h_j(v)}$
and 
\begin{align*}
\Gr_{\Phi}(\L)\coloneqq \bigoplus_{j=1}^m\Gr_{\Phi_j}(\L)
\end{align*}
for an object $(\L,\L_{\leqslant})\in \bm\St_{\Phi}^q$.
Note that 
$\Gr_{\Phi_j}(\L)\in \bm\St_{\Phi_j}^q$, $\Gr_{\Phi}(\L)\in \bm\St_{\Phi}^q$,
and $\gr(\Gr_{\Phi}(\L))=\gr\L$.
By the condition (2) in Definition \ref{Def q-Stokes},
there is an filtered isomorphism
\begin{align*}
\xi_{\bm\theta}\colon \Gr_{\Phi}(\L)_{\bm\theta}\simeqto \L_{\bm\theta}
\end{align*}
such that $\Gr_{\Phi_j}(\xi_{\bm\theta})$ is identity.

\begin{lemma}\label{key lemma1}
If an interval $I\subset \ell_a^{\bm\theta}(S^1)$ satisfies the condition that 
\begin{itemize}
\item
$I\cap \St_{i, j}(\L)$ is at most one point for any $i\neq j$,
\item
$I\cap (T_{\R_+}\cup T_{\R_-})=\emptyset$, 
\end{itemize}
then there exists
a splitting of the coarse filtration
\begin{align*}
\xi_I\colon \Gamma\(I, \Gr_{\Phi}(\L)\)\simeqto \Gamma(I,\L),
\end{align*}
i.e. $\xi_I$ preserves the coarse filtrations and 
the induced morphism
\begin{align*}
\Gr_{\Phi}(\xi_I)\colon \Gamma(I, \Gr_{\Phi}(\L))\to\Gamma(I, \Gr_{\Phi}(\L))
\end{align*}
is identity. 
Moreover, if $I\cap \St_{i, j}(\L)$ contains exactly one point for any $i\neq j$,
 then such $\xi_{I}$ is unique and in that case $\xi_{I}$
also preserves the Stokes filtration.
\end{lemma}
\begin{proof}
We firstly show that 
the first claim of this lemma (existence of the splitting of the coarse filtrations) is equivalent 
to the condition that 
$H^1(I, \L_{<\Phi_j})=0$ for every $j=1,\dots, m$. 
Consider the exact sequence 
\begin{align*}
0\longrightarrow \L_{<{\Phi_j}}\longrightarrow \L_{\leqslant {\Phi_j}}\longrightarrow
\Gr_{\Phi_j}\L\longrightarrow 0.
\end{align*}
Taking the section on $I$, 
we obtain the long exact sequence
\begin{align*}
\cdots \to H^0(I, \L_{<{\Phi_j}}) \to 
H^0(I, \L_{\leqslant {\Phi_j}})\to
H^0(I, \Gr_{\Phi_j} \L)\to
H^1(I, \L_{<{\Phi_j}})\to \cdots.
\end{align*}
Then $H^1(I, \L_{<\Phi_j})=0$ if and only if 
the morphism $H^0(I, \L_{\leqslant \Phi_j})\to H^0(I, \Gr_{\Phi_j} \L)$
is surjective. 
Let $e_1^{(j)},\dots, e_{r_j}^{(j)}$ be a free bases of $H^0(I, \Gr_{\Phi_j} \L)$. 
In other words, we have $H^0(I, \Gr_{\Phi_j} \L)=\bigoplus_{k=1}^{r_j}\k[q^{\pm 1}]e_k^{(j)}$. 
If $H^0(I, \L_{<\Phi_j})=0$, 
there exist sections $\widetilde{e}^{(j)}_1,\dots, \widetilde{e}^{(j)}_{r_j}\in H^0(I, \L_{\leqslant \Phi_j})$
such that $\widetilde{e}^{(j)}_k$ maps to ${e}^{(j)}_k$ for every $k=1,\dots, r_j$. 
Then we obtain the coarse splitting 
$\xi_I\colon \Gamma(I, \Gr_\Phi(\L))\to \Gamma(I,\L)$
defined as 
\begin{align*}
\xi_I({e}^{(j)}_k)=\widetilde{e}^{(j)}_k.
\end{align*}
To see that $\eta_I$ is an isomorphism, 
we take a point $\bm{\theta}\in I\setminus \bigcup_{i\neq j}\St_{i, j}(\L)$. 
Then the morphism 
\[\Gr_\Phi(\L)_{\bm\theta}\to \Gr_\Phi(\L)_{\bm\theta}\]
given by 
\begin{align*}
\Gr_\Phi(\L)_{\bm\theta}\xleftarrow{\sim}
\Gamma(I,\Gr_{\Phi}(\L))\xrightarrow{\xi_I}
\Gamma(I,\L)\xleftarrow{\sim} \L_{\bm\theta}
\xrightarrow{\xi_{\bm\theta}^{-1}}
\Gr_\Phi(\L)_{\bm{\theta}}
\end{align*}
is upper triangular with respect to the decomposition $\Gr_{\Phi}(\L)=\bigoplus_j \Gr_{\Phi_j}\L$ 
with the total order on $\{\Phi_j\}_j$ at $\bm{\theta}$
with identity diagonal, hence is an isomorphism.  
Conversely, if we obtain the coarse splitting $\xi_I$, we take the lift $\widetilde{e}_k^{(k)}$
by $\xi_I({e}^{(j)}_k)=\widetilde{e}^{(j)}_k$. 

We shall prove the claim 
by the induction on the number $N$ of elements in the 
set
$I\cap \(\bigcup_{i\neq j}\St_{i j}(\L)\).$
If $N=1$, we have $H^1(I, \L_{<\Phi_j})=0$ by Lemma \ref{first vanishing}. 
We then consider the case $N>1$. 
Take a connected component 
$I_0$ of $I\cap \(\bigcup_{i\neq j}\St_{i j}(\L)\)$
such that the closure of $I_0$ in $S^1$
is contained in $I$. 
Take the covering $I=I_1\cup I_2$ 
such that each $I_\ell$ is the image of an open interval in $S^1$ and 
$I_1\cap I_2=I_0$. 
The boundary $\partial I_0$ is two points $\{i_1, i_2\}$
such that $i_1\in I_1$ and $i_2\in I_2$. 
By the induction assumption, $H^1(I_\ell,\L_{<{\Phi_j}})=0$
for any section ${\Phi_j}\in H^0(I,\IQ)$ and $\ell=1,2$. 
Hence we may use the covering $I=I_1\cup I_2$
to compute the \v Cech cohomology of $H^1(I, \L_{<{\Phi_j}})$. 
Then, what we need to show is that the morphism 
\begin{align*}
\delta_{\Phi_j}\colon H^0(I_1, \L_{<{\Phi_j}})\oplus H^0(I_2,\L_{<{\Phi_j}})\longrightarrow H^0(I_0, \L_{<{\Phi_j}})
\end{align*}
defined by $\delta_{\Phi_j}(v_1, v_2)=v_{1|I_0}-v_{2|I_0}$ 
is surjective. 

By the induction assumption, 
we
have 
the coarse splittings
\begin{align*}
\xi_{I_\ell}\colon \Gamma(I, \Gr_{\Phi}(\L)) \simeqto \Gamma(I,\L)
\end{align*}
for $\ell=0,1,2$. 
By these splittings, $H^0(I_\ell, \L_{<\Phi_j})$
is identified with 
\begin{align*}
\bigoplus_{\Phi_i<_{I_\ell}\Phi_j}H^0(I_\ell, \Gr_{\Phi_i}\L)
\end{align*}
for each $\ell$
where $\Phi_i<_{I_\ell}\Phi_j$
denotes the condition that $\varphi_i<_{I_\ell}\varphi_j$
for any sections $\varphi_i\in \Gamma(I_\ell, \Phi_i)$
and $\varphi_i\in \Gamma(I_\ell, \Phi_j)$. 
Under this identification, $\delta_{\Phi_j}$ is 
translated into a morphism 
\begin{align}\label{F delta}
\bigoplus_{\Phi_i<_{I_1}\Phi_j}H^0(I_1, \Gr_{\Phi_i}\L) \oplus 
\bigoplus_{\Phi_k<_{I_2}\Phi_j}H^0(I_2, \Gr_{\Phi_k}\L)\to 
\bigoplus_{\Phi_m <_{I_0}\Phi_j} H^0(I_0, \Gr_{\Phi_m}\L).
\end{align}
The component 
\begin{align*}
H^0(I_\ell, \Gr_{\Phi_i}\L)\to H^0(I_0, \Gr_{\Phi_i}\L)\quad (\ell=0,1, \Phi_i <_{I_\ell}\Phi_j)
\end{align*}
of (\ref{F delta}) coincides with $(-1)^{\ell-1}$ times the restriction map,
and the component 
\begin{align*}
H^0(I_\ell, \Gr_{\Phi_i}\L)\to H^0(I_0, \Gr_{\Phi_k}\L)\quad (\ell=0,1, \Phi_i <_{I_\ell}\Phi_k<_{I_\ell}\Phi_j)
\end{align*}
of (\ref{F delta}) is zero.
Since $I\cap \St_{k, m}(\L)$ is at most one point for $k\neq m$ by the assumption of the claim, 
$\Phi_k<_{I_0}\Phi_m$ if and only if 
either $\Phi_k<_{I_1}\Phi_m$ or $\Phi_k<_{I_2}\Phi_m$. 
Therefore, the morphism (\ref{F delta}), and hence $\delta_{\Phi_j}$, are surjective. 

When $I\cap \St_{i,j}$ is exactly one point 
for any $i\neq j$, then 
$H^0(I,\L_{<\Phi_j})=0$ for $j=1,\dots, m$. 
It follows that $H^0(I,\L_{\leqslant\Phi_j})\to H^0(I,\Gr_{\Phi_j}\L)$
is an isomorphism, which proves the uniqueness. 
We also note that this isomorphism preserves 
the Stokes filtration, and so does the inverse. 
\end{proof}
Recall that 
$\L^+$ and $\L^-$
are the local system of free $\k[q^{\pm 1}]$-modules on $T\setminus T_-$
and $T\setminus T_+$, respectively (see Lemma \ref{Glueing L}). 
Then, the Stokes filtrations on $\L$ induces filtrations on $\L^+$ and $\L^-$. 
We may also define the coarse filtrations on $\L^+$ and $\L^-$
in a natural way. 
The coarse grading $\Gr_{\Phi}(\L^+)$ (resp. $\Gr_{\Phi}(\L^-)$)
are local systems on $T\setminus T_-$ (resp. $T\setminus T_+$), 
and there exists a splitting 
\begin{align*}
&\xi_{\bm\theta}^+\colon \Gr_{\Phi}(\L^+)_{\bm\theta}\simeqto \L^+_{\bm\theta}
&(\text{resp. }\xi_{\bm\theta}^-\colon \Gr_{\Phi}(\L^-)_{\bm\theta}\simeqto \L^-_{\bm\theta})
\end{align*}
for any $\bm\theta\in T\setminus T_-$ (resp. $\bm\theta\in T\setminus T_+$)
by the condition (2) in Definition \ref{Def q-Stokes}.
The proof of the following lemma is 
essentially the same as that of
Lemma \ref{key lemma1}:
\begin{lemma}\label{key lemma2}
If an interval $I\subset \ell_a^{\bm\theta}(S^1)$ satisfies the condition that 
\begin{itemize}
\item
$I\cap \St_{i, j}(\L)$ is at most one point for any $i\neq j$,
\item
$I\cap T_{\R_-}=\emptyset$ $($resp. $I\cap T_{\R_+}=\emptyset$$)$, 
\end{itemize}
then there exists
a splitting of the coarse filtration
\begin{align*}
&\xi_I^+\colon \Gamma\(I, \Gr_{\Phi}(\L^+)\)\simeqto \Gamma(I,\L^+),
&(\text{resp. }\xi_I^- \colon \Gamma\(I, \Gr_{\Phi}(\L^-)\)\simeqto \Gamma(I,\L^-))
\end{align*}
Moreover, if $I\cap \St_{i, j}(\L)$ contains exactly one point for any $i\neq j$,
 then such $\xi_{I}^+$ $($resp. $\xi_I^-$$)$ is unique and in that case $\xi_{I}^+$ $($resp. $\xi^-_I$$)$
also preserves the Stokes filtration.\qed
\end{lemma}
We now return to the proof of the main theorem of this section: 
\begin{proof}[Proof of Theorem $\ref{STRICTNESS}$]
Let $\lambda\colon(\L,\L_{\leqslant})\to (\L',\L'_{\leqslant})$
be a morphism in $\bm\St_{\Phi}^q$. 
For $\bm\theta\in T\setminus (T_{\R_+}\cup T_{\R_-})$,
we take $a\in \Z$ and $I\subset \ell^{\bm\theta}_a(S^1)$
so that 
\begin{itemize}
\item
$I\cap \St_{i, j}(\L)$ consists of one point for any $i\neq j$,
\item
$I\cap (T_{\R_+}\cup T_{\R_-})=\emptyset$.
\end{itemize}
Then, by the similar discussion as in the proof of Lemma \ref{m=1},
using Lemma \ref{key lemma1}, 
we obtain the commutative diagram
\begin{align}\label{Diagram}
\begin{split}
\xymatrix{
\Gr_{\Phi}(\L)_{\bm\theta}\ar[d]^{\Gr_{\Phi}(\lambda)_{\bm\theta}}
&\ar[l]\Gamma(I,\Gr_{\Phi}(\L))\ar[r]^{\ \ \xi_{I}}\ar[d]^{\Gamma(I,\Gr_{\Phi}(\lambda))}
&\Gamma(I, \L)\ar[r]\ar[d]^{\Gamma(I,\lambda)}& \L_{\bm\theta}\ar[d]^{\lambda_{\bm\theta}}\\
\Gr_{\Phi}(\L')_{\bm\theta}&\ar[l]\Gamma(I,\Gr_{\Phi}(\L'))\ar[r]^{\ \ \xi'_{I}}&\Gamma(I, \L')\ar[r]& \L'_{\bm\theta}
}
\end{split}
\end{align}
where $\xi_I$ and $\xi'_i$ are the unique splittings in Lemma \ref{key lemma1}, 
$\Gr_{\Phi}(\lambda)$ is the coarse graded morphism associated to $\lambda$. 
Then (\ref{Diagram}) means that 
$\lambda_{\bm\theta}$ is graded by coarse filtration. 
Then, for each graded part $\Gr_{\Phi_j}(\lambda)$ $(j=1,\dots, m)$, 
we can apply the discussion in the proof of Lemma \ref{m=1}
to obtain that $\Gr_{\Phi_j}(\lambda)$ is again graded with respect to
the Stokes filtration, which implies the strictness of $\lambda$ at $\bm\theta$. 

For $\bm\theta\in T_{\R_+}\cup T_{\R_-}$, we apply the Lemma \ref{key lemma2}
to obtain the similar diagram as (\ref{Diagram}) replacing 
$\L$ with $\L^+$ or $\L^-$. Then for each coarse graded part, 
we can take a splitting by the condition (2) in Definition \ref{Def q-Stokes},
whose uniqueness follows from the condition $\bm\theta\in T_{\R_+}\cup T_{\R_-}$.
Then, we obtain the strictness of $\lambda$ at ${\bm\theta}\in T_{\R_+}\cup T_{\R_-}$ 
(details are left to the reader).

The proof of abelianity is then given by the same way as in the case
of usual Stokes structure (see also the proof of Lemma \ref{m=1}),
which is left to the reader. 
\end{proof}


%
%
%
\subsection{Skew symmetric quasi-duality pairing}
Let 
$\L$ be a 
quasi-local system on 
$(T,\Theta)$. 
Let 
$\j\colon T\to \T$
be an involution 
defined by $\j(\theta_u,\theta_v)=(\theta_u+\pi,\theta_v)$.
We set
\[\j^*\L\coloneqq \k[q^{\pm 1}]\otimes_{\k[q^{\pm 1}]}\j^{-1}\L\]
where $\k[q^{\pm 1}]\to \k[q^{\pm 1}]$
is given by $q\mapsto q^{-1}.$
Then 
$\j^*\L$ is again a
quasi-local system
such that 
$(\j^*\L)^\pm=\j^*(\L^\mp)\coloneqq \k[q^{\pm 1}]\otimes_{\k[{q^{\pm 1}]}}\j^{-1}\L^\mp$
where the tensor product is defined in a similar way. 
\begin{definition}
\textit{A skew symmetric quasi-duality pairing}
on $\L$
is a pair
of 
non-degenerate pairings
\begin{align*}
\<\cdot,\cdot\>_\pm\colon \L^\pm\otimes (\j^*\L)^{\pm}\to \k[q^{\pm 1}]_{T\setminus T_{\R_\mp}}
\end{align*}
such that 
\begin{enumerate}
\item
$\<\cdot,\cdot\>_+=\<\cdot,\cdot\>_-$ on the mutual domain, and 
\item
$\j^*\<\cdot,\cdot\>_+=-\<\cdot,\cdot\>_-\circ\mathrm{ex}$
where $\mathrm{ex} $ 
denotes the exchange.  
\end{enumerate}
\end{definition}

Let $\L_{\leqslant}$ be a 
pre-Stokes filtration on $\L$.
It 
induces a pre-Stokes filteration $\j^*\L_{\leqslant}$
on $\j^*\L$
as follows:
\[(\j^*\L)_{\leqslant\varphi}\coloneqq \C1\otimes_\C\j^{-1}\L_{\leqslant \j^*\varphi} 
\subset \j^*\L\]
where $\j^*\varphi(u, v)\coloneqq \varphi(-u,v)$. 
\begin{definition}
A skew symmetric quasi-duality pairing $\<\cdot,\cdot\>_{\pm }$
on $\L$ is compatible 
with $\L_{\leqslant}$ if 
for any $\mathfrak{s}\in\L_{\leqslant\varphi}^+$, 
$\mathfrak{s}'\in (\j^*\L)_{\leqslant\varphi'}^+$, 
$\mathfrak{t}\in\L^-_{\leqslant\psi}$, 
and $\mathfrak{t}'\in (\j^*\L)^-_{\leqslant \psi'}$
we have
\begin{align*}
&\<\mathfrak{s},\mathfrak{s}'\>_+\in \k[q^{\pm 1}]_{\leqslant \varphi+\varphi'},
&\<\mathfrak{t},\mathfrak{t}'\>_-\in\k[q^{\pm 1}]_{\leqslant \psi+\psi'},
\end{align*}
where the filtration on $\k[q^{\pm 1}]$ is
defined as in Example \ref{trivial q}. 
\end{definition}
\begin{example}[Continuation of Example \ref{GAMMA FUN}]
Let $(\L_\Gamma,\L_{\Gamma\leqslant})$
be the Stokes filtered quasi-local system in Example \ref{GAMMA FUN}. 
Let $\j\colon \widetilde{B}\to \widetilde{B}$
be the involution defined by $\j(u, v)=(-u, v)$
on $B^*$ and the continuous extension. 
We then define the 
quasi-duality pairing 
$\<\cdot,\cdot\>_{\Gamma\pm}$
by the multiplication divided by $2\pi\i u$:
\begin{align*}
&\<u^{u^{-1}}v^{u^{-1}}\Gamma(u^{-1}),\j^*\( u^{u^{-1}}v^{u^{-1}}(1-q)\Gamma(u^{-1})\)\>_{\Gamma +}\\
&\coloneqq (2\pi\i u)^{-1} e^{\pi\i u^{-1}}\Gamma(u^{-1})(1-q^{-1})\Gamma(-u^{-1})\\
&=-1,
\end{align*}
where we have used the reflection formula, and the branches of 
the multivalued functions are taken in a suitable way.
The factor $(2\pi\i u)^{-1}$ induces the skew symmetry. 
\end{example}
%
%
%
%
%
\section{De Rham cohomology groups}\label{DR}
\subsection{Geometric setting}
Let $X$ be a compact connected Riemann surface. 
Let $f$ and $g$ be meromorphic functions on $X$. 
Assume that $g$ is not constantly zero.
We set 
$P\coloneqq f^{-1}(\infty)$, 
$E_0\coloneqq g^{-1}(0)\setminus P$,
$E_\infty\coloneqq g^{-1}(\infty)\setminus P$ and
$E\coloneqq E_0\cup E_\infty$. 
We also use the notations $D\coloneqq P\cup E$, 
$Y\coloneqq X\setminus D$ and $U\coloneqq X\setminus P$.
There are inclusions $Y\subset U\subset X$.  

Let $S=\C^2$ denote a complex surface with 
a coordinate $(\tw,\eq)$.
Set $\X\coloneqq S\times X$. 
We also use the notations
$\cal{D}\coloneqq S\times D$, $\cal{P}\coloneqq S\times P$, etc.
Let $\sigma\colon S\to S$ be an automorphism 
on $S$ defined as $\sigma(\tw,\eq)\coloneqq (\tw, \eq-\tw)$.
The induced automorphism on $\X$ is also denoted by $\sigma$.
Set $\euler\coloneqq \tw^2\del_\tw+\tw\eq\partial_\eq$. 
\begin{definition}\label{MFG}
Let $\M$ be a trivial $\O_\X(*\cal{D})$-module of rank one.
Fix a global frame $\e\in H^0(\X,\M)$, which induces 
an isomorphism $\M\simeqto \O_{\X}(*\cal{D}), \e\mapsto 1$.
We consider the following operations on $\M$ associated with the pair $(f,g)$:
\begin{itemize}
\item \textit{A relative connection}
$\nabla\colon \M\to \M\otimes_{\O_\X}\tw^{-1}\Omega_{\X/S}^1$
defined by
\begin{align*}
\nabla(h\e)\coloneqq \e\otimes d_{\X/S}(h)+ h\e\otimes \tw^{-1}(df-\eq g^{-1}dg)
\end{align*}
where $h$ is a local section of $\O_\X(*\cal{D})$ and $d=d_{\X/S}\colon \O_\X(*\cal{D})\to \Omega^1_{\X/S}(*\cal{D})$
denotes the relative differential. 
\item \textit{A differential operator}
\begin{align*}
\nabla_\euler\colon \M\to \M,\quad\nabla_\euler(h\e)=([\euler, h]-hf)\e.
\end{align*}
where $[\euler,\cdot]\colon \O_\X(*\cal{D})\to \O_\X(*\cal{D})$ denote the Lie differential.
\item \textit{A difference operator} 
\begin{align*}
\shift\colon\M\longrightarrow \sigma_{*}\M,\quad \shift(h\e)\coloneqq \sigma^\#(h)g\e
\end{align*}
where $\sigma^\#\colon \O_\X(*\cal{D})\to\sigma_*\O_\X(*\cal{D})$ is defined by $\sigma^\#(h)=h\circ\sigma$.
Note that $\e$ on the right hand side 
denotes the global section of $\sigma_*\M$.
\end{itemize}
In the following, we will assume that $\M=\M(f, g)$ is equipped with 
the operations $\nabla,\nabla_\euler,\shift$ defined above.
\end{definition}
Let $p_\X\colon \X\to X$ denote the projection.
Then $\Omega^1_{\X/S}=\O_\X\otimes_{p_\X^{-1}\O_X}p_\X^{-1}\Omega_X^1$. 
We have the morphisms 
$[\euler,\cdot]\colon \tw^{-1}\Omega_{\X/S}^1\to \tw^{-1}\Omega_{\X/S}^1$
and $\sigma^\#\colon \tw^{-1}\Omega_{\X/S}^1\to \sigma_*\tw^{-1}\Omega_{\X/S}^1$.
We then define 
\[\nabla_\euler^1\colon \M\otimes \tw^{-1}\Omega_{\X/S}^1
\to  \M\otimes \tw^{-1}\Omega_{\X/S}^1\]
by $\nabla_\euler^1\coloneqq \nabla_\euler\otimes \id +\id\otimes [\euler, -]$.
We also define
 \[\shift^1\colon \M\otimes \tw^{-1}\Omega_{\X/S}^1\to \sigma_*(\M\otimes\tw^{-1}\Omega_{\X/S}^1)\]
by $\shift^1\coloneqq \shift\otimes \sigma^\#$.
The following lemma, which can be proved by some easy calculations, shows that 
the operations $\nabla,\nabla_\euler,\shift$ have a kind of integrability.
\begin{lemma}\label{INT} 
$\nabla^1_\euler\circ\nabla=\nabla\circ\nabla_\euler$, 
$\shift^1\circ\nabla=(\sigma_*\nabla)\circ \shift$,
and $\shift\circ\nabla_\euler=(\sigma_*\nabla_\euler)\circ \shift$.\qed
\end{lemma}

\subsection{De Rham cohomology groups}\label{DR2}
We consider 
\begin{align*}
\DR_{\X/S}(\M)\coloneqq\left[ \M\xrightarrow{\nabla} \M\otimes_{\O_\X}\tw^{-1}\Omega_{\X/S}^1\right]
\end{align*}
as a complex placed at degree zero and one.
Let $\pi_\X\colon \X\to S$ denote the projection.
\begin{definition}
Let $k$ be an integer. The $k$-th cohomology group
\begin{align*}
\R^k\pi_{\X*}\DR_{\X/S}(\M)
\end{align*}
of the pushing forward of $\DR_{\X/S}(\M)$ by $\pi_{\X}$
will be denoted by $\H_{\dR}^k$ or $\H_\dR^k(f,g)$.
\end{definition}
Put $S^\circ\coloneqq S\setminus \{\tw=0\}\simeq \C^*\times \C$.
The following theorem will be proved in 
\S \ref{PRO}. 
\begin{theorem}\label{dRT}
Assume that $f$ is not a constant function. 
Then
\begin{enumerate}
\item  $\H^k_{\dR}=0$ for $k\neq 1$, and 
\item  $\H^1_{\dR|S^\circ}$ is a locally free $\O_{S^\circ}$-module.
\end{enumerate}
If we moreover assume that $E=\emptyset$, i.e. if the zeros and the poles of $g$ are contained in the
poles of $f$, then $\H^1_{\dR}$ is locally free over $\O_{S}$.
\end{theorem}
By Lemma \ref{INT}, 
we have the two natural operations
\begin{align*}
&\nabla_\euler\colon \H^1_{\dR|S^\circ}\to \H^1_{\dR|S^\circ}, 
&\shift\colon\H^1_{\dR |S^\circ}\to \sigma_*\H^1_{\dR|S^\circ},
\end{align*}
which satisfy $\shift\circ\nabla_\euler=\sigma_*\nabla_\euler\circ\shift$.
\begin{example}\label{Bessel}
Consider the case 
$X=\P^1$, $f=z+z^{-1}$ and $g=z$ where 
$z$ denotes the coordinate on $\C=X\setminus\{\infty\}$. 
In this case, $P=\{0,\infty\}$, $E=\emptyset$ and hence $D=P$.
We have
\begin{align*}
\H_{\dR}^1\simeq \Cok\left[\O_S[z^{\pm 1}]\xrightarrow{\nabla}\tw^{-1}\O_S[z^{\pm 1}]dz\right]
\end{align*}
with $\nabla(h z^n)=h(n z^{n-1}+\tw^{-1}(z^n-z^{n-2}-\eq z^{n-1}))dz$
for $h\in\O_S, n\in \Z$.
Let $e_n$ denote the class represented by $\tw^{-1}z^{n}d z$ for $n\in\Z$. 
We have
$\H_\dR^1= \bigoplus_{i=0,1}\O_Se_{n+i}$
and $e_n=(\eq-n\tw)e_{n-1}+e_{n-2}$ for each $n\in \Z$. 
The action of $(\nabla_\euler,\shift)$  is given as follows:
\begin{align*}
&\nabla_\euler(e_n)=e_{n+1}+e_{n-1}-\tw e_n,
&\shift(e_n)=e_{n+1}.
\end{align*}
\end{example}
\begin{example}\label{EXG}
Consider the case $X=\P^1$, and $f=g=z$.
In this case, we have $P=\{\infty\}$, $E=\{0\}$, and $D=\{0,\infty\}$.
We have
\begin{align*}
\H_{\dR}^1\simeq \Cok\left[\O_{S}[z^{\pm 1}]\xrightarrow{\nabla}\tw^{-1}\O_{S}[z^{\pm 1}]dz\right]
\end{align*}
with $\nabla(h z^n)=h(n z^{n-1}+\tw^{-1}(z^n-\eq z^{n-1}))dz$
for $h\in\O_S, n\in \Z$.
Let $e_n$ denote the class represented by $\tw^{-1}z^{n}d z$. 
We have the relation  \[e_{n}=(\eq-n\tw)e_{n-1}.\] 
For each point $s\in S^\circ$, there exists 
an open neighborhood $\nb(s)\subset S^\circ$ of $s$
and $n_s\in\Z$ such that 
$\H_{\dR|\nb(s)}^1=\O_{\nb(s)}e_{n_s}$, 
which implies that $\H_{\dR|S^\circ}^1$ is locally free.
The action of $(\nabla_\euler,\shift)$  is given as follows:
\begin{align*}
&\nabla_\euler(e_n)=e_{n+1}-\tw e_n,
&\shift(e_n)=e_{n+1}.
\end{align*}
\end{example}
\begin{remark}
The module $\H_{\dR}^1$ in Example \ref{EXG} is \textit{not} 
locally finitely generated over $\O_S$. 
Indeed, for any point $s_0=(0,\mu_0)\in S$ with $\mu_0\in\C$,
any open neighborhood $\nb(s_0)$ of $s_0$,
and any $n_0\in\Z$, there are infinitely many $n<n_0$
such that \[\{(\tw,\eq)\in S\mid \eq-n\tw=0\}\cap\nb(s_0)\neq \emptyset,\]
which implies that no $e_{n_0}$ can generate $ \H_{\dR}^1$ on $\nb(s_0)$.
This proof also indicates that 
$\H^1_{\dR}(*\{\tw\eq=0\})$ is \textit{not} finitely generated over $\O_S(*\{\tw\eq=0\})$.
\end{remark}
\subsection{Submodules}\label{S lattice}
Take a submodule
$\M_{\bm{0}}\coloneqq \O_\X(*\cal{P})\e\subset \M.$
For $(a, b)\in\Z^2$, 
set 
\begin{align*}
\M_{a, b}\coloneqq \M_{\bm{0}}(a\cal{E}_0+b\cal{E}_\infty)=\M_{\bm{0}}\otimes_{\O_\X}\O_\X(a\cal{E}_0+b\cal{E}_\infty)
\end{align*}
which is also a submodule of $\M$
(Recall that $\cal{E}_0=S\times E_0$ and $\cal{E}_\infty=S\times E_\infty$). 
The operations $\nabla,\nabla_\euler,\shift$
induce the following morphisms:
\begin{align*}
&\nabla\colon\M_{a, b}\longrightarrow \M_{a+1,b+1}\otimes\tw^{-1}\Omega_{\X/S}^1,\\
&\nabla_\euler\colon \M_{a, b}\longrightarrow \M_{a, b},\\
&\shift\colon \M_{a, b}\longrightarrow \sigma_*\M_{a-1,b+1}. 
\end{align*}

\begin{lemma}\label{dR}
For $(a, b)\in \Z^2$ and $n\in \Z$, 
set 
\[\M_{a, b}^n\coloneqq \O_\X(n(f)_\infty+a\cal{E}_0+b\cal{E}_\infty)\bm{e}\subset \M_{a, b}.\]
Then,  the inclusion of the complexes
\begin{align*}
\begin{split}
\xymatrix{
\M_{a, b}^n\ar[d]\ar[r]&\M_{a+1,b+1}^{n+1}\otimes \tw^{-1}\Omega_{\X/S}^1(\cal{P})\ar[d]\\
\M_{a,b}\ar[r]^{\nabla\ \ \ \ \ \ \ \ }&\M_{a+1,b+1}\otimes \tw^{-1}\Omega_{\X/S}^1
}
\end{split}
\end{align*}
is a quasi-isomorphism.
\end{lemma}
\begin{proof}
We need to check the claim around the points in $\cal{P}=S\times P$. 
For a point $p\in P$, 
take a local coordinate $(V, z)$ centered at $p$
such that we have $f_{|V}=z^{-m_p}$ for some $m_p\in \Z_{>0}$. 
Note that we have $g^{-1}dg_{|V}=h(z)z^{-1}dz$ 
for some holomorphic function $h$ on $V$. 
Then, on $\cal{V}=S\times V$, by $\nabla$,
$z^\ell\bm{e}\in \M_{a, b}$ $(\ell\in \Z)$ maps to 
\begin{align*}
\bm{e}\otimes\ell z^{\ell-1}dz+\bm{e}\otimes 
\tw^{-1}\(-(m_p+1)z^{\ell-1-m_p}dz-\eq h(z)z^{\ell-1}dz\).
\end{align*}
Since $m_p>0$, the morphism
\begin{align*}
\M^{n+1}_{a+1, b+1}\otimes \tw^{-1}\Omega_{\X/S}^1(\cal{P})
\longrightarrow 
\frac{\M_{a+1,b+1}\otimes \tw^{-1}\Omega^1_{\X/S}}{\nabla (\M_{a,b})}
\end{align*}
is surjective. This implies the lemma. 
\end{proof}
\begin{definition}
For $a, b\in\Z$, 
let $\DR_{\X/S}(\M_{a, b})$ denote the complex 
\begin{align*}
\M_{a, b}\xrightarrow{\nabla}\M_{a+1, b+1}\otimes \tw^{-1}\Omega_{\X/S}^1
\end{align*}
placed at degree zero and one. 
The 
$k$-th cohomology group of 
the derived push forward
$\R\pi_{\X*}\DR_{\X/S}(\M_{a, b})$
is denoted by $\H_{\dR,a, b}^k$ or $\H_{\dR, a, b}^k(f, g)$. 
\end{definition}
\begin{theorem}\label{LF}
Assume that $f$ is not a constant function.
Then, for $(a, b)\in \Z^2$, 
\begin{enumerate}
\item $\H^k_{\dR, a, b}=0$ if $k\neq 1$, and 
\item $\H^1_{\dR, a, b}$ is a locally free $\O_S$-module.
\end{enumerate}
\end{theorem}
\begin{proof}
For $a, b, n, \ell\in \Z$,
we set 
\begin{align*}
V_{a, b}^{n, \ell}&\coloneqq H^\ell\(X,\O_X(n(f)_\infty+aE_0+bE_\infty)\),\\
W_{a, b}^{n, \ell}&\coloneqq H^\ell\(X,\Omega_X^1(n(f)_\infty+aE_0+bE_\infty)\).
\end{align*}
Since $f$ is not constant, given $(a, b)\in \Z^2$, 
there exists $m=m(a, b)>0$ such that 
\[V_{a, b}^{n,\ell}=W_{a, b}^{n,\ell}=0\]
for $n>m$ and $\ell\neq 0$. 
Then, by Lemma \ref{dR},
$\R\pi_{\X*}\DR_{\X/S}(\M_{a, b})$ 
is quasi-isomorphic to the following complex
for $n>m$:
\begin{align*}
V_{a, b}^{n,0}\otimes \O_S\xrightarrow{d+\tw^{-1}(df-\eq g^{-1}dg)}
W_{a, b}^{n,0}\otimes \tw^{-1}\O_S.
\end{align*}
For any complex numbers $(\tw,\eq)\in \C^2$, 
the linear map 
\begin{align*}
\tw d+df-\eq g^{-1}dg\colon V_{a, b}^{n,0} \to W_{a, b}^{n,0}
\end{align*}
is injective, which implies the theorem. 
\end{proof}
The operators $\nabla_\euler$ and $\shift$ on $\M_{a, b}$ induce
the operators
\begin{align*}
&\nabla_\euler\colon \H_{\dR, a, b}^1\longrightarrow \H^1_{\dR, a, b},
&\shift\colon \H_{\dR, a, b}^1\longrightarrow \sigma_*\H^1_{\dR, a-1, b+1},
\end{align*}
which satisfies $\shift\circ\nabla_\euler=\sigma_*\nabla_\euler\circ\shift$.
\begin{example}[Continuation of Example \ref{EXG}]\label{Ga}
In the case of Example \ref{EXG}, 
since $E_\infty=\emptyset$, we may omit the subscript $b$.
We then obtain
$\H^1_{\dR,a}=\O_Se_{-a-1}$.
\end{example}

\subsection{Proof of Theorem \ref{dRT}}\label{PRO}
For $a, n\in \Z$, put
$H_{a,n}\coloneqq \{(\tw,\eq)\in S\mid a\tw+n\eq=0\}$.
\begin{lemma}\label{SUP}
We have the following exact sequences of $\O_S$-modules:
\begin{align*}
&0\longrightarrow \H_{\dR, a-1,b}^1\longrightarrow \H^1_{\dR, a, b}
\longrightarrow \bigoplus_{e\in E_0}\O_{H_{a,n_e}}
\longrightarrow 0,\\
&0\longrightarrow \H_{\dR, a, b-1}^1\longrightarrow
\H^1_{\dR, a, b}\longrightarrow  \bigoplus_{e\in E_\infty}\O_{H_{b,n_e}}
\longrightarrow 0.
\end{align*}
Here, $n_e\in \Z$ $(e\in E)$ denote the order of $g$ at $e$ 
$($$n_e>0$ for $e\in E_0$$)$. 
\end{lemma}
\begin{proof}
For $a, b\in \Z$, 
put $\M_{\overline{a}, b}\coloneqq \mathrm{Cok}[\M_{a-1,b}\hookrightarrow \M_{a,b}] $,
$\M'_{a, b}\coloneqq\M_{a, b}\otimes  \tw^{-1}\Omega_{\X/S}^1$, 
and $\M'_{\overline{a},b}\coloneqq\M_{\overline{a}, b}\otimes  \tw^{-1}\Omega_{\X/S}^1$.
Then we have
\[\M_{\overline{a},b}\simeq \O_{\cal{E}_0}
=\bigoplus_{e\in E_0}\O_{S\times\{e\}}\]
and 
\begin{align*}
\xymatrix{
0\ar[r]& \M_{a-1, b}\ar[d]^\nabla\ar[r]&\M_{a, b}\ar[r] \ar[d]^\nabla
& \M_{\overline{a}, b}\ar[d]^{\overline{\nabla}}\ar[r]&0\\
0\ar[r]&\M'_{a, b+1}\ar[r] &\M'_{a+1, b+1}\ar[r]&  \M'_{\overline{a+1}, b}\ar[r]&0
}\end{align*}
whose rows are exact. 
We shall describe $\overline{\nabla}$. 
Take $e\in E_0$. 
We may take  
a local coordinate $(V, z)$ centered at $e$
such that $g_{|V}=z^{n_e}$ and $f_{|V}$ is holomorphic.  
Then, if we set $\cal{V}=S\times V$, we have
$\M_{\overline{a}, b|\cal{V}}=\O_{S\times\{e\}}[z^{-a}]$, 
$\M'_{\overline{a+1}, b|\cal{V}}=\O_{S\times\{e\}}[z^{-a-1}]\otimes \tw^{-1}dz$
and 
\begin{align*}
\overline{\nabla}([z^{-a}])=-\tw^{-1}(a\tw+n_e\eq )[z^{-a-1}]dz
\end{align*}
where $[z^{-a}]$ and $[z^{-a-1}]$ are the classes represented by 
$z^{-a}\bm{e}$ and $z^{-a-1}{\bm{e}}$, respectively. 
By this expression, taking the pushing forward, we obtain the first exact sequence by Theorem \ref{LF}.
We can also obtain the second exact sequence in a similar way.  
\end{proof}
\begin{proof}[Proof of Theorem \ref{dRT}]
If $E=\emptyset$, the statement directly follows from Theorem \ref{LF}.
We also note that (1) is straightforward.  
It remains to show (2) when $E\neq \emptyset$.
For each $(a, b)\in \Z^2$,
we have the injective morphisms $\H_{\dR, a, b}^1\to\H^1_{\dR}$
which are compatible with the inclusions 
in Lemma \ref{SUP}.
In this sense, we have
\begin{align*}
\H^1_{\dR}=\lim_{a, b\to\infty}\H^1_{\dR, a, b}.
\end{align*}
For each point $s\in S^\circ$, by Lemma \ref{SUP},
there exists an open neighborhood $\nb(s)$
where the limit terminate at a finite term. 
In other words, $\H_{\dR|\nb(s)}^1\simeq \H^1_{\dR, a, b|\nb(s)}$
for some $(a, b)$.
By Theorem \ref{LF}, this implies the theorem.
\end{proof}

\subsection{Some pairings}\label{j}
Let $\j\colon S\to S$ be the involution defined by $\j(\tw,\eq)\coloneqq (-\tw,\eq)$. 
The induced involution on $\X$ is also denoted by $\j$. 
On the pull back $\j^*\M=\O_\X\otimes_{\j^{-1}\O_\X} \j^{-1}\M$, 
we naturally obtain the operators
\begin{align*}
&\j^*\nabla\colon \j^*\M\longrightarrow \j^*\M\otimes \tw^{-1}\Omega_{\X/S}^1,\\
&\quad\j^*\nabla(h\j^*\e)\coloneqq \j^*\e\otimes d_{\X/S}(h)-h\j^*\e\otimes \tw^{-1}(df-\eq g^{-1}dg),\\
&\j^*\nabla_\euler\colon \j^*\M\to \j^*\M,\quad 
\j^*\nabla_\euler(h\j^*\e)\coloneqq([\j_*\euler, h]+hf)\j^*\e,\\
&\j^*\shift\colon \j^*\M\to \j^*\sigma_*\M=(\sigma^{-1})_*\j^*\M,
\quad \j^*\shift(h\j^*\e)\coloneqq(\sigma^{-1})^\#(h)g^{-1}\j^*\e
\end{align*}
where $h\in\O_\X(*\cal{D})$, $\j^*\e=1\otimes \j^{-1}\e$, and 
$\j_*\euler=-\euler$.

Define a pairing 
\begin{align*}
\<\cdot,\cdot\>\colon \M\otimes_{\O_\X(*\cal{D})}\j^*\M\to\O_\X(*\cal{D})
\end{align*}
by $\<\e,\j^*\e\>=1$ and the $\O_\X(*\cal{D})$-linearity.

\begin{lemma}\label{DUA}
For $\bm{v}\in \M$ and $\bm{w}\in\j^*\M$, the equalities
\begin{align*}
d_{\X/S}\<\bm{v},\bm{w}\>&=\<\nabla \bm{v},\bm{w}\>+\<\bm{v},\j^*\nabla\bm{w}\>,\\
[\euler, \<\bm{v},\bm{w}\>]&=\<\nabla_\euler\bm{v},\bm{w}\>+\<\bm{v},\j^*\nabla_\euler\bm{w}\>,
\end{align*}
hold. For $\bm{v}\in \M$ and $\bm{w}\in \sigma_*\j^*\M$, the equality
\begin{align*}
\sigma_*\<\shift\bm{v},\bm{w}\>=\sigma^\#(\<\bm{v},\sigma_*\j^*\shift\bm{w}\>)
\end{align*}
holds. Note that $\sigma_*\j^*\shift\colon \sigma_*\j^*\M\to\j^*\M$.
\qed
\end{lemma}

Consider  
\begin{align*}
\DR_{\X/S}(\j^*\M)\coloneqq
 \left[ \j^*\M\xrightarrow{\j^*\nabla} \j^*\M\otimes \tw^{-1}\Omega_{\X/S}^1\right]
\end{align*}
as a complex placed at degree zero and one.
We then naturally obtain 
\[\R^k\pi_{\X*}\DR_{\X/S}(\j^*\M)\simeq \j^*\H^k_{\dR}.\]
This isomorphism 
is compatible with the actions of $\j^*\nabla_\euler$ and
$\j^*\shift$
on both sides.
Similar facts hold for $\j^*\M_{a, b}$ and 
$\j^*\H^1_{\dR, a, b}$.

\subsection{Pairings on cohomology groups}\label{pairing}
Mimicking the discussion in \cite{Yu}, which goes back to the work of Deligne \cite{Del2}, we shall construct 
a morphism
\begin{align*}
\<\cdot,\cdot\>_\dR\colon\H_{\dR, a, b}^1\otimes\j^*\H^1_{\dR, c, d}\longrightarrow \tw^{-1}\O_S
\end{align*}
for $a, b, c, d\in \Z$ with $a+c=b+d=-1$.
Use the notation in Lemma \ref{dR}.
Set 
\begin{align*}
&\M_{a, b}^{n,0}\coloneqq \M_{a,b}^n,
&\M_{a, b}^{n,1}\coloneqq \M_{a+1,b+1}^{n+1}\otimes \tw^{-1}\Omega_{\X/S}^1(\cal{P}).
\end{align*}
By Lemma \ref{dR}, 
$\DR_{\X/S}(\M_{a, b})$ is
quasi-isomorphic to
\[\M_{a, b}^{n,\bullet}\coloneqq [\M_{a, b}^{n,0}\xrightarrow{\nabla}\M_{a, b}^{n,1}].\]
Similarly,
$\DR_{\X/S}(\j^*\M)$ is quasi-isomorphic to 
\[\scr{N}_{c, d}^{m,\bullet}\coloneqq[\scr{N}_{c, d}^{m,0}\xrightarrow{\j^*\nabla}\scr{N}_{c, d}^{m,1}]\]
where
\begin{align*}
&\N_{c, d}^{m,0}\coloneqq \j^*\M_{c, d}^{m}(-\cal{P}),
&\N_{c, d}^{m,1}\coloneqq \j^*\M_{c+1,d+1}^{m+1}\otimes\tw^{-1}\Omega_{\X/S}^1.
\end{align*}
Put 
$\DR_{\X/S}(\O_\X)\coloneqq [\O_\X\xrightarrow{d_{\X/S}}\tw^{-1}\Omega_{\X/S}^1]$
placed at degree zero and one.
We define the morphism of complexes
\begin{align*}
\<\cdot,\cdot\>_{\dR}^\bullet\colon\M_{a, b}^{n,\bullet}\otimes \scr{N}_{c, d}^{m,\bullet}\to \DR_{\X/S}(\O_\X)
\end{align*}
for $a+c=b+d=n+m=-1$ as follows:
In degree zero, 
we have 
\begin{align*}
\<\cdot,\cdot\>_\dR^0\colon \M_{a, b}^{n,0}\otimes\N_{c, d}^{m,0} \xrightarrow{\<\cdot,\cdot\>}
\O_{\X}(-(f)_\infty-\cal{P})\hookrightarrow \O_\X,
\end{align*}
where $\<\cdot,\cdot\>$ is defined in \S \ref{j}.
In degree 1, we may also define 
\begin{align*}
\<\cdot,\cdot\>^1_\dR\colon\bigoplus_{i+j=1}\M_{a, b}^{n,i}\otimes \scr{N}_{c, d}^{m,j}
\longrightarrow\tw^{-1}\Omega_{\X/S}^1
\end{align*}
extending $\<\cdot,\cdot\>$ linearly with respect to the tensor product of
sections in $\tw^{-1}\Omega_{\X/S}^1$.
We set $\<\cdot,\cdot\>_\dR^2=0$ in degree 2.
Then, by Lemma \ref{DUA}, $\<\cdot,\cdot\>_{\dR}^\bullet$
is a morphism of complexes.
Taking the pushing forward with respect to $\pi_\X$,
we obtain $\<\cdot,\cdot\>_\dR$.
Here, we have fixed the isomorphism  
\begin{align*}
\R^2\pi_{\X*}\DR_{\X/S}(\O_\X)=H^1(X,\Omega_X^1)\otimes \tw^{-1}\O_S\simeq \tw^{-1}\O_S
\end{align*}
by $(2\pi\i)^{-1}\int_X\colon H^1(X,\Omega_X^1)\simeqto \C$. 
\begin{remark}
The pairing $\<\cdot,\cdot\>_{\dR}$
does not depend on the choice of the integer $n$
fixed in the construction.
Moreover, they are compatible with the inclusions in
Lemma \ref{SUP} in the following sense:
For $a',b',c',d'$ with $a'<a$, $b'<b$ and $a'+c'=b'+d'=-1$, 
take $\bm{\upsilon}\in \H^1_{\dR, a', b'}$
and $\bm{\omega}\in\j^* \H^1_{\dR, c, d}$. 
Then we have 
\begin{align*}
\<\imath(\bm{\upsilon}),\bm{\omega}\>_{\dR}=\<\bm{\upsilon},\imath(\bm{\omega})\>_\dR
\end{align*}
where $\imath$ denotes the injection (i.e. the composition of inclusions in Lemma \ref{SUP}).
\end{remark} 

By Lemma \ref{DUA}, we have the following:
\begin{corollary}
For $\bm{\upsilon}\in \H^1_{\dR, a, b}$ and $\bm{\omega}\in \j^*\H_{\dR,c,d}^1$
with $a+c=b+d=-1$, the following equality holds:
\begin{align*}
\left[\euler,\<\bm{\upsilon},\bm{\omega}\>_\dR\right]
=\<\nabla_\euler\bm{\upsilon},\bm{\omega}\>_\dR+\<\bm{\upsilon},\j^*\nabla_\euler\bm{\omega}\>_\dR.
\end{align*} 
For $\bm{\upsilon}\in \H^1_{\dR, a, b}$ and $\bm{\omega}\in \sigma_*\j^*\H_{\dR, c+1,d-1}^1$
with $a+c=b+d=-1$,
we have
\begin{align*}
\sigma_*\<\shift\bm{\upsilon},\bm{\omega}\>_\dR=
\sigma^\#(\<\bm{\upsilon},\sigma_*\j^*\shift\bm{\omega} \>_\dR)
\end{align*}
Note that $\shift \bm{\upsilon}\in \sigma_*\H_{\dR,a-1, b+1}^1$ and
$\sigma_*\j^*\shift\bm{\omega}\in \j^*\H_{\dR, c, d}^1$.\qed
\end{corollary}

\subsection{Perfectness of the pairings}
In this subsection, we prove the following:
\begin{theorem}[{c.f. \cite[Theorem 2.1]{Yu} }]\label{Perfect theorem}
The pairing $\<\cdot,\cdot\>_\dR$ is perfect, i.e.
the induced morphism
\begin{align*}
\cal{I}_\dR\colon \H_{\dR, a, b}^1\longrightarrow
\hom_{\O_S}(\j^*\H^1_{\dR, c, d},\tw^{-1}\O_S),
\quad \bm{\upsilon}
\mapsto
(\bm{\omega}\mapsto\<\bm{\upsilon},\bm{\omega}\>_\dR)
\end{align*}
is an isomorphism for $a, b, c, d\in \Z$ with $a+c=b+d=-1$.
\end{theorem}
Let us introduce the following notations:
\begin{align*}
&(-)^\vee\coloneqq \hom_{\O_\X}(-,\O_\X),
&(-)^\wedge\coloneqq \hom_{\O_\X}(-,\tw^{-1}\Omega_{\X/S}^1).
\end{align*}
We have 
\begin{align*}
(\N_{c, d}^{m,0})^\wedge&=(\j^*\M_{c, d}^{m})^\vee\otimes \tw^{-1}\Omega^1_{\X/S}(\cal{P}),\\
(\N_{c, d}^{m,1})^\wedge&=(\j^*\M_{c+1,d+1}^{m+1})^\vee.
\end{align*}
Then, the 
isomorphisms
\begin{align*}
\cal{J}_\dR^i\colon \M_{a, b}^{n,1+i}\longrightarrow (\N_{c, d}^{m, -i})^\wedge, 
\quad \bm{v}\mapsto (\bm{w}\mapsto \<\bm{v},\bm{w}\>_\dR^1),
\quad(i=0,-1)
\end{align*}
induced from $\<\cdot,\cdot\>_\dR^1$
are identified with $\bm{v}\mapsto \<\bm{v},\cdot\> \in (\j^*\M_{c+1,d+1}^{m+1})^\vee$
for $i=-1$,
and 
$\bm{v}\otimes \omega \mapsto\<\bm{v},\cdot\>
\otimes\omega\in (\j^*\M_{c, d}^{-m})^\vee\otimes \tw^{-1}\Omega^1_{\X/S}(\cal{P})$
for $i=0$, respectively. 
Define
\begin{align*}
(\j^*\nabla)^\vee\colon (\j^*\M_{c+1,d+1}^{m+1})^\vee\longrightarrow 
(\j^*\M_{c, d}^{m})^\vee\otimes \tw^{-1}\Omega^1_{\X/S}(\cal{P})
\end{align*}
by the following formula:
\begin{align*}
d_{\X/S}{\<}\bm{\psi},\bm{w}{\>}_{\rm ev}=
{\<}(\j^*\nabla)^\vee(\bm{\psi}),\bm{w}{\>}_{\rm ev}
+{\<}\bm{\psi},{\j^*\nabla}\bm{w}{\>}_{\rm ev}
\end{align*}
where $\bm{\psi}\in  (\j^*\M_{c+1,d+1}^{m+1})^\vee$,
$\bm{w}\in \j^*\M_{c+1,d+1}^{m+1}$,
and ${\<}\cdot,\cdot{\>}_{\rm ev}$ denotes the evaluation.
We regard $(\scr{N}_{c, d}^{m,-\bullet})^\wedge$
as a complex with differential $-(\j^*\nabla)^\vee$. 
Then $\cal{J}_\dR^\bullet$ induces an isomorphism of complexes
$\M_{a, b}^{m, \bullet}[1]\simeqto (\N_{c, d}^{m,-\bullet})^\wedge$.
Taking the pushing forward functor, 
we obtain
\begin{align*}
\cal{J}_\dR\colon \H^1_{\dR, a, b}\simeqto \R^0\pi_{\X*}(\scr{N}_{c, d}^{m,-\bullet})^\wedge.
\end{align*}
\begin{lemma}\label{Serre}
There is a natural isomorphism
\begin{align*}
\eta\colon\R^0\pi_{\X*}(\scr{N}_{c, d}^{m,-\bullet})^\wedge\simeqto \hom_{\O_S}\(\j^*\H_{\dR, c, d}^1,\tw^{-1}\O_S\).
\end{align*}
\end{lemma}
\begin{proof}
By Lemma \ref{dR}, we may assume that 
$m$ is sufficiently small so that 
\begin{align*}
\R^k\pi_{\X*}(\N_{c, d}^{m, i})^\wedge=\R^\ell\pi_{\X*}(\N_{c, d}^{m, j})=0,
\end{align*}
for $j=0,1$, $k\neq 0$, and $\ell\neq 1$.
By the Grothendieck-Serre duality, 
there are natural isomorphisms
\begin{align*}
\eta_j\colon\pi_{\X*}(\scr{N}_{c, d}^{m, j})^\wedge \simeqto \hom_{\O_S}\(\R^1\pi_{\X*}(\scr{N}_{c, d}^{m ,j}),\tw^{-1}\O_S\)
\end{align*}
for $j=0,1$.
Using the notation
\begin{align*}
(-)^\vee_S\coloneqq \hom_{\O_S}(-,\tw^{-1}\O_S), 
\end{align*}
we obtain the following diagram:
\begin{align*}
\xymatrix{
0\ar[r]&\pi_{\X*}(\N_{c, d}^{m, 1})^\wedge\ar[r]\ar[d]^{\eta_1}& 
\pi_{\X*}(\N_{c, d}^{m,0})^\wedge\ar[r]\ar[d]^{\eta_0}&
\R^0\pi_{\X*}(\scr{N}_{c, d}^{m,-\bullet})^\wedge\ar[r]&0\\
0\ar[r]&(\R^1\pi_{\X*}\N_{c, d}^{m,1})^\vee_S\ar[r]&(\R^1\pi_{\X*}\N_{c, d}^{m,0})^\vee_S
\ar[r]&(\j^*\H_{\dR, c, d}^1)^\vee_S\ar[r]&0.}
\end{align*}
Since the square in this diagram commutes, we obtain $\eta$.

Commutativity of the square:
The Grothendieck-Serre duality isomorphisms $\eta_j$
are induced from the following morphisms:
\begin{align*}
\R^1\pi_{\X*} \N\otimes \pi_{\X*}\hom(\N,\tw^{-1}\Omega_{\X/S}^1)
&\xrightarrow{\<\cdot,\cdot\>_{\rm ev}}
\R^1\pi_{\X*}(\tw^{-1}\Omega_{\X/S}^1)\\
&\xrightarrow{\frac{1}{2\pi\i}\int_X}
\tw^{-1}\O_S
\end{align*}
where $\N=\N_{c, d}^{m, j}$, and $j=0,1$.
Then the commutativity 
follows from the equality
\begin{align*}
\frac{1}{2\pi\i}\int_X\<\j^*\nabla\bm{\upsilon},\bm{\omega}\>_{\rm ev}
=\frac{1}{2\pi\i}\int_X\<\bm{\upsilon},-(\j^*\nabla)^\vee\bm{\omega}\>_{\rm ev}
\end{align*}
for $\bm{\upsilon}\in\R^1\pi_{\X*} \N_{c, d}^{m,0}$, and 
$\bm{\omega}\in \pi_{\X*}\hom(\N_{c, d}^{m,1},\tw^{-1}\Omega_{\X/S}^1)$.
This equality follows from the fact that the difference
of this equality is given 
by
\begin{align*}
\frac{1}{2\pi\i}\int_Xd_{\X/S}\<\bm{\upsilon},\bm{\omega}\>_{\rm ev}=0
\end{align*}
where $\bm{\omega}$ is considered as the section of
$\pi_{\X*}\hom(\N_{c, d}^{m,0},\O_\X)$.
\end{proof}
Composing $\eta$ in Lemma \ref{Serre} and
the morphism $\cal{J}_\dR$ constructed above, we obtain
\begin{align*}
\eta\circ\cal{J}_\dR\colon \H^1_{\dR, a, b}\simeqto \hom_{\O_S}\(\j^*\H_{\dR, c, d}^1,\tw^{-1}\O_S\).
\end{align*}
Theorem \ref{Perfect theorem}
follows from the following:
\begin{lemma}
$\cal{I}_\dR=\eta\circ\cal{J}_\dR$.
\end{lemma}
\begin{proof}
We may assume that 
$n$ is sufficiently large and hence $m$ is sufficiently small
so that we have 
\begin{align*}
\H_{\dR, a, b}^1&\simeq \Cok\left[\pi_{\X*}\M_{a, b}^{n,0}\xrightarrow{\nabla}\pi_{\X*}\M_{a, b}^{n,1}\right],\\
\j^*\H_{\dR, c, d}^1&\simeq \mathrm{Ker}\left[\R^1\pi_{\X*}\N_{c, d}^{m,0}
\xrightarrow{\j^*\nabla}\R^1\pi_{\X*}\N_{c, d}^{m, 1}\right].
\end{align*}
Hence a section $\bm{\upsilon}\in \H_{\dR, a, b}^1$
can be represented by a section $\upsilon$ of 
\[\pi_{\X*}\M_{a, b}^{n,1}=\pi_{\X*}\M_{a+1,b+1}^{n+1}\otimes\Omega_{\X/S}^1(\cal{P})\]
and a section $\bm{\omega}\in \j^*\H_{\dR, c, d}^1$
can be seen as a section $\omega$ of 
\[\R^1\pi_{\X*}\N_{c, d}^{m,0}=\R^1\pi_{\X*}(\j^*\M_{c, d}^{m}(-\cal{P})).\]
Using this expression,
both $\<\cal{I}_\dR(\bm{v}), \bm{\omega}\>_{\rm ev}$
and $\<\eta\circ\cal{J}_\dR(\bm{v}),\bm{\omega}\>_{\rm ev}$
are given by
\begin{align*}
\frac{1}{2\pi\i}\int_X\<\upsilon,\omega\>
\end{align*}
which implies the lemma.
\end{proof}

\subsection{Explicit description and examples of pairings}
Take $\bm{\upsilon}\in \H^1_{\dR,a, b}$ and
$\bm{\omega}\in \j^*\H^1_{\dR, c, d}$ with 
$a+c=b+d=-1$. 
As we have seen before, 
we can take sufficiently large $n$ so that
$\bm{\upsilon}$ 
can be represented by 
a section $\upsilon$
of $\pi_{\X*}(\M^{n+1}_{a+1, b+1}\otimes \tw^{-1}\Omega_{\X/S}^1(\cal{P}))$.
Similarly, if one take $m'$ large enough, 
then $\bm{\omega}$
is also represented by 
a section $\omega'$ of 
$\pi_{\X*}(\j^*\M_{c+1, d+1}^{m'}\otimes\tw^{-1} \Omega_{\X/S}^1)$.

For each point $p\in P$, 
there exist 
an open disk $V_p$ centered at $p$ 
and a section
$\alpha_p\in \Gamma(\cal{V}_p,\j^*\M_{c, d})$ on $\cal{V}_p=S\times V_p$
such that 
\begin{align*}
\omega'_{|\cal{V}_p}-\j^*\nabla \alpha_p\in \(\j^*\M_{c+1, d+1}^{m+1}\otimes \tw^{-1}\Omega_{\X/S}^1\)_{\big|\cal{V}_p}.
\end{align*}
where $m=-n-1$ (See the proof of Lemma \ref{dR}).
\begin{proposition}[{c.f. \cite[p. 112]{Del2}}]
Using the notations above,
we have
\begin{align*}
\<\bm{\upsilon},\bm{\omega}\>_\dR=\sum_{p\in P}\mathrm{Res}_{S\times\{p\}}\<\upsilon,\alpha_p\>.
\end{align*}
\end{proposition}
\begin{proof}
Take a characteristic function 
$\psi_p$ of $(V_p, p)$ and put 
\[\alpha\coloneqq \sum_{p\in P}\psi_p\alpha_p.\]
Then, 
$\omega'-(\j^*\nabla+\overline{\del}_X)\alpha$
also represents $\bm{\omega}$
in the relative Dolbeault resolution
of the complex $\N_{c, d}^{m,\bullet}$,
where $\overline{\del}_X$ denotes the $\overline{\del}$-operator 
along the $X$-direction.

We have
\begin{align*}
\<\bm{\upsilon},\bm{\bm{\omega}}\>_\dR&=\frac{1}{2\pi\i}\int_X\<\upsilon,\omega'-(\j^*\nabla+\overline{\del}_X)\alpha\>\\
&=\frac{1}{2\pi\i}\sum_{p\in P}\int_X-\overline{\del}_X\psi_p\<\upsilon,\alpha_p\>\\
&=\sum_{p\in P}\mathrm{Res}_{S\times\{p\}}\<\upsilon,\alpha_p\>.
\end{align*}
Hence we obtain the proposition.
\end{proof}
\begin{example}[Continuation of Example \ref{Bessel}]\label{BDual}
In the case of Example \ref{Bessel}, we may omit the subscripts $a, b, c, d$
since $E=\emptyset$. 
We can take $n$ to be $0$
to obtain
\begin{align*}
\H^1_{\dR}\simeq \Cok\left[\O_S\xrightarrow{\tw^{-1}(1-z^{-2}-\eq z^{-1})dz}
\bigoplus_{-2\leq k\leq 0}\tw^{-1}\O_Sz^k dz\right].
\end{align*} 
We have put $e_k\coloneqq [\tw^{-1}z^k d z]$ $(k\in \Z)$.
We shall compute
\begin{align*}
\<e_k,\j^* e_\ell\>_\dR
\end{align*}
for $-2\leq k,\ell\leq 0$.
At $0\in P=\{0,\infty\}$, $\j^*e_{0|\P^1\setminus\{\infty\}}$,
\begin{align*}
\j^*e_{-1}|_{\P^1\setminus\{\infty\}}-\j^*\nabla(z),
\text{ and }\quad
\j^*e_{-2}|_{\P^1\setminus\{\infty\}}-\j^*\nabla(1+\eq z)
\end{align*}
are in $ \(\j^*\M^0\otimes\tw^{-1}\Omega_{\X/S}^1\)|_{\P^1\setminus\{\infty\}}$.
Near $\infty\in P$,
take the coordinate
$w=z^{-1}$. 
Then $\j^*e_{-2|\P^1\setminus\{0\}}$, 
\begin{align*}
\j^*e_{-1}|_{\P^1\setminus\{0\}}-\j^*\nabla(-w), \text{ and }\quad
\j^*e_0|_{\P^1\setminus\{0\}}-\j^*\nabla(-1-\eq w)
\end{align*} 
are in  $ \(\j^*\M^0\otimes\tw^{-1}\Omega_{\X/S}^1\)|_{\P^1\setminus\{0\}}$.
Hence we obtain
\begin{align*}
\<e_k,\j^*e_\ell\>_\dR=
\begin{cases}
0&(|k-\ell|=0)\\
\tw^{-1}&(|k-\ell|=1)\\
\tw^{-1}\eq&(|k-\ell|=2).
\end{cases}
\end{align*}
\end{example}
\begin{example}[Continuation of Examples \ref{EXG} and \ref{Ga}]\label{GDual}
Consider the case of Example \ref{EXG} and \ref{Ga}. 
For each $a\in \Z$, take $n$ to be  $-a-1$ to obtain
\[\H_{\dR,a}^1=\Cok\left[0\to\tw^{-1}\O_Sz^{-a-1}dz\right]=\tw^{-1}\O_Sz^{-a-1}dz.\]
Set $e_{-a-1}=\tw^{-1}z^{-a-1}dz$.
Take another coordinate $w=z^{-1}$. 
Since we have
\[(\j^*e_a)|_{\P^1\setminus0}-\j^*\nabla(-w^{-a})
\in\(\j^*\M_{-a}^{a+1}\otimes \tw^{-1}\Omega_{\X/S}^1\)_{\big|\P^1\setminus\{0\}},  \]
we obtain
\[\<e_{-a-1}, \j^*e_a\>_\dR=\mathrm{Res}_{w=0} \((-w^{-a})\cdot \(-\tw^{-1}w^{a}\frac{dw}{w}\)\)=\tw^{-1}.  \]
\end{example}

\subsection{Analytic description of cohomology groups}
Let $\varpi_X\colon \widetilde{X}\to X$
denote the real blowing up of $X$ along $D$.
For a subset $H\subset D$, 
we set $\widetilde{H}\coloneqq \varpi_{X}^{-1}(H)$. 
Let 
$\A_{\widetilde{X}}^{\leqslant D}$
denote the sheaf of 
holomorphic functions on $Y=\widetilde{X}\setminus \widetilde{D}$
which are of moderate growth
along $\widetilde{D}$. 
For a subset $H\subset D$, 
let $\A_{\widetilde{X},\widetilde{D}}^{<H}$
denote the subsheaf of $\A_{\widetilde{X}}^{\leqslant D}$
whose section is of
rapid decay along $\widetilde{H}$. 
If $H=D$, we also use the notation $\A_{\widetilde{X}}^{<D}\coloneqq \A_{\widetilde{X},\widetilde{D}}^{<D}$.
If $H=\emptyset$, then $\A_{\widetilde{X},\widetilde{D}}^{\emptyset}=\A_{\widetilde{X}}^{\leqslant D}$. 
We have a differential
\[d\colon \A_{\widetilde{X},\widetilde{D}}^{<H}\longrightarrow
\A_{\widetilde{X},\widetilde{D}}^{<H}\otimes_{\varpi_X^{-1}\O_X} \varpi_X^{-1}\Omega^1_X.\]

Set $\widetilde{\X}\coloneqq S^\circ\times \widetilde{X}$, and
${\X}^\circ\coloneqq S^\circ\times X$.
Let $\varpi_{\X}\colon \widetilde{\X}\to \X^\circ$,
$p_{\widetilde{\X}}\colon \widetilde{\X}\to \widetilde{X}$,
and $\pi_{\widetilde{\X}}\colon\widetilde{\X}\to S^\circ$ denote the projections.
We also use the notations 
$\widetilde{\cal{D}}\coloneqq S^\circ\times \widetilde{D}$, e.t.c.
We take the restrictions 
of sheaves on $\X$ to $\X^\circ$ (resp. $S$ to $S^\circ$) without a mention.
Put 
\begin{align*}
\A_{\widetilde{\X},\widetilde{\cal{D}}/S^\circ}^{<H}\coloneqq \varpi_{\X}^{-1}\O_{\X^\circ}
\otimes_{p_{\widetilde{\X}}^{-1}\varpi_{X}^{-1}\O_{X}}
p_{\widetilde{\X}}^{-1}\A_{\widetilde{X},\widetilde{D}}^{<H}.
\end{align*}
There is the canonical relative differential
\begin{align*}
d_{\X/S}\colon \A_{\widetilde{\X},\widetilde{\cal{D}}/S^\circ}^{<H}
\to \A_{\widetilde{\X},\widetilde{\cal{D}}/S^\circ}^{<H}\otimes\varpi_{\X}^{-1}\Omega_{\X/S}^1.
\end{align*}
Then 
$\varpi_{\X}^{-1}\nabla\colon \varpi_\X^{-1}\M\to \varpi_{\X}^{-1}(\M\otimes \tw^{-1}\Omega^1_{\X/S})$
induces the connection 
\begin{align*}
\nabla\colon\A_{\widetilde{\X},\widetilde{\cal{D}}/S^\circ}^{<H}\otimes \varpi_\X^{-1}\M
\longrightarrow
\A_{\widetilde{\X},\widetilde{\cal{D}}/S^\circ}^{<H}\otimes 
\varpi_{\X}^{-1}(\M\otimes \tw^{-1}\Omega^1_{\X/S})
\end{align*}
by the Leibniz rule.
\begin{definition}
For a subset $H\subset D$, let
\begin{align*}
\DR_{\widetilde{\X},\widetilde{\cal{D}}/S^\circ}^{<H}(\M)
\coloneqq 
\left[\A_{\widetilde{\X},\widetilde{\cal{D}}/S^\circ}^{<H}\otimes \varpi_\X^{-1}\M
\xrightarrow{\nabla}
\A_{\widetilde{\X},\widetilde{\cal{D}}/S^\circ}^{<H}\otimes 
\varpi_{\X}^{-1}(\M\otimes \tw^{-1}\Omega^1_{\X/S}) \right]
\end{align*}
denote the complex placed at degree $0$ and $1$.
We then set 
\begin{align*}
\H^k_{\dR,H!}\coloneqq \R^k\pi_{\widetilde{\X}*}\DR_{\widetilde{\X},\widetilde{\cal{D}}/S^\circ}^{<H}(\M).
\end{align*}
\end{definition}

Let 
$\widetilde{\P}^1$ denote the 
real blowing up of $\widetilde{\P}^1$ at $\infty$. 
The boundary of $\widetilde{\P}^1$ is denoted by 
$\widetilde{\infty}=\{e^{\i\theta}\infty\mid \theta\in\R\}$
where $e^{\i\theta}\infty$ denotes the limit $\underset{r\to \infty}{\lim}re^{\i\theta}$
in $\widetilde{\P}^1$. 
Assume that $f$ is non-constant. 
The map $\tw^{-1}f\colon \X^\circ\to \P^1$, $\((\tw,\eq),x\)\mapsto \tw^{-1}f(x)$
uniquely lifts to the continuous map
\begin{align*}
\widetilde{f/\tw}\colon \widetilde{\X}\longrightarrow \widetilde{\P}^1.
\end{align*}
We define $\cal{P}^\rd\subset \widetilde{\cal{P}}=S^\circ\times \widetilde{P}$
as
$(\widetilde{f/\tw})^{-1}(\{e^{\i\theta}\infty \mid -\pi/2<\theta<\pi/2\})$.

For $H\subset D$, 
we set \[\X^\rd_H\coloneqq \cal{Y}^\circ\cup \cal{P}^\rd\cup \(S^\circ\times (E\setminus H)\),\] 
where $\cal{Y}^\circ\coloneqq S^\circ\times Y$.
Let 
$\widetilde{\imath}^H\colon \cal{Y}^\circ \to \X^\rd_H$ and
$\widetilde{\jmath}^H\colon \X^\rd_H\to \widetilde{\X}$ denote the inclusions.

Let $\K_\O$ denote 
the kernel of $\nabla\colon \M_{|\cal{Y}^\circ}\to\M_{|{\cal{Y}}^\circ}\otimes \Omega_{\cal{Y}^\circ/S^\circ}^1$.
\begin{lemma}\label{malgrange}
On $\widetilde{\X}$, we obtain the following$\colon$
\begin{enumerate}
\item
The $k$-th cohomology group 
$\H^k(\DR_{\widetilde{\X},\widetilde{\cal{D}}/S^\circ}^{<H}(\M))$
vanishes if $k\neq 0$.
\item
$\H^0(\DR_{\widetilde{\X},\widetilde{\cal{D}}/S^\circ}^{<H}(\M))\simeq
 \widetilde{\jmath}^H_!\widetilde{\imath}^H_*\K_\O$.\qed
\end{enumerate}
\end{lemma}
From this lemma, 
we have 
$\DR^{<H}_{\widetilde{\X},\widetilde{\cal{D}}/S^\circ}(\M){\simeq} \DR^{<H'}_{\widetilde{\X},\widetilde{\cal{D}}/S^\circ}(\M)$
in the derived category
if $H\setminus P=H'\setminus P$.
Hence, we may assume that $H\subset E=D\setminus P$. 

\begin{proposition}\label{limits}
On $S^\circ$, we have the following natural isomorphisms:
\begin{align*}
&\H_{\dR,\emptyset!}^1\simeq \H^1_{\dR}\simeq \lim_{a, b\to\infty}\H^1_{\dR, a, b},
&&\H_{\dR,E!}^1\simeq \lim_{a, b\to-\infty}\H^1_{\dR, a, b},
\\
&\H_{\dR,E_0!}^1\simeq \lim_{\substack{a\to-\infty,\\ b\to\infty }} \H_{\dR, a, b}^1,
&&\H^1_{\dR,E_\infty!}\simeq \lim_{\substack{a\to\infty,\\ b\to-\infty}}\H^1_{\dR, a, b}.
\end{align*}
\end{proposition}
\begin{proof}
We firstly consider the case $H=\emptyset$.
Then, by Malgrange, 
we obtain that 
$\R^i\varpi_{\X*}(\A_{\widetilde{\X}/S^\circ}^{\leqslant D})=0$
for $i\neq 0$, and $\varpi_{\X*}(\A_{\widetilde{\X}/S^\circ}^{\leqslant D})=\O_{\X}(*\cal{D})$.
It follows that 
\[\R\varpi_{\X*}\DR_{\widetilde{\X},\widetilde{\cal{D}}/S^\circ}^{<\emptyset}(\M)\simeq 
\DR_{\X/S}(\M),\]
which implies the first isomorphism. 
In the case
$H=E$, 
for each point 
$(s,e)\in S^\circ\times E$, 
we can take a neighborhood $\cal{V}_e=\nb(s)\times V_e$
on which we have the following exact sequence:
\[
0\longrightarrow\jmath^e_!(\K_\O)_{|\nb(s)\times (V_e\setminus\{e\})}
\longrightarrow (\M_{a, b})_{|\cal{V}_e}\longrightarrow (\M_{a, b}\otimes\tw^{-1}\Omega_{\X/S}^1)_{|\cal{V}_e}
\longrightarrow 0. \]
for sufficiently small $a$ or $b$,
where $\jmath^e\colon\nb(s)\times (V_e\setminus\{e\})\to \cal{V}_e$
denote the inclusion.
By Lemma \ref{malgrange}, we obtain 
\[\R\varpi_{\X*}\DR_{\widetilde{\X},\widetilde{\cal{D}}/S^\circ}^{<E}(\M)_{|\cal{V}_e}\simeq 
\DR_{\X/S}(\M_{a, b})_{|\cal{V}_e},\]
in the derived category. 
By glueing these (quasi-)isomorphisms, 
we obtain the quasi-isomorphism of complexes and hence 
$ \H_{\dR,E!}^1\simeq \lim_{a, b\to-\infty}\H^1_{\dR, a, b}$.
The other cases $H=E_0,E_\infty$ can also be proved in a similar way.
\end{proof}
\subsection{Analytic description of the pairings}
We may also define 
$\DR^{<H}_{\widetilde{\X},\widetilde{\cal{D}}/S^\circ}(\j^*\M)$
in a similar way,
and obtain the similar results as in the previous sections. 
In particular we have 
$\H^k(\DR^{<H}_{\widetilde{\X},\widetilde{\cal{D}}/S^\circ}(\j^*\M))=0$
for $k\neq 0$, 
and 
\[\H^0(\DR^{<H}_{\widetilde{\X},\widetilde{\cal{D}}/S^\circ}(\j^*\M))\simeq 
\j^{*}\widetilde{\jmath}^{H}_!\widetilde{\imath}_*^H\K_\O\]
where $\j$ denotes the involution on $\widetilde{\X}$ induced from that on $S^\circ$, and
$\j^*$ denotes the pull back as a $\pi_{\widetilde{\X}}^{-1}\O_{S^\circ}$-module,
i.e.
\[\j^{*}\widetilde{\jmath}^{H}_!{\imath}_*^H\K_\O\coloneqq \pi_{\widetilde{\X}}^{-1}\O_{S^\circ}
\otimes_{\j^{-1}\pi_{\widetilde{\X}}^{-1}\O_{S^\circ}}\j^{-1}\widetilde{\jmath}^{H}_!{\imath}_*^H\K_\O.\]
For $H\subset E$,
$\<\cdot,\cdot\>\colon \M\otimes \j^*\M\to\O(*\cal{D})$ induces the duality pairing
\begin{align*}
\<\cdot,\cdot\>_{\dR}^H\colon
\DR^{<H}_{\widetilde{\X},\widetilde{\cal{D}}/S^\circ}(\M)
\otimes 
\DR^{<E\setminus H}_{\widetilde{\X},\widetilde{\cal{D}}/S^\circ}(\j^*\M)
\longrightarrow \DR_{\widetilde{\X}/S^\circ}^{<D}(\O_{\X},d_{\X/S})
\end{align*}
where $\DR_{\widetilde{\X}/S^\circ}^{<D}(\O_{\X},d_{\X/S})$ denotes the following complex:
\begin{align*}
\A_{\widetilde{\X}/S^\circ}^{<D}\xrightarrow{d_{\X/S}}
\A_{\widetilde{\X}/S^\circ}^{<D}\otimes \varpi_{\X}^{-1}(\tw^{-1}\Omega_{\X/S}^1).
\end{align*}

\begin{theorem}\label{analytic duality}
The pairing $\<\cdot,\cdot\>_{\dR}^H$
induces the perfect pairing 
\begin{align*}
\<\cdot,\cdot\>_{\dR} 
\colon \H^1_{\dR,H!}\otimes \j^*\H^1_{\dR,(E\setminus H)!}
\longrightarrow \O_{S^\circ}.
\end{align*}
If $H=E,E_0,$ or $E_\infty$, then the pairings are compatible with 
the isomorphisms in Proposition $\ref{limits}$, 
and pairings in $\S \ref{pairing}$.
\end{theorem}
\begin{proof}
We also have 
$\H^k(\DR_{\widetilde{\X}/S^\circ}^{<D}(\O_{\X},d_{\X/S^\circ}))=0$ for $k\neq 0$
and 
\[\H^0(\DR_{\widetilde{\X}/S^\circ}^{<D}(\O_{\X},d_{\X/S}))=\widetilde{\jmath}_!\pi_{\cal{Y}^\circ}^{-1}\O_{S^\circ},\]
where $\widetilde{\jmath}\colon \cal{Y}^\circ\to \widetilde{\X}$
denote the inclusion. 
Hence, in the derived category,
to consider $\<\cdot,\cdot\>_{\dR}^H$
is equivalent to consider the pairing
\begin{align*}
\<\cdot,\cdot\>_{\dR}^H\colon 
\widetilde{\jmath}_!^H\widetilde{\imath}^H_*\K_\O
\otimes 
\j^*\widetilde{\jmath}_!^{H'}\widetilde{\imath}^{H'}_*\K_\O
\longrightarrow \widetilde{\jmath}_!\pi_{\cal{Y}^\circ}^{-1}\O_{S^\circ}
\end{align*}
where $H'\coloneqq E\setminus H$. 
This pairing is perfect in the sense that 
the induced morphisms 
\begin{align*}
\widetilde{\jmath}_!^H\widetilde{\imath}^H_*\K_\O\longrightarrow 
\R\hom_{\pi_{\widetilde{\X}}^{-1}\O_{S^\circ}}(\j^*\widetilde{\jmath}_!^{H'}\widetilde{\imath}^{H'}_*\K_\O,
\widetilde{\jmath}_!\pi_{\cal{Y}^\circ}^{-1}\O_{S^\circ})\\
\j^*\widetilde{\jmath}_!^{H'}\widetilde{\imath}^{H'}_*\K_\O\longrightarrow 
\R\hom_{\pi_{\widetilde{\X}}^{-1}\O_{S^\circ}}(\widetilde{\jmath}_!^H\widetilde{\imath}^H_*\K_\O,
\widetilde{\jmath}_!\pi_{\cal{Y}^\circ}^{-1}\O_{S^\circ})
\end{align*}
are both isomorphisms
(This can be proved by the same way as in \cite{Hien}
considering the fact that $\widetilde{\cal{P}}\setminus \cal{P}^\rd$ is the closure of $\j(\cal{P}^\rd)$
in $\widetilde{\cal{P}}$).
Then, by the Verdier duality, we obtain the theorem (c.f. \cite{Hien}).
To check the compatibility when $H=E,E_0, E_\infty$ is left to the reader.
\end{proof}
\section{Betti homology groups and period integrals}
\subsection{Preliminary}
Let $M$ be a compact oriented smooth manifold with the boundary $\del M$. 
For a non-negative integer $\ell$, 
let $C_\ell(M)=\bigoplus_{c}\Q\<c\>$ denote the $\Q$-vector space 
generated by piecewise smooth maps $c\colon \triangle_\ell\to M$
from a $\ell$-simplex \[\triangle_\ell=\{(t_1,\dots, t_\ell) \in \R^\ell\mid0\leq t_1\leq \cdots\leq t_\ell\leq 1 \}\]
to $M$. 
For a closed subset $A\subset M$, 
let $C_\ell(A)$ denote the subspace of $C_\ell(M)$
generated by the maps whose image is contained in $A$. 
We put $C_\ell(M, A)\coloneqq C_\ell(M)/ C_\ell(A)$. 
Then, we define $\mathscr{C}_{M, \del M}^{-\ell}$
as a sheaf associated to the presheaf 
\begin{align*}
V\mapsto C_\ell \(M, (M\setminus V)\cup \del M \),
\quad 
\end{align*}
where $V$ is an open subset in $\widetilde{X}$. Together with the usual boundary operator, 
we obtain a co-chain complex $\mathscr{C}_{M, \del M}^{\bullet}$ of sheaves 
of $\Q$-vector spaces on $M$. 
It is known that $\mathscr{C}_{M, \del M}^{\bullet}$
is a homotopically fine resolution of $\Q_M[\dim M]$
(See \cite{Swan},\cite{Hien}).

%

Let $(N,\del N)$ be another compact oriented smooth manifold with boundary. 
Let $h\colon (N,\del N)\to (M,\del M)$ be a closed embedding. 
\begin{lemma}\label{push}
For a $\Q_M$-module $\F$, we have a natural morphism
\begin{align*}
h_*\colon h_*(\scr{C}_{N,\del N}^\bullet\otimes h^{-1}\F)
\longrightarrow \scr{C}_{M,\del M}^\bullet \otimes \F.
\end{align*}
\end{lemma}  
 \begin{proof}
 We firstly consider the case $\F=\Q_{M}$. 
 In this case, 
 \[h_*\colon h_*\scr{C}_{N,\del N}^{-\ell}\to \scr{C}_{M,\del M}^{-\ell}\]
 is given in the usual way;
for a piecewise smooth map $c$ from $\ell$-simplex to $N$, 
take $h_*\<c\>\coloneqq \<h\circ c\>$.
We then consider the general case. 
By the projection formula,
we have
\begin{align*}
h_*(\scr{C}_{N,\del N}^{-\ell}\otimes h^{-1}\F)
\xleftarrow{\sim}&h_*({\scr{C}_{N,\del N}^{-\ell}})\otimes \F.
\end{align*}
Then $h_*\otimes \id_\F$ defines 
the desired morphism.
 \end{proof}

Let $I=[0,1]$ be the closed interval.
Let $h_I\colon I\times N\to M$ be the $C^\infty$ family of 
closed embeddings $(N,\del N)\hookrightarrow (M,\del M)$. 
For $t\in I$, set $h_t\coloneqq h_{I|\{t\}\times N}\colon N\to M$.  
A sheaf 
$\scr{G}$ on $I\times N$ is said to be trivial along $I$ if
the adjunction 
$\pr^{-1}\pr_{*}(\scr{G})\to \scr{G}$ is an isomorphism
where $\pr\colon I\times N\to N$ denotes the projection.
\begin{lemma}\label{homotopy}
Let $\F$ be a $\Q_M$-module. 
Assume that $h_I^{-1}\F$ is trivial along $I$.
Then the morphisms 
\[h_{i*}\colon \Gamma(N,\scr{C}_{N,\del N}^{\bullet}\otimes h_i^{-1}\F)
\longrightarrow \Gamma(M,\scr{C}_{M,\del M}^\bullet \otimes \F)
\quad (i=0,1)\]
are chain homotopic to each other under the identification $h_0^{-1}\F\simeq h_1^{-1}\F$.  
\end{lemma}
\begin{proof}
We have 
\[h_\star^{(\ell)}\colon h_{I*}(\pr^{-1}\scr{C}_{N,\del N}^{-\ell})\longrightarrow \scr{C}_{M,\del M}^{-\ell+1}\] 
in the usual way:
$h^{(\ell)}_\star(\<c\>)\coloneqq \sum_i (-1)^i \<h_I\circ (\id_I\times c)\circ s_i\>$,
where $s_i\colon \triangle_{\ell}\hookrightarrow I\times \triangle_{\ell-1}$
is defined as
$s_i(t_1,\dots, t_{\ell})=(t_i,( t_1,\dots,t_{i-1},t_{i+1},\dots, t_{\ell})) $.
We then consider 
$h_\star^{(\ell)}\otimes \id_{\F}$.
Taking $\Gamma(M,-)$, we obtain
\begin{align*}
h^{(\ell)}_\star\colon \Gamma(N,\scr{C}_{N,\del N}^{-\ell}\otimes\pr_*h_I^{-1}\F )\to 
\Gamma(M,\scr{C}^{-\ell+1}_{M,\del M}\otimes \F)
\end{align*}
Since $h_I^{-1}\F$ is trivial along $I$, 
we have natural isomorphisms
\begin{align*}
\Gamma(N,\scr{C}_{N,\del N}^{-\ell}\otimes\pr_*h_I^{-1}\F )
\simeqto \Gamma(N,\scr{C}_{N,\del N}^{\bullet}\otimes h_i^{-1}\F).
\end{align*}
We then obtain the lemma by the usual calculation.
\end{proof}
%
%
%
%
%
\subsection{Betti homology groups}Fix a sub-field $\k\subset \C$.
We shall use the notations in \S \ref{DR}.
Put $\q\coloneqq \exp(2\pi\i \eq/\tw)$, 
which is a holomorphic function on $S^\circ$.  
\begin{definition}
Let $\K=\K(f, g)$ be the subsheaf of $\M_{|\cal{Y}^\circ}$
defined as follows:
\begin{align*}
\K\coloneqq \k[\q^{\pm 1}]_{\cal{Y}^\circ}e^{-f/\tw}g^{\eq/\tw}\bm{e}.
\end{align*}
\end{definition}
Although $g^{\eq/\tw}$ is a multivalued function, 
$\K$ is well defined since the ratio of every two distinct 
values of $g^{\eq/\tw}$ is given by $q^m$ for some $m\in\Z$. 
$\K$ is a local system of free $\k[q^{\pm 1}]$-modules of rank one.
We have $\K\subset \K_\O$
and $\K\otimes_{\k[q^{\pm 1}]}\pi_{\cal{Y}^\circ}^{-1}\O_{S^\circ}=\K_\O$. 

We also note that $\K$ is trivial along $S^\circ$
in the sense that 
the adjunction 
\begin{align}\label{ADJK}
p_{\cal{Y}^\circ}^{-1}p_{\cal{Y}^\circ *}\K\longrightarrow \K
\end{align}
is an isomorphism, 
where $p_{\cal{Y}^\circ}\colon \cal{Y}^\circ\to Y$ denotes the projection. 

\begin{definition}
For a subset $H\subset E$, 
we set $\K_{H!}\coloneqq \widetilde{\jmath}_!^H\widetilde{\imath}_*^{H}\K$, 
\[\scr{C}_{\widetilde{\X},\widetilde{\cal{D}}/S^\circ}^{<H}(\M)
\coloneqq p_{\widetilde{\X}}^{-1}
\scr{C}_{\widetilde{X},\widetilde{D}}^\bullet\otimes
\K_{H!},\]
and 
\begin{align*}
\H^{\Be}_{\ell,H!}=\H^{\Be}_{\ell,H!}(f,g;\k)\coloneqq 
\R^{-\ell}\pi_{\widetilde{\X}*}
\scr{C}_{\widetilde{\X},\widetilde{\cal{D}}/S^\circ}^{<H}(\M).
\end{align*}
We also use the notations 
$\H^\rd_\ell\coloneqq \H^\Be_{\ell,E!}$, and $\H^\mg_{\ell}\coloneqq \H^\Be_{\ell,\emptyset!}$.
\end{definition}
Since $\scr{C}_{\widetilde{X},\widetilde{D}}^\bullet$
is homotopically fine, we have
\[\H^\Be_{\ell,H!}=\H^{-\ell}\(\pi_{\widetilde{\X}*}\(p_{\widetilde{\X}}^{-1}
\scr{C}_{\widetilde{X},\widetilde{D}}^\bullet\otimes
\K_{H!}\)\).\]

\begin{lemma}
$\H^{\Be}_{\ell,H!}$ is a locally trivial $\k[q^{\pm 1}]_{S^\circ}$-module. 
\end{lemma}
\begin{proof}
This lemma easily follows from the facts that 
the adjunction (\ref{ADJK}) is an isomorphism
and 
that 
the projection $\X_H^\rd\to S^\circ$  is locally trivial. 
\end{proof}
\begin{lemma}\label{BEV}
If $f$ is not constant, then
$\H^\Be_{\ell,H!}=0$ for $\ell\neq 1$. 
\end{lemma}
\begin{proof}
It is enough to see that $\H^\Be_{\ell,H!}=0$ for $\ell=0,2$. 
 We have
\[\H^\Be_{0, H!}=\mathrm{Cok}
\(\pi_{\widetilde{\X}*}
(\scr{C}^1_{\widetilde{X},\widetilde{D}}\otimes \K_H)
\to \pi_{\widetilde{\X}*}
(\scr{C}^0_{\widetilde{X},\widetilde{D}}\otimes \K_H)\)=0\]
since any point in $Y$ can be connected by a path with 
a boundary point in $\cal{P}^\rd$. 
We also have
\begin{align*}
\H^\Be_{2,H!}=\R^{-2}\pi_{\widetilde{\X}*}(\K_H[2])=\pi_{\widetilde{\X}*}\K_H=0
\end{align*}
since the section should be locally constant and
zero on $\cal{P}^\rd$. 
Both facts relies on the assumption that $P\neq \emptyset$. 
\end{proof}

We put 
\begin{align*}
&\j^*\H_{\ell,H!}^\Be\coloneqq \k[q^{\pm 1}]_{S^\circ}\otimes_{\k[q^{\pm1}]}\j^{-1}\H^\Be_{\ell,H!},
\end{align*}
where the tensor product is given by using the morphism $\k[q^{\pm 1}]\simeqto \k[q^{\pm 1}],q\mapsto q^{-1}$.
\begin{proposition}
For $H\subset E$, there is a natural perfect pairing 
\begin{align*}
\<\cdot,\cdot\>_{\Be}\colon\H^\Be_{\ell,H!}\otimes \j^*\H^{\Be}_{2-\ell,(E\setminus H)!}
\longrightarrow \k[q^{\pm 1}]_{S^\circ}.
\end{align*}
\end{proposition}
\begin{proof}
Put 
\begin{align*}
\j^*\K_{H!}\coloneqq 
\k[q^{\pm 1}]_{\widetilde{\X}}\otimes_{\k[q^{\pm1}]}\j^{-1}\K_{H!}
\end{align*}
where the tensor product is given by using the morphism $\k[q^{\pm 1}]\simeqto \k[q^{\pm 1}],q\mapsto q^{-1}$.
Then, the pairing $\<\cdot,\cdot\>_{\dR}^H$ in the proof of Theorem \ref{analytic duality}
is restricted to the following pairing:
\begin{align*}
\<\cdot,\cdot\>^H_{\Be}\colon \K_{H!}
\otimes 
\j^*\K_{H'!}
\longrightarrow \widetilde{\jmath}_!\k[q^{\pm 1}]_{\cal{Y}^\circ}
\end{align*} 
where $H'\coloneqq E\setminus H$.
This pairing is perfect in a similar sense 
as in the proof of  Theorem \ref{analytic duality} (the proof is also similar). 
Noting that $\scr{C}_{\widetilde{X},\widetilde{D}}^\bullet \simeq \Q_{\widetilde{X}}[2]$
in the derived category, 
by the Verdier duality and the universal coefficient theorem, we obtain the following 
exact sequence: 
\begin{align*}
0\longrightarrow 
\scr{E}\!xt^1(\H^\Be_{0,H!}, \k[q^{\pm 1}]_{S^\circ})
\longrightarrow \j^*\H^\Be_{1,H'!}
\longrightarrow \hom(\H^\Be_{1,H!},\k[q^{\pm 1}]_{S^\circ})\longrightarrow 0
\end{align*}
Then we obtain the proposition by
Lemma \ref{BEV}. 
\end{proof}
\begin{corollary} $\H^\Be_{1,H!}$ is torsion free.\qed
\end{corollary}
We note that 
$\<\cdot,\cdot\>_{\Be}\colon \H_{1,H!}^\Be\otimes\j^*\H^\Be_{1,H'!}\to \O_{S^\circ}$
is skew symmetric, i.e. satisfies the relation 
$\<\cdot,\cdot\>_\Be=-\j^*\<\cdot,\cdot\>_\Be\circ\mathrm{ex}$, 
where $\mathrm{ex}\colon \H_{1,H!}^\Be\otimes\j^*\H^\Be_{1,H'!}
\to \j^*\H^\Be_{1,H'!}\otimes  \H_{1,H!}^\Be$
denote the exchange
and $\j^*\<\cdot,\cdot\>_\Be\colon \j^*\H^\Be_{1,H'!}\otimes  \H_{1,H!}^\Be\to \k[q^{\pm 1}]$
denotes the pull back.

For $e\in E$, let $n_e$ be the order of $g$
at $e$, i.e. $g_{|V_e}=z^{n_e}$
for some coordinate neighborhood $(V_e, z)$
centered at $e$.
\begin{lemma}\label{cokernel}
For $H_1\subset H_2\subset E$,
there is a canonical exact sequence
\[0\longrightarrow\H^\Be_{1,H_2!}\longrightarrow \H^\Be_{1,H_1!}\longrightarrow 
\bigoplus_{e\in H_2\setminus H_1}\(\frac{\k[q^{\pm 1}]}{(1-q^{n_e})}\)_{S^\circ}\longrightarrow 0.\]
\end{lemma}
\begin{proof}
Let $\kappa\colon \widetilde{\jmath}^{H_2}_!\widetilde{\imath}^{H_2}_*\K
\to  \widetilde{\jmath}^{H_1}_!\widetilde{\imath}^{H_1}_*\K$
be the canonical extension. 
Then, $\Cok(\kappa)$
is supported on 
$\bigsqcup_{e\in H_2\setminus H_1}S^\circ\times \widetilde{e}$, 
where $\widetilde{e}\coloneqq \varpi_X^{-1}(e)$ for each $e\in H_2\setminus H_1$. 
On $S^\circ\times \widetilde{e}$, $\Cok(\kappa)$
is a local system trivial along $S^\circ$ and the monodromy 
around $\widetilde{e}$ is $q^{n_e}$. 
Hence
\begin{align*}
\R^{-\ell}\pi_{\widetilde{\X}*}\(p_{\widetilde{X}}^{-1}\scr{C}_{\widetilde{X},\widetilde{D}}\otimes \Cok(\kappa)\)
&\simeq \R^{-\ell}\pi_{\widetilde{\X}*}\(\Q_{\widetilde{\X}}[2]\otimes \Cok(\kappa)\)\\
&\simeq \begin{cases}
\bigoplus_{e\in H_2\setminus H_1}\(\frac{\k[q^{\pm 1}]}{(1-q^{n_e})}\)_{S^\circ}
&(\ell=1)\\
0&(\ell\neq 1),\end{cases}
\end{align*}
which implies the lemma.
\end{proof}
\subsection{Local sections of Betti homology group sheaves} 
For an open subset $V\subset S^\circ$,
put $I_V\coloneqq V\times I$, 
and $\widetilde{\X}_{|V}\coloneqq V\times \widetilde{X}$. 
Let 
\[\gamma\colon I_V\to \widetilde{\X}_{|V}, \quad
(s, t)\mapsto (s, \gamma_s(t))\]
be a family of 
piecewise smooth closed embeddings $\gamma_s\colon I\to \widetilde{X}$
over $V$
such that $\gamma_s(\del I)\subset \widetilde{D}$. 
Let $p_I\colon I_V\to I$
and $\pi_I\colon I_V\to V$
denote the projections. 

Take $H\subset E$ and
a global section
$\varepsilon\in \Gamma(I_V, \gamma^{-1}\K_{H!})$.
Then
$\<\id_I\>\otimes \varepsilon$
defines a section of $(-1)$-th cohomology group
\begin{align*}
\H_{1}(I,\gamma^{-1}\K_{H!})\coloneqq \H^{-1}\(\pi_{I*}\(p_I^{-1}\scr{C}^{\bullet}_{I,\del I}\otimes \gamma^{-1}\K_{H!}\)\).
\end{align*}
The section will be denoted by $[\id\otimes \varepsilon]$. 
%
%
By Lemma \ref{push}, 
we have
\[\gamma_{s*}\colon 
\H_{1}(I,\gamma^{-1}\K_{H!})_s
\to (\H_{1,H!}^{\Be})_s\]
for $s\in V$.
The following lemma follows from
Lemma \ref{homotopy} by the standard argument:
\begin{lemma}\label{section kousei}
There exists
a unique section 
$[\gamma\otimes \varepsilon]\in\H_{1,H!}^{\Be}(V)$
such that 
\[[\gamma\otimes \varepsilon]_s=\gamma_{s*}([\id\otimes \varepsilon]_s)\]
 for any $s\in V$.\qed
\end{lemma}

Let $\gamma'\colon I_{V'}\to \widetilde{\X}_{|V'}$ 
be another family of 
paths over $V'\coloneqq \j(V)$. 
Let $\varepsilon'$ be a global section in $\Gamma(I_{V'}, \gamma'^{-1}\K_{H'})$, 
where $H'\coloneqq E\setminus H$.  
In a similar way as above, we have a section $[\gamma'\otimes\varepsilon']\in \j^*\H_{1,H'}^\Be(V)$.

\subsection{Period pairings}\label{3.4}
The contents of this subsection is mostly contained in the works \cite{MMT},
\cite{matsubara}, \cite{FSY} (see also \cite{Hien}) in a more general setting. 
We recall a part of their results in our (trivially) relative setting. 
Let $\Db_{\widetilde{X}}^{\rd, -r}$
denote the sheaf of
rapid decay distributions on $\widetilde{X}$, i.e. 
the sheaf whose section
on small open $V\subset \widetilde{X}$
are the distributions
\begin{align*}
\psi\in \Hom_{\rm cont}(\Gamma_c(V, \Omega_{\widetilde{X}}^{\infty, r}),\C)
\end{align*}
on the space $\Omega_{\widetilde{X}}^{\infty, r}$ of 
$C^\infty$ differential forms on $\widetilde{X}$
of degree $r$ with compact support in $V$
satisfying the rapid decay condition along $V\cap \widetilde{D}$
(c.f \cite{Hien2}).
We set
\[\Db_{\widetilde{\X}/S^\circ}^{\rd,-r}\coloneqq \varpi_{\X}^{-1}\O_{\X^\circ}
\otimes_{p_{\widetilde{X}}^{-1}\varpi_{X}^{-1}\O_X}\(p_{\widetilde{X}}^{-1}\varpi_{X}^{-1}\Db_{\widetilde{X}}^{\rd, -r}\).\]
We 
then obtain the local period pairing
\begin{align*}
\<\cdot,\cdot\>_{{\rm Per}}^H\colon
\scr{C}^{<H}_{\widetilde{\X},\widetilde{\cal{D}}/S^\circ}(\M) 
\otimes
\DR_{\widetilde{\X},\widetilde{\cal{D}}/S^\circ}^{<H'}(\j^*\M) 
&\longrightarrow 
\Db_{\widetilde{\X}/S^\circ}^{\rd,-\bullet},\\
(c\otimes \varepsilon)\otimes \omega&\mapsto (\eta\mapsto \int_c\eta\wedge\<\varepsilon,\omega\>),
\end{align*}
which induces 
the (family of) period pairing(s)
\begin{align*}
\<\cdot, \cdot\>_{\rm Per}\colon 
\H_{1,H!}^{\Be}\otimes_{\k[q^{\pm 1}]}
\j^*\H_{\dR,H'!}^1\longrightarrow \O_{S^\circ}.
\end{align*}

Recall that 
there is a natural
the injection $i_H\colon \H_{1,H!}^{\Be}\to \H_{\dR,H!}^1$
induced from $\K\hookrightarrow \K_\O$. 
On the other hand, using the pairings
$\<\cdot,\cdot\>'_\dR\coloneqq 2\pi\i\<\cdot,\cdot\>_\dR$ and $\<\cdot,\cdot\>_{\rm Per}$, 
we also obtain 
an injection
$i_H'\colon \H_{1,H!}^{\Be}\to \H_{\dR,H!}^1$.

\begin{lemma}\label{COMPATI}
The injections $i_H$ and $i'_H$ defined above coincide with each other. 
\end{lemma}
\begin{proof}
The statement follows from the following commutative diagram:
\begin{align*}
\xymatrix{
\DR_{\widetilde{\X},\widetilde{\cal{D}}/S^\circ}^{<H}(\M)[2]\otimes 
\DR^{<H'}_{\widetilde{\X},\widetilde{\cal{D}}/S^\circ}(\j^*\M) 
\ar[r]^{\ \ \ \ \ \ \ \ \ \ \ \ \<\cdot,\cdot\>'_\dR}&\DR_{\widetilde{\X}/S}^{<D}(\O_{\X},d_{\X/S})[2]\\
\K_H[2]\otimes \j^*\K_{\O,H'}\ar[r]^{\ \ \ \ \ \  \<\cdot,\cdot\>}
\ar[u]\ar[d]&\widetilde{\jmath}!\pi_{{\cal{Y}^\circ}}^{-1}\O_{S^\circ}[2]
\ar[u]\ar[d]\\
\scr{C}^{<H}_{\widetilde{\X},\widetilde{\cal{D}}/S^\circ}(\M)\otimes
\DR_{\widetilde{\X},\widetilde{\cal{D}}/S^\circ}^{<H'}(\j^*\M) 
\ar[r]^{\ \ \ \ \ \ \ \ \<\cdot,\cdot\>_{\rm Per}}&\Db_{\widetilde{\X}/S^\circ}^{\rd,-\bullet} }
\end{align*}
where we have put $\j^*\K_{\O,H'}\coloneqq\j^{*}\widetilde{\jmath}^{H}_!\widetilde{\imath}_*^H\K_\O$.
\end{proof}
We also note that
$i_H=i'_H$ 
induces an isomorphism
\[\mathrm{i}_H\colon \H_{1,H!}^\Be\otimes_{\k[q^{\pm 1}]} \O_{S^\circ}
\simeqto \H^1_{\dR,H!}.\]
This isomorphism
trivialize
the actions of $\nabla_\euler$ and $\shift$
in the following sense:
if we define 
$\nabla_\euler=\id\otimes [\euler,\cdot]$
and $\shift=\id\otimes \sigma^\#$
on $\H_{1,H!}^\Be\otimes_{\k[q^{\pm 1}]} \O_{S^\circ}$,
then $\mathrm{i}_H$
is compatible with these operators.

By the similar arguments as in the proof of Lemma \ref{COMPATI},
we obtain the following:
\begin{lemma}\label{compatible}
We have the following commutative diagram:
\begin{align*}
\xymatrix{
\H^\Be_{1,H!}\otimes \j^*\H^\Be_{1,H'!}\ar[r]^{\ \ \ \ \ \ \ \<\cdot,\cdot\>_\Be}\ar[d]&\k[q^{\pm}]_{S^\circ}\ar[d]\\
\H_{\dR,H!}^1\otimes \j^*\H^1_{\dR,H'!}\ar[r]^{\ \ \ \ \ \ \ \<\cdot,\cdot\>'_\dR}&\O_{S^\circ}
}
\end{align*}
where the vertical arrows are the inclusions.\qed
\end{lemma}

\section{Geometric construction of Stokes filtered quasi-local systems}
\subsection{Glueing Betti homology groups}\label{Glueing Betti}
Set $S^*\coloneqq S\setminus|(\tw\eq)_0|\simeq (\C^*)^2$.
Define subsets $S_{\R_+}^*$ and $S_{\R_-}^*$ of $S^*$ by 
\begin{align*}
&S_{\R_+}^*\coloneqq\{(\tw,\eq)\in S^*\mid \eq/\tw\in\R_{>0} \}\\
&S_{\R_-}^*\coloneqq\{(\tw,\eq)\in S^*\mid \eq/\tw\in \R_{<0} \}. 
\end{align*}
Put $S_\R^*\coloneqq S_{\R_+}^*\cup S_{\R_-}^*$.
We then define a quasi-local system on $S^*$ (see \S \ref{QLS})
constructible with respect to the stratification 
$S^*=S_{\R_+}^*\sqcup S_{\R_-}^* \sqcup (S^*\setminus S_\R^*) $
as follows:
\begin{definition}
Let $\H^\Be_{1,\bm{0}}=\H^\Be_{1,\bm{0}}(f,g)$
be the sheaf of $\k[q^{\pm 1}]$-modules on $S^*$ defined as follows:
For an open subset $V\subset S^*$, we set
\begin{align*}
\H^\Be_{1,\bm{0}}(V)\coloneqq 
\begin{cases}
\H^\mg_{1}(V)&(V\cap S^*_\R=\emptyset)\\
\H^{\Be}_{1,E_0!}(V)&(V\cap S^*_{\R_+}= \emptyset, V\cap S^*_{\R_-}\neq\emptyset)\\
\H^\Be_{1,E_\infty!}(V)&(V\cap S^*_{\R_+}\neq \emptyset, V\cap S^*_{\R_-}=\emptyset)\\
\H^\rd_1(V)&(V\cap S^*_{\R_+}\neq \emptyset, V\cap S^*_{\R_-}\neq\emptyset).
\end{cases}
\end{align*}
The restrictions are defined as the usual restrictions or the canonical morphisms. 
\end{definition}
By Lemma \ref{cokernel}, $\H_{1,\bm{0}}^\Be$ is a quasi-local system on $S^*$. 
Set $\H^1_{\dR,\bm{0}}\coloneqq \H^1_{\dR,0,0}$. 
\begin{corollary}
There is a natural isomorphism
\begin{align*}
\mathrm{i}_{\bm{0}}\colon\H^\Be_{1,\bm{0}}\otimes_{\k[q^{\pm 1}]}\O_{S^*}\simeqto \H^1_{\dR,\bm{0}|S^*}
\end{align*}
\end{corollary}
\begin{proof}
By Proposition \ref{limits} and Lemma \ref{SUP}
we have the natural 
isomorphisms
\begin{align*}
\H_{\dR,\bm{0}}^1\simeq 
\begin{cases}
\H_{\dR,E_0!}^1&\text{on }S^*\setminus S^*_{\R_+}\\
\H_{\dR,E_\infty!}^1&\text{on }S^*\setminus S^*_{\R_-}\\
\H_{\dR}^1&\text{on }S^*\setminus S^*_\R.
\end{cases}
\end{align*}
Then the isomorphisms $\mathrm{i}_H$
in \S \ref{3.4}
induce the desired isomorphism $\mathrm{i}_{\bm{0}}$.
\end{proof}

Let $B$
be the complex plane considered in \S \ref{stokes}.
Let 
$\phi_S\colon B\to S$
be the morphism 
defined by $\phi(u, v)=(uv, v)$, i.e. 
$\tw=uv$, and $\eq=v$.
It induces an isomorphism
$\phi_S\colon B^*\simeqto S^*$. 

\begin{definition}
We set 
\begin{align*}
\L^\Be=\L^\Be{(f,g)}\coloneqq 
\imath_T^{-1}\widetilde{\jmath}_{B*}\phi_S^{-1}\H_{1,\bm{0}}^\Be,
\end{align*}
which is a quasi-local system on $(T,\Theta)$. 
\end{definition}
\begin{lemma}
$\L^\Be$ is saturated. 
\end{lemma}
\begin{proof}
Since the sum $\K_{E_0!}\oplus \K_{E_\infty!}\to \K_{\emptyset !}$ is surjective,
we obtain the lemma. 
\end{proof}
Let $e_1,\dots, e_r$
be a frame of $\phi_S^*\H_{\dR,\bm{0}}^1$ around the origin of $B$.
\begin{definition}
For an open subset $V\subset T$
and a section $\varphi\in \Gamma(V,\IQ)$, 
we define the subsheaf
$\L^\Be_{\leqslant \varphi}\subset \L^\Be_{|V}$
as follows:
Let $\mathfrak{s}$ be a local section of $ \L^\Be_{|V}$. 
Then, there exist
an open subset $\widetilde{V'}\subset \widetilde{B}$ with $T\cap \widetilde{V'}\subset V$ and
a representative $\widetilde{\mathfrak{s}}\in \Gamma(\widetilde{V'}\setminus T, \H^\Be_{1,\bm{0}})$
of $\mathfrak{s}$.
We have the expression
\begin{align*}
\widetilde{\mathfrak{s}}=\sum_{i=1}^rh_i(u,v)e_i
\end{align*}
for some $h_i\in \Gamma(\widetilde{V'},\widetilde{\jmath}_{B*}\O_{B^*})$.
The section $\mathfrak{s}$ is a local section of $\L^\Be_{\leqslant \varphi}$
if and only if $e^{-\varphi}h_i\in \A_{B}^{\leqslant Z}$
for all $i=1,\dots, r$. 
\end{definition}
\begin{lemma}
$\L^\Be_{\leqslant}$ is a pre-Stokes filtration on $\L^\Be$.\qed
\end{lemma}
By Lemma \ref{compatible}, 
we have the following description of
$\L^\Be_{\leqslant}$:
\begin{lemma}\label{period criterion}
A local section $\mathfrak{s}$ of $\L^\Be$
is in $\L^\Be_{\leqslant\varphi}$
if the period integral
\begin{align*}
e^{-\varphi}\ \<\mathfrak{s},\omega \>_{\rm Per}
\end{align*}
twisted by $e^{-\varphi}$ is of moderate growth
for any local section $\omega$ of $\j^*\H_{\dR,\bm{0}}^1$.\qed
\end{lemma}

We also note that we have the 
quasi-duality pairing
$\<\cdot,\cdot\>_{\Be,\pm}$
 on $\L^\Be$
 using $\<\cdot,\cdot\>_{\Be}$ (defined in the previous section)
 in a natural way. 

\begin{lemma}
The quasi-duality pairing
$\<\cdot,\cdot\>_{\Be\pm}$ 
on $\L^\Be$
is compatible with
the pre-Stokes filtration $\L^\Be_{\leqslant}$. \qed
\end{lemma}

\subsection{Geometric conditions}

In the following discussion, 
we will assume the following conditions:
\begin{align}
\label{A_1}
&\text{The restriction $f_{|U}$ has at most $A_1$-singularities.}\\
\label{nowhere}
&\text{The divisor $E$ and the critical points of $f_{|U}$ do not intersect.}
\end{align}

\begin{definition}
Let $\Sigma(f, g)_Y$ be  the zero locus of the relative one form
\begin{align*}
df-\eq g^{-1}dg\in
H^0 \(\C \times Y, \Omega^1_{\C\times Y/\C}\). 
\end{align*}
We then define $\Sigma(f, g)$ to be the closure of $\Sigma(f, g)_Y$ in $\C \times U$. 
\end{definition}

\begin{lemma}\label{cover}
The natural projection $\pi_\Sigma\colon\Sigma(f, g)\to \C$ 
is a non-ramified finite covering over the disc 
$\Delta_\eq(r)=\{\eq\in\C_\eq\mid |\eq|<r\}$ 
for sufficiently small $r>0$. 
\end{lemma}
\begin{proof}
There exist $r>0$ and a small neighborhood $V$ of $P\subset X$
such that we have $({\Delta_\eq(r)}\times V)\cap\Sigma(f, g)=\emptyset$.
It follows that $\Sigma(f, g)\cap({\Delta_\eq(r)}\times U)$ is closed in ${\Delta_\eq(r)}\times X$.
Hence the projection
\begin{align*}
\pi_{r}:\Sigma(f, g)\cap (\Delta_\eq(r)\times Y)\longrightarrow {\Delta_\eq(r)}
\end{align*}
is a finite morphism.
It remains to prove that $\pi_r$ is non-ramified for sufficiently small $r>0$.
Take a point $p$ in $\pi_\Sigma^{-1}(0)$. 
By Condition (\ref{nowhere}), we have $\pi_\Sigma^{-1}(0)=E\sqcup \Crit(f)$.

Consider the case $p\in \Crit(f)$.
By Condition (\ref{A_1}), there exists a chart $(V_p;x)$
on which we have $f(x)=x^2+f(p)$. 
Let $h(x)$ be the function on $V_p$ defined by 
$h(x)d x=g^{-1}dg$. 
By Condition (\ref{nowhere}), $h(x)$ is a holomorphic function.
On ${\Delta_\eq(r)}\times V_p$, we have
\begin{align*}
df-\eq g^{-1}dg=(2x-\eq h(x))dx.
\end{align*}
Since $h(x)$ is holomorphic, 
\[\Sigma(f, g)\cap ({\Delta_\eq(r)}\times V_p)=\{(\eq,x)
\in{\Delta_\eq(r)}\times V_p\mid 2x-\eq h(x)=0\}\]
is smooth over ${\Delta_\eq(r)}$ for sufficiently small $r>0$.

Consider the case $p\in E$. 
There exists a coordinate 
$(W_p, y)$ such that $y(p)=0$ and 
$g(y)=y^{n_p}$ for some $n_p\in \Z$ on $W_p$. 
Let $f'(y)$ be the holomorphic function on $W_p$ by defined by $df=f'(y)dy$.
By Condition (\ref{nowhere}), we may assume that $f'(y)$ is nowhere vanishing on $W_p$. 
Then 
\begin{align*}
\Sigma(f, g)\cap ({\Delta_\eq(r)}\times W_p)=\{(\eq, y)\in {\Delta_\eq(r)}\times W_p\mid y f'(y)-n_p\eq =0 \}
\end{align*}
is smooth over ${\Delta_\eq(r)}$ for sufficiently small $r>0$.

Hence we obtain the lemma. 
\end{proof}

We also note that under this assumption, we have the following:
\begin{corollary}
The rank of $\H^1_{\dR}$ over $\O_{S^\circ}$ 
is equal to 
the number of elements of $\pi_{\Sigma}^{-1}(0)=E\sqcup \Crit(f)$. 
\end{corollary}
\begin{proof}
By Theorem \ref{LF}, the rank is the same as the dimension of 
the fiber of $\H^1_{\dR,0,0}$ at $(z, s)=(0,0)$.
The fiber is canonically isomorphic to 
\[H^0(U,\omega_{U}(E)/\O_{U} df). \]
Since the support of $\omega_{U}(E)/\O_{U} df$
is $\pi_\Sigma^{-1}(0)$ (the multiplicity at each point is one), we obtain the lemma.
\end{proof}
Hence we also have:
\begin{corollary}
If $(f,g)$ satisfies $(\ref{A_1})$, $(\ref{nowhere})$ then
the rank of $\L^\Be$ is equal to 
the number of elements in $\pi_{\Sigma}^{-1}(0)=E\sqcup \Crit(f)$. \qed
\end{corollary}
For each point $p$ in $\pi_\Sigma^{-1}(0)$, let
$\nu_{p}\colon{\Delta_\eq(r)} \simeqto \Sigma_p$ denote the section of 
of the projection $\pi_\Sigma$ to the sheet 
$\Sigma_p \subset \Sigma(f, g)\cap({\Delta_\eq(r)} \times U)$ which contains $p$. 
We will identify 
$\Sigma_p$
with its image in $U$ via the projection ${\Delta_\eq(r)}\times U\to U$. 
Then $\nu_p$ will be identified with 
the composition with the projection, i.e. 
it will be considered as a map $\nu_p\colon {\Delta_\eq(r)}\to U$.

We then put $f_p \coloneqq f\circ\nu_{p} \colon {\Delta_\eq(r)} \to \C$
and $g_p\coloneqq g\circ \nu_{p} \colon {\Delta_\eq(r)} \to \C$. 
Let $n_p$ be the  integer such that 
$g_p(\s)$ is of the form $\s^{n_p}h(\s)$ for $h(\s)$ is holomorphic with $h(0)\neq 0$. 
Using these notations, we define the goodness of $(f, g)$:
\begin{definition}\label{good pair}
We will say that the pair  $(f, g)$ is \textit{good} if 
for any two distinct points $p, p'\in\pi_\Sigma^{-1}(0)$
with $(f_p, g_p)\neq (f_{p'}, g_{p'})$ as pairs of germs in $\O_{{\Delta_\eq(r)},0}$,
either
$f_p(0)\neq f_{p'}(0)$ or $n_p\neq n_{p'}$ holds. 
\end{definition}

\begin{example}\label{spec bessel}
Consider the Example \ref{Bessel}.
In this case, we have $E=\emptyset$, 
$\Crit(f)=\{\pm 1\}$, and 
\[\Sigma(f, g)=\{(\s, y)\in  \C\times Y \mid y^2-\s y-1=0\}.\]
We have
$f_{\pm 1}(\s)=\pm \sqrt{\s^2+4}$
and $g_{\pm 1}(\s)=2^{-1}(\eq\pm \sqrt{\s^2+4})$.
\end{example}

\begin{example}\label{spec gamma}
Consider the Example \ref{EXG}.
In this case, we have
$E=E_0=\{0\}$, 
$\Crit(f)=\emptyset$ and 
\[\Sigma(f, g)=\{(\s, y)\in  \C \times U \mid y-\s=0\}.\] 
We have $f_{0}(\s)=g_0(\s)=\s$. 
\end{example}

In the following, we always assume that $(f, g)$
is a good pair.

\subsection{Statement of the main theorem}
The
main result of this paper is the following:
\begin{theorem}\label{main theorem}
Assume that the pair $(f, g)$ is good.
In particular, they satisfy the conditions $(\ref{A_1})$, $(\ref{nowhere})$.
Then, 
$\L^\Be_{\leqslant}$
is a good Stokes filtration on $\L^\Be_{}=\L^\Be(f,g)$. 
Moreover, the exponential factor 
$\Phi(f, g)\coloneqq \Phi(\L^\Be_{}, \L^\Be_{\leqslant})$
is given by 
\begin{align*}
&\Phi(f,g)=\bigsqcup_{p\in \pi_{\Sigma}^{-1}(0)} \Phi_p(f,g),
&
\Phi_p(f, g)\coloneqq \left[ u^{-1}\( \log g_p(v)-\frac{f_p(v)}{v}+2\pi\i\Z\) \right].
\end{align*}
\end{theorem}

The proof of this theorem will be finished in \S \ref{saddle point method}. 
The goodness of $\Phi(f,g)$
directly follows from the assumption
that the pair $(f,g)$
is good in the sense of
Definition \ref{good pair}. 
We firstly 
explain 
what will be actually shown in the proof. 

\begin{definition}
A real number $\phi\in \R$
is called \textit{generic} (with respect to $(f, g)$)
if 
\begin{align*}
\Im(e^{-\i\phi}(f(p)-f(p')))\neq 0 
\end{align*}
for $p, p'\in \pi_\Sigma^{-1}(0)$ with $f(p)\neq f(p')$. 
\end{definition}

Remark that if $\phi$ is generic, then  $\phi+\pi\i n$ ($n\in \Z$)
is also generic. 
For a generic $\phi$, 
we set
\begin{align*}
&V^\phi\coloneqq \left\{(\theta_u,\theta_v)\in T\middle| \cos(\theta_u+\theta_v-\phi)>0 \right\},\\
&V^\phi_{+}\coloneqq \left\{(\theta_u,\theta_v)\in V^\phi\middle| \cos(\theta_v-\phi)>0\right\},\\
&V^\phi_{-}\coloneqq\left\{(\theta_u,\theta_v)\in V^\phi\middle| \cos(\theta_v-\phi)<0\right\}.
\end{align*}
Note that we have
$V^\phi_+\cap T_{\R_+}\neq \emptyset$, $V^\phi_+\cap T_{\R_-}=\emptyset$, 
$V^\phi_-\cap T_{\R_-}\neq \emptyset$, and $V^\phi_-\cap T_{\R_+}=\emptyset$.

We also have
\begin{align*}
&\iota(V^{\phi})=V^{\phi+\pi}, 
&\iota(V^{\phi}_{\pm})=V^{\phi+\pi}_{\mp}.
\end{align*}
Theorem \ref{main theorem}
follows from the following:
\begin{theorem}\label{WTS}
For any generic $\phi$,
we have the 
sections 
\begin{align*}
&\mathfrak{s}_c^\phi\in \L^\Be_{}(V^\phi),&&(c\in \Crit(f))\\
&\mathfrak{s}_{e,\pm}^\phi\in \L^\Be_{}(V^\phi_{\pm}), 
&&(e\in E)
\end{align*}
with the following properties$:$
\begin{enumerate}
\item There are sections 
\begin{align*}
&\varphi_c\in H^0(V^\phi, \Phi_c(f,g))&& (c\in \Crit(f)),\\ 
&\varphi_{e,\pm}\in H^0(V^\phi_{\pm},\Phi_e(f,g)), 
&&(e\in E)
\end{align*}
such that 
\begin{enumerate}
\item
$\mathfrak{s}_c^\phi\in \L_{\leqslant \varphi_c}^\Be(V^\phi)$
for $c\in \Crit(f)$,
$\mathfrak{s}_{e,\pm}^\phi\in \L^\Be_{\leqslant \varphi_{e,\pm}}(V^\phi_{\pm})$
for $e\in E$, and
\item
the induced sections $\gr_{\varphi_c}(\mathfrak{s}_c^\phi)$, 
$\gr_{\varphi_{e,\pm}}(\mathfrak{s}_{e,\pm}^\phi)$ are non-zero. 
\end{enumerate}
\item
Let $\mathfrak{s}_{c,\pm}^\phi$ denote the restrictions of $\mathfrak{s}_{c}^\phi$
to $V^\phi_{\pm}$. 
Then the intersection matrix
\[\(\<\mathfrak{s}^\phi_{p,+}, \mathfrak{s}^{\phi+\pi}_{p',-}\>_{\Be+}\)_{p, p'\in \pi_{\Sigma}^{-1}(0)}\]
is the identity matrix.
\item
The quotient classes 
\begin{align*}
[&\mathfrak{s}^\phi_{e,+}]_{\bm\theta}\in \L_{\bm\theta}^\Be/(\L^\Be)^-_{\bm\theta}
&&\text{for $e\in E_0$, $\bm{\theta}\in V^\phi_+\setminus T_{\R_+}$
and} 
\\
&[\mathfrak{s}^\phi_{e', -}]_{\bm\theta'}\in \L_{\bm\theta'}^\Be/(\L^\Be)^+_{\bm\theta'}
&&\text{for $e'\in E_\infty$, $\bm\theta'\in V^\phi_-\setminus T_{\R_-}$}
\end{align*}
are non-zero,
while $[\mathfrak{s}^\phi_{e,-}]_{\bm\theta'}\in \L_{\bm\theta'}^\Be/(\L^\Be)^+_{\bm\theta'}$ 
and $[\mathfrak{s}^\phi_{e', +}]_{\bm\theta}\in \L_{\bm\theta}^\Be/(\L^\Be)^-_{\bm\theta}$ are zero. 
\end{enumerate}
\end{theorem}
\begin{proof}[Proof of ^^ ^^ Theorem \ref{WTS} $\Rightarrow$ Theorem \ref{main theorem}"]
Let
$\bm{\theta}$ be a point in $T_+\cup T_-$. 
Take generic 
$\phi_+$ and $\phi_-$ such that 
both $V^{\phi_+}_+$ and $V^{\phi_-}_-$
contain $\bm\theta$. 
If we assume Theorem \ref{WTS}, 
we have the sections 
$\mathfrak{s}^{\phi_\pm}_{c}\in \Gamma(V^{\phi_\pm},\L^\Be)$ 
for $c\in \Crit(f)$,
and the sections 
$\mathfrak{s}_{e,\pm}^{\phi_\pm}\in 
\Gamma(V^{\phi_\pm}_\pm,\L^\Be)$ for 
$e\in E$
which satisfy the conditions in Theorem \ref{WTS}. 
By the conditions in Theorem \ref{WTS},  
we have the following frames of $\L^+,\L^-$ and $\L$ on a neighborhood $\nb(\bm\theta)$ of $\bm\theta$
(We put $\L\coloneqq \L^\Be$ for simplicity of notations):
\begin{align}\label{+-}
&\L^+_{|\nb(\bm\theta)}= \bigoplus_{p\in\pi_{\Sigma}^{-1}(0)} \k[q^{\pm 1}]\mathfrak{s}^{\phi_+}_{p,+},
&\L^-_{|\nb(\bm\theta)}=\bigoplus_{p\in\pi_{\Sigma}^{-1}(0)} \k[q^{\pm 1}]\mathfrak{s}^{\phi_-}_{p,-},
\end{align}
\begin{align*}
&\L_{|\nb(\bm\theta)}=\(\bigoplus_{c\in \Crit(f)}\k[q^{\pm 1}]\mathfrak{s}^{\phi_+}_c\)
\oplus \(\bigoplus_{e\in E_0}\k[q^{\pm 1}]\mathfrak{s}^{\phi_+}_{e,+}\)
\oplus\(\bigoplus_{e'\in E_\infty}\k[q^{\pm 1}]\mathfrak{s}^{\phi_-}_{e',-}\).
\end{align*}
Using these frames, 
we obtain the conditions (1) and (2) in Definition \ref{Def q-Stokes} on $\nb(\bm\theta)$.
(In the case $\bm\theta\in T_{\R_+}\cup T_{\R_-}$, 
we only need to consider (\ref{+-}) to see (1) and (2) in Definition \ref{Def q-Stokes}.) 
We also obtain (3) in Definition \ref{Def q-Stokes} from (3) in Theorem \ref{WTS}. 
\end{proof}

\subsection{Families of paths}\label{properties of gamma}
For  $\phi\in\R$
and $\delta>0$,
we set $\Delta^*_\eq(\delta)\coloneqq \Delta_\eq(\delta)\setminus\{0\}$ and
\begin{align*}
\cal{S}_\pm^\phi(\delta)\coloneqq \{\eq\in \Delta^*_\eq(\delta)\mid \pm \Re(e^{-\i\phi}\eq)>0\}.
\end{align*}
Note that $\cal{S}_+^{\phi+\pi}(\delta)=\cal{S}_-^\phi(\delta)$ and 
$\cal{S}^{\phi+\pi}_-(\delta)=\cal{S}_+^{\phi}(\delta)$.

For a sufficiently small $\delta>0$ and a generic $\phi$, we shall construct 
the families of paths
\begin{align*}
&\gamma^\phi_c\colon \Delta_\eq(\delta)\times \R\to \Delta_\eq(\delta)\times Y,
&&(\eq, t)\mapsto \(\eq,\gamma^\phi_{c,\eq}(t)\)\\
&\gamma^\phi_{e,\pm}\colon \cal{S}_\pm^\phi(\delta)\times \R\to
 \cal{S}_\pm^\phi(\delta)\times Y,
&&(\eq, t)\mapsto \(\eq,\gamma^\phi_{c,\pm,\eq}(t)\)
\end{align*}
 indexed by $c\in\Crit(f)$, $e\in E$, and $\pm$
with the following properties:
\begin{enumerate}
\item[$(\gamma1)$] $\gamma^\phi_{c,\eq}(0)=\nu_c(\eq)$, and $\gamma^\phi_{e,\pm,\eq}(0)=\nu_e(\eq)$.  
\item[$(\gamma2)$] For a fixed $\eq\in \cal{S}_\pm^\phi(\delta)$, 
the path $\gamma^\phi_{c,\eq}$ $(c\in \Crit(f))$
does not intersect with $\gamma^{\phi+\pi}_{c',\eq}$ $(c'\in \Crit(f), c'\neq c)$
or $\gamma^{\phi+\pi}_{e',\pm,\eq}$ $(e'\in E)$.
Similarly, $\gamma^\phi_{e,\pm,\eq}$ $(e\in E)$
does not intersect with $\gamma^{\phi+\pi}_{c',\eq}$ or
$\gamma^{\phi+\pi}_{e',\pm,\eq}$ ($e'\neq e\in E$).
 
\item[$(\gamma3)$] For $c\in \Crit(f)$, 
take a branch of $\log g$ at $c$. 
For $e\in E$, 
take a branch of $\log g$
on $\{\nu_e(\eq)\in Y\mid \eq\in\cal{S}_\pm^\phi(\delta)\}$ (see \S \ref{Around E} below).
Then the function $\Re(e^{-\i\theta}(f-\eq \log g))$
monotonically increases along $\gamma^\phi_{c,\eq}$
and $\gamma^\phi_{e,\pm,\eq}$
as $|t|$ increases. 
\end{enumerate}
The construction will be given in \S \ref{construction of gamma}
after the preparations in \S \ref{dumbbell} and \S \ref{foliation}. 
Then we will use these paths 
to construct the sections in 
Theorem \ref{WTS} after some modifications.

\subsection{Preliminaries}\label{dumbbell}
Let $\phi\in \R$ be generic.
For $c\in \Crit(f)$,
let $L^\phi_c\colon\R\to Y$, $t\mapsto L^\phi_c(t)$
be a path 
which satisfy the following conditions:
\begin{itemize}
\item $L_c^\phi(0)=c$, and
\item
$f\circ L_c^\phi(t)=f(c)+e^{\i\phi}t^2$ for $t\in\R$. 
\end{itemize}
These conditions determine $L^\phi_c$ up to the re-parametrization $t\mapsto -t$. 
For $e\in E$, 
let $L^\phi_e\colon \R_{\geq 0}\to Y$, $t\mapsto L^\phi_e(t)$
be a path with the following conditions:
\begin{itemize}
\item $L_e^\phi(0)=e$, and 
\item $f\circ L_e^\phi(t)=f(e)+e^{\i\phi}t$ for $t\geq 0$. 
\end{itemize}
Since $\phi$ is generic, any distinct two of 
$L_c^{\phi}$ $(c\in \Crit(f))$ and $L_e^{\phi}$ $(e\in E)$ do
not intersect with each other.
Moreover, if $p, p'\in \Crit(f)\cup E$ are different (i.e. $p\neq p')$, then
$L_p^\phi$ and $L_{p'}^{\phi+\pi}$ do not intersect with each other. 

There exists the following limits in $X$:
\begin{align*}
&p^\phi_{c,\pm}\coloneqq \lim_{t\to\pm \infty} L^\phi_c(t),
&p_{e}^\phi\coloneqq \lim_{t\to \infty} L_e^\phi(t).
\end{align*}

Let $\Delta(p_{c,\pm}^\phi)$
be a coordinate disk centered at $p_{c,\pm}^\phi$.
Take $R>0$
with 
\begin{align*}
\text{
$L^\phi_c([-R,R])\cap \Delta(p_{c,+}^\phi)\neq \emptyset$
and $L^\phi_c([-R,R])\cap \Delta(p_{c,-}^\phi)\neq \emptyset$.}
\end{align*}
Then we take a tuber neighborhood 
$\bm{L}^\phi_c(R)$ of $L^\phi_c([-R,R])$
which is relatively compact in $Y$. 
We take a  disk $\Delta(c)\subset \bm{L}^\phi_c(R)$
centered at $c$.

We can take the tuber neighborhood $\bm{L}^\phi_c(R)$
and the disk $\Delta(c)$ so that 
the following condition hold:
For any $x\in \bm{L}^\phi_c(R)\setminus \Delta(c)$, 
the path $\ell_x(t)$ given by 
\begin{align*}
\text{$\ell_x(0)=x$, 
and $f\circ\ell_x(t)=f(x)+e^{\i\phi}t$}
\end{align*}
intersects with the boundary of 
$ \bm{L}^\phi_c(R)\setminus \(\Delta(c)\cup \Delta(p^\phi_+)\cup \Delta(p^\phi_-)\)$
only at the boundary of $\Delta(c)\cup \Delta(p_{c,+}^\phi)\cup \Delta(p_{c-}^\phi)$.

Similarly, 
let $\Delta(p_e^\phi)$ and $\Delta(e)$ be coordinate disks centered at
$p_e^\phi$ and $e$, respectively. 
Take $r,R>0$ such that 
$L^\phi_e([r,R])\cap \Delta(e)\neq\emptyset$ and $L^\phi_e(r,R)\cap \Delta(p_{e}^\phi)\neq \emptyset$.
Let $\bm{L}^\phi_e(r,R)$
be a relatively compact tuber neighborhood of $L_e([r,R])$
with a similar property as $\bm{L}_c^\phi(R)$,
i.e.
for any $y\in \bm{L}_e(r,R)$, the path $\ell_y^\phi$
defined by the same condition
intersects with the boundary of $\bm{L}^\phi_e(r,R)\setminus \(\Delta(e)\cup\Delta(p^\phi_e)\)$
only at the boundary of $\Delta(e)\cup\Delta(p^\phi_e)$.

Since $\bm{L}^\phi_e(r,R)$ and $\bm{L}_c^\phi(R)$
are assumed to be 
relatively compact, 
we have:
\begin{lemma}\label{GDG}
For given $R>0,r>0$, 
fix the branches of $\log g$ on $\bm{L}^\phi_e(r,R)\setminus \Delta(e)$
and on $\bm{L}_c^\phi(R)$.
There exists $\delta>0$
such that 
if $|\eq|<\delta$, 
then
the following property holds: 
for any $x\in \bm{L}^\phi_c(R)\setminus \Delta(c)$, 
and any $y\in \bm{L}_e^\phi(r,R)\setminus \Delta(e)$, 
the functions
\begin{align*}
&\Re\(e^{-\i\phi}(f-\eq \log g)\)\circ \ell_x^\phi (t),
&\Re\(e^{-\i\phi}(f-\eq\log g)\)\circ\ell_y^\phi (t)
\end{align*}
monotonically increases as $t$ increase,
as long as $\ell_x^\phi(t)$ and $\ell_y^\phi(t)$
are in the closures of 
 $\bm{L}^\phi_c(R)\setminus\( \Delta(c)\cup\Delta(p^\phi_{c,+})\cup \Delta(p^\phi_{c,-})\)$
 and $\bm{L}_e^\phi(r,R)\setminus \(\Delta(e)\cup \Delta(p_e^\phi)\)$, 
 respectively. \qed
\end{lemma}

In the following discussion, 
when we consider a coordinate function $z$ on 
a disk $\Delta$ in $X$, we always assume 
that $z$ is defined on a neighborhood of the closure of $\Delta$.

\subsection{Foliations of Abelian differentials around the singular points}\label{foliation}
For each $\eq\in\C$, we set
\begin{align*}
\alpha_\eq\coloneqq df-\eq g^{-1}dg. 
\end{align*}
We shall describe the straight arcs of 
the Abelian differential $\alpha_\eq$
around the singular points of 
$\alpha_\eq$
when $|\eq|>0$ is sufficiently small. 

Here, by the term ^^ straight arc', 
we mean a path along which the function 
$\Im (e^{-\i\phi}(f-\eq\log g))$ is constant 
and $\Re(e^{-\i\phi}(f-\eq \log g))$ increases 
for some $\phi$. 
When we specify $\phi$, we also use the term $\phi$-arc. 
The pole of $\alpha_\eq$
is $D=P\sqcup E$, 
and the zero of 
$\alpha_\eq$
is $\{\nu_p(\eq)\mid p\in \pi_{\Sigma}^{-1}(0)=E\sqcup \Crit(f)\}$
by definition. 
\subsubsection{Around $P$} 
A maximal straight arc (in a direction) is called
\textit{a trajectory}. 
\begin{lemma}\label{trap}
Fix a positive real number $\delta>0$. 
Then there exists an open neighborhood
$\Delta(p)$ of $p$
such that any trajectory of $\alpha_\eq$
which intersects $\Delta(p)$
tends to $p$
for any $\eq\in \Delta_\eq(\delta)$. 
\end{lemma}
\begin{proof}The residue 
of $\alpha_\eq$
at $p$
is given by $\eq$ times constant
and hence bounded on $\Delta_\eq(r)$. 
Then we obtain the claim 
by the proof
of \cite[Theorem 7.4]{Strebel}.
\end{proof}

\subsubsection{Around $\Crit(f)$}\label{around c}
\begin{lemma}\label{w_c}
There exist $\delta>0$,
a neighborhood 
$\Delta(c)$ of $c$,
and 
an open embedding 
$w\colon \Delta_\eq(\delta)\times \Delta(c)\to \Delta_\eq(\delta)\times \C$
over $\Delta_\eq(\delta)$
such that 
for each $\eq\in \Delta_\eq(\delta)$, 
$w_\eq\coloneqq w_{|\{\eq\}\times \Delta(c)}$ is a coordinate on $\Delta(c)$
centered at $\nu_e(\eq)$
with 
$\alpha_\eq=2w_\eq d w_\eq$. 
\end{lemma}
\begin{proof}
We can take a neighborhood $\nb(c)$
of $c$ so that  $\log g$ is univalent on $\nb(c)$. 
Take a coordinate $z=z_c$ on $\nb(c)$ centered at $c$. 
Then we set 
\[F(\eq,z)\coloneqq f(z)-\eq \log g(z).\]
Note that we have $dF_\eq=\alpha_\eq$ on $\nb(c)$ for each 
$\eq$, where $F_\eq(z)=F(\eq,z)$. 
Then let
$G(\eq,z)$ 
be the function defined by
the following equation:
\begin{align*}
F(z,\eq)-F(\nu_c(\eq),\eq)=(z-\nu_c(\eq))^2G(\eq,z).
\end{align*}
where we denote $z\circ\nu_c(\eq)$ by $\nu_c(\eq)$ for simplicity. 
Since $G(\eq,\nu_c(\eq))= \partial_z^2F(\eq,\nu_c(\eq))\neq 0$, 
taking $\delta>0$ and $\Delta(c)\subset \nb(c)$ small enough, 
we may assume that 
$G(\eq, z)$ is nowhere vanishing on $\Delta_\eq(r)\times \Delta(c)$. 
Then we take
\[w_\eq(z)\coloneqq (z-\nu_c(\eq))\sqrt{G(\eq, z)}\]
and $w(\eq,z)=(\eq, w_\eq(z))$, which satisfies $\alpha_\eq=w_\eq d w_\eq$. 
\end{proof}
\begin{remark}
There is an ambiguity of the signature of $w_\eq$.
\end{remark}

\subsubsection{Around $E$}\label{Around E}
Take a disk $\Delta(e)$
with a coordinate $z=z_e$ centered at $e$
such that $f(z)=z+f(e)$ on $\Delta(e)$. 
We have $g(z)=z^{n_e}h(z)$
where $n_e\in \Z\setminus \{0\}$ and 
$h$ is nowhere vanishing holomorphic function on $\Delta(e)$
(taking $\Delta(e)$ smaller if necessary).
In this coordinate, 
we have
\[\alpha_\eq=d z-\eq n_e \frac{dz}{z}-\eq d\log h.\]
and hence 
$\nu_e(\eq)=z(\nu_e(\eq))$ is characterized by the 
following equation:
\begin{align}\label{nu_e}
\(1-\eq\frac{1}{h(\nu_e(\eq))}\frac{dh}{dz}(\nu_e(\eq))\)\nu_e(\eq)=n_e\eq.
\end{align}
Roughly speaking, 
(\ref{nu_e}) implies that 
$\nu_e(\eq)$ is close to $n_e\eq$
when $|\eq|$ is small. 
Let $\cal{S}\subset \C^*$ be a
proper sector, i.e. subset of the form
$\cal{S}=\{\eq\in \C\mid a<\arg(\eq)<b\}$
with $|a-b|<2\pi$. We set $\cal{S}(\delta)\coloneqq \cal{S}\cap \Delta_\eq(\delta)$.
Such intersections are also called proper sectors. 
The equation (\ref{nu_e}) implies the following:
\begin{lemma}\label{sector}
For any proper sector $\cal{S}$, 
there exists $\delta>0$ such that 
$\nu_e(\cal{S}(\delta))$ is contained in a proper sector in $\Delta(e)$. \qed
\end{lemma}
In a similar way as Lemma \ref{w_c}, 
we obtain:
\begin{lemma}\label{w_e}
Let $\cal{S}$ be a proper sector in $\C_\eq$. 
Then, there exist $\delta>0$, 
an open neighborhood 
$\mathcal{U}_{e}\subset \cal{S}(\delta)\times \Delta(e)$
of the graph of $\nu_e$ on a proper sector $ \cal{S}(\delta)\subset \Delta_{\eq}^*(\delta)$, 
and an open embedding 
$w\colon \cal{U}_{e}^\phi\to \cal{S}(\delta)\times \C$
over $\cal{S}(\delta)$
such that 
we have 
$\alpha_\eq=w_\eq d w_\eq$ on the fiber $\cal{U}_{e,\eq}=\cal{U}_e\times_{\Delta^*_\eq(\delta)}\{\eq\}$. 
\end{lemma}
\begin{proof}
By Lemma \ref{sector}, 
on a neighborhood of the graph of $\nu_e$ on $\cal{S}(\delta)$, 
can take a branch of $\log g$. 
Then the remaining proof is similar to that of 
Lemma \ref{w_c}
\end{proof}

For each $\eq\in\Delta^*_\eq(\delta)$, 
let  $\varphi_\eq\colon \Delta(e)\to\C$
be the holomorphic function
 defined by 
\[\varphi_\eq(z)\coloneqq z\exp\(-\frac{z}{n_e\eq}+\frac{1}{n_e}\int_0^z d\log h\).\]
Then we have
\[-n_e\eq d\varphi_\eq=\varphi_\eq \alpha_\eq.\]
By this formula,
$\varphi_\eq$
is
branched at $\nu_e(\eq)$
with ramification index two. 
We also have 
$\alpha_\eq=\varphi^*(-n_e\eq \zeta^{-1}d\zeta)$. 
Hence the straight arcs of $\alpha_\eq$
can be seen as 
the pull back of straight arcs of $-n_e\eq \zeta^{-1}d\zeta$. 
In particular, 
$\{z\in \Delta(e)\mid |\varphi_\eq(z)|=|\varphi_\eq(\nu_e(\eq))|\}$
defines a union of straight arcs in $\Delta(e)$. 

We shall consider the pull back
$V_\eq\coloneqq \varphi_\eq^{-1}(\{\zeta\in \C\mid |\zeta|<|\varphi_\eq(\nu_e(\eq))|\})$.
It has two connected components $V^0_\eq$ and $V^1_\eq$, 
where $e\in V^0_\eq$. 
\begin{lemma}\label{||=||}
If $\delta>0$ is sufficiently small and $\eq\in \Delta^*_\eq(\delta)$, 
then the closure of $V^0_\eq$ in $X$ is in the interior of $\Delta(e)$. 
\end{lemma}
\begin{proof}
Take $R_e>0$ so that 
$\Delta(e)=\{z\mid |z|<R_e\}$. 
On the one hand, 
the equation
\begin{align*}
&|z|= R_e, &|\varphi_\eq(z)|=|\varphi_\eq(\nu_e(\eq))|
\end{align*}
has two solutions for each $\eq\in \Delta^*_\eq(\delta)$ if $\delta>0$ is sufficiently small.
On the other hand, the closure $\overline{V}^1_\eq$ of $V^1_\eq$
intersect with $\del \Delta(e)$. 
Hence we have
$\# \del(\overline{V}^1_\eq\cap \del\Delta(e))\geq2$. 
It follows that the closure $\overline{V}^0_\eq$
does not intersect with $\del \Delta(e)$. 
\end{proof}

Let $\gamma^\phi_{\eq}\colon (-\epsilon',\epsilon')\to \cal{U}_{e,\eq}$
($\eq\in \cal{S}$ for a proper sector $\cal{S}$)
be a path
determined by the conditions
(i) $\gamma^\phi_{\eq}(0)=\nu_e(\eq)$,
and  
(ii) $w_\eq\circ\gamma_{\eq}^\phi=e^{\i\phi/2}t$ 
for $\phi\in \R$ and sufficiently small $\epsilon'>0$. 

In the case $\Re(e^{-\i\phi}n_e\eq)>0$,
the composition
$\varphi_\eq\circ\gamma^\phi_\eq$
satisfies the inequality
\begin{align*}
|\varphi_\eq(\gamma_\eq^\phi(t))|\leq  |\varphi(\nu_e(\eq))|,
\end{align*}
where the equality holds if and only if $t=0$. 
In this case, taking the signature of $w$ properly, by Lemma \ref{||=||},
the path $\gamma^\phi_\eq$ extends
to the path $\gamma^\phi_\eq\colon (-\infty, \epsilon]\to X$, ($\epsilon>\epsilon'$)
such that
\begin{itemize}
\item the path is the $\phi$-arc of 
$\alpha_\eq$ on $(-\infty,\epsilon)\setminus \{0\}$, and that
\item the image $\gamma^\phi_\eq(-\infty, \epsilon)$ is in $\Delta(e)$
with $\gamma^\phi_\eq(\epsilon)\in\del\Delta(e)$. 
\end{itemize}
We have 
$\lim_{t\to -\infty} \gamma(t)=e$, 
and $\Re\(e^{-\i\phi}z(\gamma^\phi_\eq(\epsilon))\)>0$
for $\eq\in \Delta_\eq(\delta)$ with $\delta>0$ sufficiently small. 

In the case $\Re(e^{-\i\phi}n_e\eq)>0$,
the composition
$\varphi_\eq\circ\gamma^\phi_\eq$
satisfies the inequality
\begin{align*}
|\varphi_\eq(\gamma_\eq^\phi(t))|\geq  |\varphi(\nu_e(\eq))|,
\end{align*}
where the equality holds if and only if $t=0$.
In this case, 
the path $\gamma^\phi_\eq$ extends
to the path $\gamma^\phi_\eq\colon [-\epsilon_0, \epsilon_1]\to X$, ($\epsilon_0,\epsilon_1>\epsilon'$)
such that
\begin{itemize}
\item the path is the $\phi$-arc of $\alpha_\eq$ on $(-\epsilon_0,\epsilon_1)\setminus\{0\}$, and that
\item the image $\gamma^\phi_\eq((-\epsilon_0,\epsilon_1))$ is in $\Delta(e)$ with 
$\gamma^\phi_\eq(\epsilon_i)\in \del\Delta(e)$, $(i=0,1)$.
\end{itemize}
For sufficiently small $\delta>0$, 
$\Re\(e^{-\i\phi}z(\gamma^\phi_\eq(\epsilon_i))\)>0$
($i=0,1$ and $\eq\in\Delta_\eq(\delta)$). 

\subsection{Construction of family of paths}\label{construction of gamma}
Let $\phi\in \R$ be generic. 
We shall construct
the paths explained in \S \ref{properties of gamma}. 
\subsubsection{Construction of $\gamma^\phi_c$}\label{Cc}
Take a family of coordinates 
$w\colon \Delta_\eq(\delta)\times \Delta(c)\to \Delta_\eq(\delta)\times\C$ 
which satisfies the conditions Lemma \ref{w_c}. 
Then, 
there exists 
a family of closed intervals 
$I_\eq=[-\epsilon_\eq^-,\epsilon_\eq^+]$
($\eq\in \Delta_\eq(\delta)$, $\epsilon_\eq^\pm>0$)
and 
a family of paths
$\gamma^\phi_{c,\eq}\colon I_\eq\to Y$
such that $\gamma_{c,0}^\phi=L^\phi_c$ (see \S \ref{dumbbell}) restricted to $[\epsilon^-_0, \epsilon^+_0]$, 
\begin{itemize}
\item $\gamma^\phi_{c,\eq}((-\epsilon_\eq^-,\epsilon_\eq^+))\subset \Delta(c)$, $\gamma^\phi_{c,\eq}(0)=\nu_c(\eq)$, 
\item $w_\eq\circ\gamma^\phi_{c,\eq}(t)=e^{\i\phi/2}t$ for $t\in (-\epsilon_\eq-,\epsilon_\eq^+)$, and that 
\item $\gamma^\phi_{c,\eq}(\epsilon_\eq^\pm)\in \del\Delta(c)$. 
\end{itemize}
Take $\bm{L}_c^\phi(R)$ as in \S \ref{dumbbell}. 
Then we can extend $\gamma^\phi_{c,\eq}$
connecting it
with $\ell^\phi_{x^+_\eq}$ and $\ell^{\phi+\pi}_{x^{-}_\eq}$ 
where we put $x^\pm_\eq\coloneqq \gamma^\phi_{c,\eq}(\epsilon_\eq^\pm)$
(we take linear re-parametrization if necessary). 
Then, we can take
$R^\pm_{c,\eq}>0$ smoothly depending on $\eq$
such that
$\gamma^\phi_{c,\eq}(\pm R^\pm_{c,\eq})\in \Delta(p_{c,\pm}^\phi)$. 
We then extend the path 
so that 
it is the $\phi$-arc of $\alpha_\eq$ 
for $t>R^+_{c,\eq}$, and
$(\phi+\pi)$-arc of $\alpha_\eq$
for $t<-R^-_{c,\eq}$.  
By Lemma \ref{trap}, 
taking $\delta>0$ 
sufficiently small, 
we have
\begin{align*}
\lim_{t\to \pm \infty}\gamma^\phi_{c,\eq}(t)=p^\phi_{c,\pm}
\end{align*}
for each $\eq\in \Delta_\eq(\delta)$.

In summary,
we take 
the family of paths $\gamma^\phi_{c,\eq}\colon \R\to Y$
by the following conditions:
\begin{enumerate}
\item On $(-\epsilon^{-}_{\eq},\epsilon^+_{\eq})$, we have $w_\eq(\gamma_\eq^\phi(t))=e^{\i\phi/2}t$.
\item On  $[\epsilon_\eq^+,R_{c,\eq}^+]$ (resp. $[-R^-_{c,\eq}, -\epsilon^-_\eq]$), 
$\gamma^\phi_{c,\eq}$ is  $\phi$-arc (resp. $(\phi+\pi)$-arc) of $\alpha_0=df$. 
\item On  $(R_{c,\eq}^+,\infty)$ (resp. $(-\infty, -R_{c,\eq}^-)$ ), 
$\gamma^\phi_{c,\eq}$ is  $\phi$-arc (resp. $(\phi+\pi)$-arc) of $\alpha$.
\end{enumerate}

\begin{lemma}\label{case of c}
The family of paths 
$\gamma^\phi_c\colon \Delta_\eq(\delta)\times \R\to \Delta_\eq(\delta)\times Y$, 
$(\eq, t)\mapsto (\eq, \gamma^\phi_{c,\eq}(t))$
defined above satisfies conditions $(\gamma 1)$ and $(\gamma 3)$
in \S $\ref{properties of gamma}$. 
\end{lemma}
\begin{proof}
The condition $(\gamma 1)$ follows directly from the definition. 
The condition $(\gamma 3)$ follows from the construction 
and Lemma \ref{GDG}. 
\end{proof}

\subsubsection{Construction of $\gamma^\phi_{e,\pm}$}\label{construct gamma e}
If
$e\in E_0$,
the integer $n_e$ is positive. 
Hence we have
$ \Re(e^{-\i\phi}n_e\eq)>0$ on $\cal{S}^\phi_+(\delta)$
and $-\Re(e^{-\i\phi}n_e\eq)>0$ on $\cal{S}^\phi_-(\delta)$. 
Then we put 
$\gamma^\phi_{e,+,\eq}(t)=\gamma^\phi_{\eq}(t)$
for $t\in (-\infty,\epsilon)$, $\eq\in \cal{S}^\phi_{+}(\delta)$
and $\gamma^\phi_{e,-,\eq}(t)=\gamma^\phi_{\eq}(t)$
for $t\in [\epsilon_0,\epsilon_1]$, $\eq\in \cal{S}_-^\phi(\delta)$. 

Similarly to the case of 
$\gamma^\phi_c$, 
we extend 
$\gamma^\phi_{e,+,\eq}$
to $(-\infty, R_{e,+,\eq}]$
so that $\gamma^\phi_{e,+,\eq}$
is the $\phi$-arc of $\alpha_0=df$ on $(\epsilon, R_{e,+,\eq})$
(here, we note that we may assume $\gamma^\phi_{e,+}((\epsilon, R_{e,+,\eq}))\subset \bm{L}^\phi_e(r, R)$), 
and $\gamma^\phi_{e,+}(R_{e,+,\eq})\in \Delta(p_{e})$. 
Then on $(R_{e,+,\eq},\infty)$, we extend 
$\gamma^\phi_{e,+}$ as the $\phi$-trajectory ray of $\alpha_\eq$. 
Then, $\underset{t\to \infty}{\lim}\gamma^\phi_e(t)=p_e^\phi$
by Lemma \ref{trap}.

Similarly, 
we extend $\gamma^\phi_{e,-,\eq}$
to $[-R^-_{e,-,\eq}, R^+_{e,+,\eq}]$
so that $\gamma^\phi_{e,-,\eq}$
is $\phi$-arc (resp. $\phi+\pi$-arc) of $\alpha_0=df$
on $(\epsilon_1, R^{+}_{e,-,\eq})$
(resp. $(-R_{e,-,\eq}^-,-\epsilon_-)$) 
and $\gamma^\phi_{e,-,\eq}(\pm R_{e,-,\eq}^\pm)\in \Delta(e)$. 
Then, 
we extend it to $\R$ 
by considering the $\phi$-arc and $(\phi+\pi)$-arcs on $t>R^+_{e,-,\eq}$ and 
$t<-R^-_{e,-,\eq}$, respectively. 
We also have
$\underset{t\to\pm \infty}{\lim}\gamma^\phi_{e,-,\eq}(t)=p_e^\phi. $

If $e'\in E_\infty$, 
we have $\mp\Re(e^{-\i\phi}n_{e'} \eq)>0$
for $\eq\in\cal{S}^\phi_{\pm}(\delta)$. 
Hence we can construct $\gamma^\phi_{e',\pm}$
by the same way as $\gamma^\phi_{e,\mp}$ $(e\in E_0)$. 

\begin{lemma}
The families of paths $\gamma^\phi_{c}$ $(c\in \Crit(f))$, $\gamma^\phi_{e,\pm}$ $(e\in E)$
satisfy the conditions $(\gamma1)$, $(\gamma2)$, and $(\gamma3)$ in \S $\ref{properties of gamma}$. 
\end{lemma}
\begin{proof}
The conditions $(\gamma1)$ and $(\gamma3)$ for $\gamma^\phi_c$ is
proved in Lemma \ref{case of c}, and the same conditions 
for $\gamma^\phi_{e,\pm}$ can also be proved in the same way. 
The condition ($\gamma 2$)
follows from the fact that 
for generic $\phi$, 
we can take $\bm{L}^\phi_c(R_c)$, $\bm{L}_{e}^\phi(r_e,R_e)$,
and $\bm{L}_e^{\phi+\pi}(r'_e,R'_e)$
so that they don't have nontrivial intersections
and the fact that 
on $\Delta(p)$, the paths are taken to be 
the $\phi$ (or, $\phi+\pi$)-arcs of $\alpha_\eq$. 
\end{proof}

\subsection{End of the proof of the Theorem
\ref{WTS}}\label{saddle point method}
We shall construct
the sections 
$\mathfrak{s}^\phi_c$ for $c\in \Crit(f)$
and $\mathfrak{s}^\phi_{e,\pm}$ for $e\in E$
in Theorem \ref{WTS}, i.e. 
finish the proof.
 
\subsubsection{Construction of $\mathfrak{s}^\phi_c$}\label{construct s_c}
Let 
$\gamma_c^\phi$ be the path constructed in \S \ref{Cc}. 
Put $\overline{\R}\coloneqq \R\cup\{\pm \infty\}$, 
which is isomorphic to the 1-simplex $[0,1]$. 
Then,
we can extend it to the family of paths 
$\widetilde{\gamma}^\phi_c\colon \Delta_\eq(\delta)\times \overline{\R}\to \Delta_\eq(\delta)\times \widetilde{X}$. 
For $(\tw,\eq)\in S^\circ$, 
set  \[P_\tw^\rd\coloneqq \{(\tw,\eq)\}\times_{S^\circ}\cal{P}^\rd\subset \widetilde{X}.\] 
Then we can easily see that 
if $\Re(e^{-\i\phi}\tw)>0$, 
then $\widetilde{\gamma}^\phi_{c,\eq}(\pm \infty)\in P^\rd_\tw$. 
Hence,
extending $\widetilde{\gamma}^\phi_c$
to $\C^*_\tw\times \Delta_\eq(\delta)$ trivially, 
and taking branch of $\log g$ on $\nu_e(\Delta_\eq(\delta))$, 
we obtain a section 
\begin{align*}
{\mathfrak{s}}^\phi_c\coloneqq \widetilde{\gamma}_c^\phi\otimes e^{-f/\tw}g^{\tw/\eq}\bm{e}
\end{align*}
of $\H^\rd_{1}$, which can also be seen as a section in $\Gamma(V^\phi,\L^\Be)$. 

\begin{lemma}\label{c-case}
The section $\mathfrak{s}^\phi_c$ $(c\in \Crit(f))$ 
constructed above satisfies the condition $(1)$ in Theorem \ref{WTS}. 
\end{lemma}
\begin{proof}
Let $\omega$ be a section of $\pi_{\X*}\j^*\M\otimes \Omega^1_{X}$ 
which represents a section $[\omega]$ of $\j^*\H^1_{\dR,\bm{0}}$. 
Then we have 
\begin{align*}
\<\mathfrak{s}_{c}^\phi, [\omega]\>_{\rm Per}=\int_{\gamma^\phi_{c,\eq}}e^{-f/\tw}g^{\eq/\tw}\omega.
\end{align*}
Then, by the condition $(\gamma 3)$, the description in \S \ref{around c}, and the standard arguments
of saddle point method (see e.g  \cite{AGV}, \cite{Mochizuki}), 
we obtain that
 \[\exp(-\varphi_c(\tw,\eq))\<\mathfrak{s}_{c}^\phi, [\omega]\>_{\rm Per}\]
 is of moderate growth (and not of rapid decay)
when $(\tw,\eq)\to (0,0)$ in the neighborhood of $V^\phi$ in $\widetilde{B}$
where $\varphi_c(\tw,\eq)$ correspond to the branch of $\log g$ 
fixed when $\mathfrak{s}^\phi_c$ is defined. 
This implies the condition (1) of Theorem \ref{WTS} by Lemma \ref{period criterion}. 
\end{proof}

\subsubsection{Construction of $\mathfrak{s}_{e,-}^\phi$ $(e\in E_0)$ and $\mathfrak{s}^\phi_{e',+}$ $(e'\in E_\infty)$}
\label{construct s_e 1}
For $e\in E_0$ let $\gamma^\phi_{e,-}$ be a 
path constructed in \S \ref{construct gamma e}. 
For $e'\in E_\infty$
In both cases, the paths are extended  to 
the families of paths
\begin{align*}
&\widetilde{\gamma}^\phi_{e,-}\colon \cal{S}^\phi_-(\delta)\times \overline{\R}\to\cal{S}^\phi_-\times \widetilde{X},
&\widetilde{\gamma}^\phi_{e',+}\colon \cal{S}^\phi_+(\delta)\times \overline{\R}\to \cal{S}^\phi_+\times \widetilde{X} .
\end{align*}
Similarly as \S \ref{construct s_c}, 
the end points $\widetilde{\gamma}^\phi_{e,-}(\pm\infty)$
are in $P^\rd_{\tw}$ if $\Re(e^{-\i\phi}\tw)>0$. 
By Lemma \ref{sector}, 
we can take a branch of $\log g$ on $\nu_e(\cal{S}^\phi_-(\delta))$
and on $\nu_{e'}(\cal{S}^\phi_+(\delta))$. 
Then we obtain the sections
\begin{align*}
&\mathfrak{s}^\phi_{e,-}\coloneqq \widetilde{\gamma}_{e,-}^\phi\otimes e^{-f/\tw}g^{\eq/\tw}\bm{e},
&\mathfrak{s}^\phi_{e',+}\coloneqq \widetilde{\gamma}_{e',+}^\phi\otimes e^{-f/\tw}g^{\eq/\tw}\bm{e}.
\end{align*}
\begin{lemma}
The sections $\mathfrak{s}^\phi_{e,-}$ 
and $\mathfrak{s}^\phi_{e',+}$ satisfy
the conditions $(1)$ and $(3)$ in Theorem \ref{WTS}. 
\end{lemma}
\begin{proof}
The proof of the condition (1) is similar to that of Lemma \ref{c-case}. 
The condition (3) is trivial by the construction. 
\end{proof}

\subsubsection{Construction of $\mathfrak{s}_{e,+}^\phi$ $(e\in E_0)$ and $\mathfrak{s}^\phi_{e',-}$ $(e'\in E_\infty)$}
For $e\in E_0$
let $\gamma^\phi_{e,+}$
be the family of paths constructed in \S \ref{construct gamma e}.
Then,
the limit ${\gamma}^\phi_{e,+}(\infty)\coloneqq \lim_{t\to \infty}\gamma$ exists in $\widetilde{X}$ 
and similarly to \S \ref{construct s_c} and \S \ref{construct s_e 1}
the limit ${\gamma}^\phi_{e,+}(\infty)$ is in $P^\rd_\tw$ if $\Re(e^{-\i\phi}\tw)>0$. 
However, 
the limit $\lim_{t\to-\infty}\gamma^{\phi}_{e,+,\eq}(t)$ does not exist (in general) in $\widetilde{X}$. 
Hence we modify the path 
by
taking $\widetilde{\gamma}^\phi_{e,+,\eq}\colon \overline{\R}\to \widetilde{X}$
as follows:
\begin{itemize}
\item $\widetilde{\gamma}^\phi_{e,+,\eq}(t)=\gamma^{\phi}_{e,+,\eq}(t)$ for $t\in [-\epsilon', \infty]$ 
(see \S \ref{Around E} for the definition of $\epsilon'$). 
\item $\widetilde{\gamma}^\phi_{e,+,\eq}([-\infty,\epsilon'])$ is a line segment which connects 
$\widetilde{\gamma}^\phi_{e,+,\eq}(\epsilon')$
and a point in $\varpi_{X}^{-1}(e)$ of constant argument with respect to the coordinate $z=z_e$. 
\end{itemize}
In a similar way as in \S \ref{construct s_e 1}, 
we obtain a section 
\begin{align*}
\mathfrak{s}^\phi_{e,+}\coloneqq \widetilde{\gamma}^\phi_{e,+}\otimes e^{-f/\tw}g^{\eq/\tw}\bm{e}.
\end{align*}
By the same way with suitable changes of $+$ and $-$, 
we also obtain the 
family of paths $\widetilde{\gamma}^\phi_{e',-}$ and the
section
\begin{align*}
\mathfrak{s}^\phi_{e',-}\coloneqq \widetilde{\gamma}^\phi_{e',-}\otimes e^{-f/\tw}g^{\eq/\tw}\bm{e}
\end{align*}
for $e'\in E_\infty$. 

\begin{lemma}
The sections $\mathfrak{s}^{\phi}_{e,+}$ $(e\in E_0)$
and $\mathfrak{s}^\phi_{e',-}$ $(e'\in E_\infty)$
satisfy
the conditions $(1)$ and $(3)$ in Theorem \ref{WTS}. 
\end{lemma}
\begin{proof}
The condition $(3)$
is trivial by the construction. 
We shall check the condition $(1)$ in the case $e\in E_0$ 
(the case $e'\in E_\infty$ can also be treated in the same way).
Let $[\omega]$
be a section of $\j^*\H^1_{\dR,\bm{0}}$
on a neighborhood of $(0,0)\in S$
represented by $\omega\in \pi_{\X*}\M\otimes \Omega^{1}_{\X/S}$. 
Set 
\[\widetilde{V}^\phi_+\coloneqq 
\(\{\tw\in \C\mid \Re(e^{-\i\phi}\tw)>0\}
\times\cal{S}^\phi_+(\delta)\) \subset{S}^*. \]
Then, we have $\j^*\H^1_{\dR,\bm{0}}(\widetilde{V}^\phi_+)
=\j^*\H^1_{\dR, E_0!}(\widetilde{V}^\phi_+)=\j^*(\lim\H^{1}_{\dR, a ,b})$, 
and hence 
$[\omega]$ is also represented by
$\omega_a\in\j^*\pi_*(\M_{a, b}(*\cal{P})\otimes \Omega^1_{\X/S})$ where  
$a$ is arbitrary small.  
Taking $a$ sufficiently small, the period pairing can be computed as follows:
\begin{align*}
\<\mathfrak{s}^\phi_{e,+}, [\omega]\>_\Per
&=\int_{\widetilde{\gamma}^\phi_{e,+}}e^{-f/\tw}g^{\eq/\tw}\omega_a\\
&=\int_{\gamma^\phi_{e,+}}e^{-f/\tw}g^{\eq/\tw}\omega_a&(\text{limit Stokes theorem})\\
&=\int_{\gamma^\phi_{e,+}}e^{-f/\tw}g^{\eq/\tw}\omega&(\text{limit Stokes theorem})
\end{align*}
Then,  we can use the same discussion as in Lemma \ref{c-case} to see (1).
\end{proof}

\begin{proof}[Proof of Theorem \ref{WTS}]
It remains to see that 
the sections constructed above satisfies 
the condition (2) in Theorem \ref{WTS}. 
This follows
from the following observations: 
\begin{itemize}
\item The paths $\gamma_{c,\eq}^{\phi}$ and $\gamma^\phi_{e,\eq,\pm} $
do not intersect with 
$\gamma_{c',\eq}^{\phi+\pi}$ or $\gamma_{e',\eq,\pm}^{\phi+\pi} $ if $c\neq c'$ 
and $e\neq e'$, where $c,c'\in \Crit(f)$ and $e, e'\in E$ (condition $(\gamma 2)$).
\item
The paths $\gamma_{c,\eq}^\phi$ and $\gamma_{c,\eq}^{\phi+\pi}$
(resp. $\gamma_{e,\eq,\pm}^\phi$ and $\gamma^{\phi+\pi}_{e,\eq,\mp}$)
intersect transversally only at $\nu_c(\eq)$ (resp. $\nu_e(\eq)$), 
at which we can take the branch of $\log g$ so that 
$\<\mathfrak{s}_c^\phi,\mathfrak{s}_c^{\phi+\pi}\>=1$
(resp. $\<\mathfrak{s}_{e\pm}^\phi, \mathfrak{s}^{\phi+\pi}_{e,\mp}\>=1$).
\end{itemize}
(see \cite{MMT} and \cite{FSY}.) 
\end{proof}

\subsection{Examples}\label{last examples}
In this last section,
we shall give two examples of 
$\L^\Be(f,g)$. 
The first one is related to the gamma function,
and the second one is related to cylindrical functions. 
\subsubsection{Gamma function}
We consider the case 
treated in Examples \ref{EXG}, \ref{Ga}, \ref{GDual}, and \ref{spec gamma}.
In this case, $\pi_{\Sigma}^{-1}(0)=E_0=\{0\}$. 
Any $\phi\in \R$ is generic. 

For $(\tw,\eq)\in \widetilde{V}^\phi_+$, 
we have
\begin{align*}
\<\widetilde{\mathfrak{s}}^\phi_{0,+},\j^*e_{-1}\>_\Per
&=\int_{\gamma^{\phi}_{0,+,\eq}}e^{-z/\tw}z^{\eq/\tw}\tw^{-1}\frac{dz}{z}\\
&=\tw^{-1}\int_{\gamma^{\phi}_{0,+,\eq}}e^{-\zeta}(\zeta\tw)^{\eq/\tw}\frac{d\zeta}{\zeta}\\
&=\tw^{\eq/\tw-1}\Gamma(\tw/\eq).
\end{align*}
By the computation in \ref{GDual},
we obtain
\[\widetilde{\mathfrak{s}}^\phi_{0,+}=(2\pi\i)^{-1}\tw^{\eq/\tw}\Gamma(\eq/\tw)e_{0}\]
in $\H^1_{\dR,\bm{0}|S^*}$. 
Note that 
we can directly check the 
relations $[\nabla_\euler,\widetilde{\mathfrak{s}}^\phi_{0,+}]=0$ and 
$\shift(\widetilde{\mathfrak{s}}^\phi_{0,+})=\widetilde{\mathfrak{s}}^\phi_{0,+}$
in this case. 
In a similar way, 
using the Hankel's counter integral representation of the gamma function,
we also obtain
\[\widetilde{\mathfrak{s}}^{\phi}_{0,-}=(2\pi\i)^{-1}(1-q)\tw^{\eq/\tw}\Gamma(\eq/\tw)e_{0}\]
on $V^\phi_-$. 
Hence we obtain the following:
\begin{proposition}
We have the natural isomorphism
\begin{align*}
(\L^\Be(f,g),\L^\Be_{\leqslant}(f,g))\simeq (\L_\Gamma,\L_{\Gamma\leqslant})
\end{align*}
where $f, g$ is as in Example \ref{EXG}
and $\L_\Gamma$ is as in Example \ref{GAMMA FUN}.\qed
\end{proposition}
\subsubsection{Cylindrical functions}
We consider the case
treated in Examples \ref{Bessel}, \ref{BDual}, and \ref{spec bessel}. 
In this case,
we have 
$\pi_{\Sigma}^{-1}(0)=\Crit(f)=\{\pm 1\}$. 
We put $c_1=1$, $c_2=-1$. 
$\phi\in\R$ is generic if and only if 
$e^{\i\phi}\notin\R$. 
Take $\phi$ so that $\Im (e^{\i\phi})>0$. 
Let $H^{(1)}_\nu(z)$ and $H^{(2)}_\nu(z)$ $(\nu, z \in \C,i=1,2)$ denote the first and second
Hankel function. 
We then put 
\[C^{(i)}(\tw,\eq)\coloneqq \pi\i e^{\pi\i \eq /2\tw}H^{(i)}_{\eq/\tw}(2\i \tw^{-1})\]
for $(\tw,\eq)\in S^*$ and $i=1,2$. 
Then we have
\begin{align*}
\<\mathfrak{s}_{c_i}^\phi, \j^*e_{j}\>
&=\tw^{-1}\int_{\gamma^\phi_{c}}e^{-({z+z^{-1}})/\tw}z^{\eq/\tw-j-1}\frac{dz}{z}\\
&=\tw^{-1}e^{\pi\i(\eq/\tw-j-1)/2}\int_{\gamma^\phi_c}e^{-(\zeta-\zeta^{-1})/\tw}\zeta^{\eq/\tw-j-1}\frac{d\zeta}{\zeta}\\
&=\tw^{-1}C^{(i)}(\tw, \eq-(j+1)\tw)
\end{align*}
for $i=1,2$ and $j=0,-1$.
By the computation given in Example \ref{BDual}, 
we
obtain
\begin{align*}
2\pi\i\mathfrak{s}_{c_i}^\phi= C^{(i)}(\tw,\eq)e_0+C^{(i)}(\tw,\eq-\tw) e_{-1}
\quad (i=1,2).
\end{align*}
We can directly check that $\nabla_\euler(\mathfrak{s}_{c_i}^\phi)=0$ and 
$\shift(\mathfrak{s}_{c_i}^\phi)=\mathfrak{s}_{c_i}^\phi$. 

Note that in the coordinate $(u, v)$, 
we have 
\[C^{(i)}(uv, v)= \pi\i e^{\pi\i /2u}H^{(i)}_{u^{-1}}(2\i u^{-1}v^{-1})\]
Hence Theorem \ref{main theorem} in this case
seems to be closely related to 
the asymptotic analysis of Olver \cite{Olver} on Bessel functions. 

\subsection*{Acknowledgement}
The author is grateful to 
Fumihiko Sanda, Jeng-Daw Yu, Saiei-Jaeyeong Matsubara-Heo, 
Keiji Matsumoto, 
Mikhail Kapranov, and Tatsuki Kuwagaki for valuable discussions. 
The discussion with Saiei-Jaeyeong Matsubara-Heo 
gave the author an idea on the way of
taking basis in 
the proof of Theorem \ref{WTS}.  
Keiji Matsumoto told the author the paper \cite{MMT}. 

He would like to express his gratitude to Takuro Mochizuki 
and Kyoji Saito  
for their kindness and encouragement. 
Takuro Mochizuki also kindly pointed out some mistakes in an 
early draft.

He also would like to thank his son, Hinata,
who joined his family when he 
was writing the first draft of this paper, 
for giving him happiness and pleasure. 

This work was supported by JSPS KAKENHI Grant Number JP18H05829, 
19K21021, 20K14280, and partially by 18H01116.
This work was also supported by World Premier International Research Center Initiative (WPI), MEXT, Japan.
%
%
%
%
%
%
%
%

\end{document}